\documentclass[reqno]{amsart}

\usepackage[english]{babel}

\usepackage{amssymb,amsmath,amsfonts,amsthm,mathrsfs}
\usepackage[colorlinks=true, citecolor=blue, anchorcolor=red]{hyperref}

\usepackage{verbatim}
\usepackage{enumerate}
\usepackage{amssymb}
\usepackage{graphicx}
\usepackage{calc}
\usepackage{latexsym}

\usepackage{amstext}

\usepackage{bm}

\usepackage{txfonts}

\newcommand{\mc}[1]{\mathcal #1}
\newcommand{\supp}{\operatorname{supp}}
\newcommand{\R}{\mathbb R}
\newcommand{\N}{\mathbb N}
\newcommand{\C}{\mathbb C}
\newcommand{\Z}{\mathbb Z}
\newcommand{\E}{\mathbb E}

\renewcommand{\Re}{\mathop{\text{\upshape{Re}}}}

\newcommand{\ms}[1]{\mathscr{#1}}

\renewcommand{\epsilon}{\varepsilon}
\renewcommand{\bar}[1]{\overline{#1}}
\newcommand{\id}{\mathop{\textrm{id}}\nolimits}

\renewcommand{\tilde}{\widetilde}

\newtheorem{theorem}{Theorem}[section]
\newtheorem{lemma}[theorem]{Lemma}
\newtheorem{corollary}[theorem]{Corollary}
\newtheorem{proposition}[theorem]{Proposition}
\theoremstyle{definition}
\newtheorem{definition}[theorem]{Definition}
\newtheorem{remark}[theorem]{Remark}
\newtheorem{assumption}[theorem]{Assumption}

\numberwithin{equation}{section}

\begin{document}

\title[]{Boundary Value Problems of Elliptic and Parabolic Type with Boundary Data of Negative Regularity}
\author{Felix Hummel}
\address{Technical University of Munich, Faculty of Mathematics, Boltzmannstra{\ss}e 3, 85748 Garching bei M\"unchen, Germany}
\email{hummel@ma.tum.de}
\thanks{}
\date{\today}

\begin{abstract}
We study elliptic and parabolic boundary value problems in spaces of mixed scales with mixed smoothness on the half space. The aim is to solve boundary value problems with boundary data of negative regularity and to describe the singularities of solutions at the boundary. To this end, we derive mapping properties of Poisson operators in mixed scales with mixed smoothness. We also derive $\mathcal{R}$-sectoriality results for homogeneous boundary data in the case that the smoothness in normal direction is not too large.
\end{abstract}

\maketitle

\section{Introduction}
In recent years, there were some efforts to generalize classical results on the bounded $\mathcal{H}^\infty$-calculus (\cite{Denk_et_al_2004,DHP03,Dore_Venni_2002,Duong_1990}) and maximal regularity (\cite{DHP03,DHP07,Denk_Pruess_Zacher_2008,Dore_Venni_1987,Hieber_Pruess_1997}) of elliptic and parabolic equations to cases in which rougher boundary data can be considered. The main tool in order to derive these generalizations are spatial weights, especially power weights of the form
\[
 w_r^{\partial\mathcal{O}}(x):=\operatorname{dist}(x,\partial\mathcal{O})^r \quad(x\in\mathcal{O}),
\]
which measure the distance to the boundary of the domain $\mathcal{O}\subset\R^n$. Including weights which fall outside the $A_p$-range, i.e. weights with $r\notin(-1,p-1)$, provides a huge flexibility concerning the smoothness of the boundary data which can be considered. We refer the reader to \cite{Lindemulder_Veraar_2018} in which the bounded $\mathcal{H}^\infty$-calculus for the shifted Dirichlet Laplacian in $L_p(\mathcal{O},w_r^{\partial\mathcal{O}})$ with $r\in(-1,2p-1)\setminus\{p-1\}$ has been obtained and applications to equations with rough boundary data are given. One even obtains more flexibility if one studies boundary value problems in weighted Besov and Triebel-Lizorkin spaces. Maximal regularity results for the heat equation with inhomogeneous boundary data have been obtained in \cite{Lindemulder_2018}. In \cite{Hummel_Lindemulder_2019} similar results were shown for general elliptic and parabolic boundary value problems.\\
The elliptic and parabolic equations we are interested in are of the form
\begin{align}\label{EllipticBVP}
  \begin{aligned}
 \lambda u -A(D)u&=f\quad\;\text{in }\R^n_+,\\
 B_j(D)u&=g_j\quad\text{on }\R^{n-1}\quad(j=1,\ldots,m),\\
 \end{aligned}
\end{align}
and
\begin{align}\label{ParabolicBVP}
  \begin{aligned}
 \partial_t u -A(D)u&=f\quad\;\text{in }\R\times\R^n_+,\\
 B_j(D)u&=g_j\quad\text{on }\R\times \R^{n-1}\quad(j=1,\ldots,m),
 \end{aligned}
\end{align}
where $A,B_1,\ldots,B_m$ is a homogeneous constant-coefficient parameter-elliptic boundary system, $f$ is a given inhomogeneity and the $g_j$ $(j=1,\ldots,m)$ are given boundary data. Of course, $f$ and the $g_j$ in \eqref{ParabolicBVP} may depend on time. We will also study the case in which \eqref{ParabolicBVP} is supplemented by initial conditions, i.e.
\begin{align}\label{ParabolicIBVP}
  \begin{aligned}
 \partial_t u -A(D)u&=f\quad\;\text{in }(0,T]\times\R^n_+,\\
 B_j(D)u&=g_j\quad\text{on }(0,T]\times \R^{n-1}\quad(j=1,\ldots,m),\\
 u(0,\,\cdot\,)&=u_0
 \end{aligned}
\end{align}
for some $T\in\R_+$. The focus will lie on the systematic treatment of boundary conditions $g_j$ which are only assumed to be tempered distributions. In particular, boundary data of negative regularity will be included. However, we still have some restrictions on the smoothness in time of the boundary data for \eqref{ParabolicIBVP}.\\
One reason for the interest in the treatment of rougher boundary data is that they naturally appear in problems with boundary noise. The fact that white noise terms have negative pathwise regularity (see for example \cite{Aziznejad_Fageot_Unser_2018,Fageot_Fallah_Unser_2017,Veraar_2011}) was one of the main motivations for this work. It was already observed in \cite{DaPrato_Zabczyk_1993} that even in one dimension solutions to equations with Gaussian boundary noise only have negative regularity in an unweighted setting. By introducing weights, this issue was resolved for example in \cite{Alos_Bonaccorsi_2002}. We also refer to \cite{Brzezniak_et_al_2015} in which the singularities at the boundary of solutions of Poisson and heat equation with different kinds of noise are analyzed.\\
One drawback of the methods in \cite{Alos_Bonaccorsi_2002,Brzezniak_et_al_2015,DaPrato_Zabczyk_1993} is that solutions are constructed in a space which is too large for traces to exist, i.e. the operators
\[
 \operatorname{tr}_{\partial\mathcal{O}}B_j(D)\quad(j=1,\ldots,m)
\]
are not well-defined as operators from the space in which the solution is constructed to the space of boundary data. This problem is avoided by using a mild solution concept, which is a valid approach in the classical setting and therefore, it seems reasonable to accept mild solutions as good enough, even though $\operatorname{tr}_{\partial\mathcal{O}}B_j(D)u$ does not make sense on its own.\\
In this paper, we propose a point of view which helps us to give a meaning to $\operatorname{tr}_{\partial\mathcal{O}}B_j(D)u$ in a classical sense. We will exploit that solutions to \eqref{EllipticBVP}, \eqref{ParabolicBVP} and \eqref{ParabolicIBVP} are very smooth in normal directions so that taking traces will easily be possible, even if the boundary data is just given by tempered distributions. This can be seen by studying these equations in spaces of the form $\mathscr{B}^k(\R_{+},\mathscr{A}^s(\R^{n-1}))$, where $\mathscr{A}$ and $\mathscr{B}$ denote certain scales of function spaces with smoothness parameters $s$ and $k$, respectively. The parameter $k$ corresponds to the smoothness in normal direction and will be taken large enough so that we can take traces and the parameter $s$ corresponds to smoothness in tangential directions and will be taken small enough so that $\mathscr{A}^s$ contains the desired boundary data. This way, we will not only be able to give a meaning to $\operatorname{tr}_{\partial\mathcal{O}}B_j(D)u$, but we will also obtain tools which help us to analyze the singularities of solutions at the boundary. This supplements the quantitative analysis in \cite{Hummel_Lindemulder_2019,Lindemulder_2018,Lindemulder_Veraar_2018}.\\
The idea to use spaces with mixed smoothness is quite essential in this paper, even if one refrains from using mixed scales. We refer to \cite[Chapter 2]{Schmeisser_Triebel_1987} for an introduction to spaces with dominating mixed smoothness. It seems like these spaces have not been used in the theory of partial differential equations so far. Nonetheless, we should mention that they are frequently studied in the theory of function spaces and have various applications. In particular, they are a classical tool in approximation theory in a certain parameter range, see for example \cite[Chapter 11]{Trigub_Bellinsky_2004}.\\ \\
This paper is structured in the following way:
\begin{itemize}
 \item Section \ref{Section:Preliminaries} briefly introduces the tools and concepts we use throughout the paper. This includes some notions and results from the geometry of Banach spaces, $\mathcal{R}$-boundedness and weighted function spaces.
 \item In Section \ref{Section:Pseudos} we study pseudo-differential operators in mixed scales with mixed smoothness. This will be important for the treatment of Poisson operators, as we will view them as functions in normal direction with values in the space of pseudo-differential operators of certain order in tangential directions.
 \item Section \ref{Section:Poisson} is the central part of this paper and the basis for the results in the subsequent sections. Therein, we derive various mapping properties of Poisson operators with values in spaces of mixed scales and mixed smoothness.
 \item In Section \ref{Section:Resolvent} we study equation \eqref{EllipticBVP} in spaces of mixed scales and mixed smoothness with homogeneous boundary data, i.e. with $g_j=0$. We derive $\mathcal{R}$-sectoriality of the corresponding operator under the assumption that the smoothness in normal direction is not too high. As a consequence, we also obtain maximal regularity for \eqref{ParabolicIBVP} with $g_j=0$ in the UMD case.
 \item Finally, we apply our techniques to the equations \eqref{EllipticBVP}, \eqref{ParabolicBVP} and \eqref{ParabolicIBVP} in Section \ref{Section:BVP}. We will be able to treat \eqref{EllipticBVP} and \eqref{ParabolicBVP} for arbitrary regularity in space and time. However, for the initial boundary value problem \eqref{ParabolicIBVP} we still have some restrictions concerning the regularity in time of the boundary data.
\end{itemize}

\subsection*{Comments on Localizations}
We should emphasize that we do not address questions of localization or perturbation in this work. Thus, we do not yet study what kind of variable coefficients or lower order perturbations of the operators we can allow. We also do not yet study how our results can be transferred to more general geometries than just the half space. Nonetheless, we want to give some ideas on how one can proceed.\\
The usual approach to transfer results for boundary value problems from the model problem on the half-space with constant coefficients to more general domains with compact smooth boundary and operators with non-constant coefficients is quite technical but standard. One takes a cover of the boundary which is fine enough such that in each chart the equation almost looks like the model problem with just a small perturbation on the coefficients and the geometry. In order to formally treat the local problem in a chart as the model problem, one has to find suitable extensions of the coefficients and the geometry such that the parameter-ellipticity is preserved and such that the coefficients are constant up to a small perturbation. One also carries out similar steps in the interior of the domain, where the equation locally looks like an elliptic or parabolic equation in $\R^n$. The essential step is then to derive perturbation results which justify that these small perturbations do not affect the property which one wants to transfer to the more general situation. Such a localization procedure has been carried out in full detail in \cite{Meyries_PHD-thesis}. We also refer the reader to \cite[Chapter 8]{DHP03}.\\
The localization approach also seems to be reasonable in our situation. However, there is an additional difficulty: As described above we work in spaces of the form $\mathscr{B}^k(\R_{+},\mathscr{A}^s(\R^{n-1}))$ which splits the half-space in tangential and normal directions. Since this splitting uses the geometric structure of the half-space, one might wonder what the right generalization of these spaces to a smooth bounded domain would be. In order to answer this question, we think that the notion of a collar of a manifold with boundary should be useful. More precisely, due to Milnor's collar neighborhood theorem (see for example \cite[Corollary 3.5]{Milnor_1965}) there exists an open neighborhood $U$ of the boundary $\partial M$ of a smooth manifold $M$ which is diffeomorphic to $\partial M\times [0,1)$. This neighborhood $U$ is a so-called collar neighborhood. On $\partial M\times [0,1)$ it is straightforward how to generalize our spaces with mixed smoothness: One could just take $\mathscr{B}^k([0,1),\mathscr{A}^s(\partial M))$. Now one could define the space 
\[
	\mathscr{B}^k(\mathscr{A}^s(U)):=\big\{u\colon U\to \C \;\vert\; u\circ\Phi^{-1}\in \mathscr{B}^k([0,1),\mathscr{A}^s(\partial M))\big\},
\]
where $\Phi\colon U\to \partial M\times [0,1)$ denotes the diffeomorphism provided by the collar neighborhood theorem. It seems natural to endow $\mathscr{B}^k(\mathscr{A}^s(U))$ with the norm
\[
	\|u\|_{\mathscr{B}^k(\mathscr{A}^s(U))}:=\| u\circ\Phi^{-1} \|_{\mathscr{B}^k([0,1),\mathscr{A}^s(\partial M))}.
\]
This definition allows us to give a meaning to the splitting in normal and tangential directions for a neighborhood of the boundary of general domains. It is less clear how to extend this splitting to the interior of the domain. But fortunately, this is also not important in our analysis. Indeed, the solution operators for inhomogeneous boundary data have a strong smoothing effect so that solutions are arbitrarily smooth in the interior of the domain with continuous dependence on the boundary data. Therefore, in the case of smooth coefficients one may work with smooth functions in the interior. To this end, we can take another open set $V\subset \operatorname{int} M$, where $\operatorname{int} M$ denotes the interior of $M$, such that $M=U \cup V$ and $\Phi(U\cap V)\subset\partial M\times[\tfrac{1}{2},1)$.  Moreover, we take $\phi,\psi\colon C^{\infty}(M)$ such that $\supp \phi\subset U$, $\supp \psi\subset V$ and $\phi+\psi\equiv 1$. Finally, we think that on bounded domains the right spaces should be given by
\[
	\{u\colon M\to \C \;\vert\; \phi\cdot u\in \mathscr{B}^k(\mathscr{A}^s(U)),\; \psi\cdot u\in  C^{\infty}(V) \},
\]
since they preserve the splitting in normal and tangential directions close to the boundary and since they use the smoothing effect of the solution operators where the splitting can not be preserved anymore.

\subsection*{Comparison To Other Works}
There are several other works which study boundary value problems with rough boundary data. The approach which seems to be able to treat the most boundary data is the one by Lions and Magenes, see \cite{Lions_Magenes_1972_a, Lions_Magenes_1972_b, Lions_Magenes_1973}. It also allows for arbitrary regularity at the boundary. One of the main points is that the trace operator is extended to the space $D_A^{-r}$ (see \cite[Theorem 6.5]{Lions_Magenes_1972_a}) where it maps into a boundary space with negative regularity. The space $D_A^{-r}$ contains those distributions $u\in H^{-r}(\Omega)$ such that $Au$ is in the dual of the space 
\[
\Xi^{r+2m}:=\big\{u\;\vert\;\operatorname{dist}(\,\cdot\,,\partial\Omega)^{|\alpha|}\partial^{\alpha}u\in L_2(\Omega),\;|\alpha|\leq r+2m\bigg\},
\]
i.e. $\Xi^{r+2m}$ contains functions whose derivatives may have singularities at the boundary with a certain order. This extension of the trace operator was generalized to the scale of H\"ormander spaces in \cite{Anop_Denk_Murach_2020}, where the authors used suitable interpolation techniques. One advantage of our approach compared to \cite{Anop_Denk_Murach_2020,Lions_Magenes_1972_a, Lions_Magenes_1972_b, Lions_Magenes_1973} is that we can give a detailed quantitative analysis of the smoothness and singularities of solutions at the boundary. For example, we show in Theorem~\ref{Thm:MainThm1} that solutions of \eqref{EllipticBVP} are arbitrarily smooth in normal direction if the smoothnes in tangential direction is chosen low enough. Moreover, we describe the singularities at the boundary if the smoothness in tangential direction is too high.\\
Another technique to treat rougher boundary data is to systematically study boundary value problems in weighted spaces. This has for example been carried out in \cite{Hummel_Lindemulder_2019,Lindemulder_2018,Lindemulder_Veraar_2018}. By using power weights in the $A_{\infty}$ range one can push the regularity on the boundary down to almost $0$. However, if one works with $A_{\infty}$ weights, then many Fourier analytic tools are not available anymore in the $L_p$ scale. The situation is better in Besov and Triebel-Lizorkin spaces, where Fourier multiplier techniques can still be applied. This has been used for second order operators with Dirichlet boundary conditions in \cite{Lindemulder_2018} and for general parameter-elliptic and parabolic boundary value problems in \cite{Hummel_Lindemulder_2019}. In both references, maximal regularity with inhomogeneous boundary data has been derived. \cite{Lindemulder_Veraar_2018} derives a bounded $\mathcal{H}^\infty$-calculus for the Dirichlet Laplacian in weighted $L_p$-spaces even for some weights which fall outside the $A_p$ range. The main tool therein to replace Fourier multiplier techniques are variants of Hardy's inequality. The results derived in \cite{Hummel_Lindemulder_2019,Lindemulder_2018, Lindemulder_Veraar_2018} are stronger than the ones we derive here in the sense that we do not derive maximal regularity with inhomogeneous boundary data or a bounded $\mathcal{H}^\infty$-calculus. As in our work, the singularities at the boundary are described by the strength of the weights one has to introduce. However, we can treat much more boundary data since \cite{Hummel_Lindemulder_2019,Lindemulder_2018, Lindemulder_Veraar_2018} are restricted to positive regularity on the boundary.\\
There are also works dealing with rough boundary data for nonlinear equations such as the Navier-Stokes equation, see for example \cite{Amann_2003,Grubb_2001} and references therein. The former reference uses the notion of very weak solutions as well as semigroup and interpolation-extrapolation methods. The latter reference studies the problem in the context of the Boutet de Monvel calculus. Our methods would still have to be extended to nonlinear problems. However, in both of the cited works there are restrictions on the regularity of the boundary data which can be considered.
\subsection*{Notations and Assumptions}
We write $\N=\{1,2,3,\ldots\}$ for the natural numbers starting from $1$ and $\N_0=\{0,1,2,\ldots\}$ for the natural numbers starting from $0$. Throughout the paper we take $n\in\N$ to be the space dimension and write
\[
 \R^n_+:=\{x=(x_1,\ldots,x_n)\in\R^n: x_n>0\}.
\]
If $n=1$ we also just write $\R_+:=\R^1_+$. Given a real number $x\in\R$, we write
\[
 x_+:=[x]_+:=\max\{0,x\}.
\]
We will frequently use the notation with the brackets for sums or differences of real numbers.
Oftentimes, we split $x=(x',x_n)\in\R^{n-1}\times\R$ or in the Fourier image $\xi=(\xi',\xi_n)$ where $x',\xi'\in\R^{n-1}$ refer to the directions tangential to the boundary $\R^{n-1}=\partial\R^{n}_+$ and $x_n,\xi_n\in\R$ refer to the normal directions. Given $x\in\C^n$ or a multi-index $\alpha\in\N_0^n$ we write
\[
 |x|:=\bigg(\sum_{j=1}^n|x_j|^2\bigg)^{1/2}\quad\text{or}\quad|\alpha|=\sum_{j=1}^n|\alpha_j|
\]
for the Euclidean length of $x$ or the $\ell^1$-norm of the multi-index $\alpha$, respectively. Even though this notation is ambiguous, it is convention in the literature and we therefore stick to it. We write
\[
 xy:=x\cdot y:=\sum_{j=1}^n x_j\overline{y}_j\quad(x,y\in\C^n)
\]
for the usual scalar product. The Bessel potential will be denoted by
\[
    \langle x\rangle:=(1+|x|^2)^{1/2}\quad(x\in\C^n).
\]
Given an angle $\phi\in(0,\pi]$ we write
\[
 \Sigma_{\phi}:=\{z\in\C:|\operatorname{arg} z|<\phi\}.
\]
If $M$ is a set then we use the notation 
\[
	\operatorname{pr}_j\colon M^n\to M,\; (a_1,\ldots,a_n)\to a_j\quad(j=1,\ldots,n)
\]
for the canonical projection of $M^n$ to the $j$-th component.\\
Throughout the paper $E$ will denote a complex Banach space on which we impose additional conditions at certain places. The topological dual of a Banach space $E_0$ will be denoted by $E_0'$. By $\mathscr{S}(\R^n;E)$ and $\mathscr{S}'(\R^n;E)$ we denote the spaces of $E$-valued Schwartz functions and $E$-valued tempered distributions, respectively. Given a domain $\mathcal{O}\subset\R^n$, we write $\mathscr{D}(\mathcal{O};E)$ and $\mathscr{D}'(\mathcal{O};E)$ for the spaces of $E$-valued test functions and $E$-valued distributions, respectively. If $E=\C$ in some function space, then we will omit it in the notation. On $\mathscr{S}(\R^n;E)$ we define the Fourier transform
\[
 (\mathscr{F}f)(\xi):=\frac{1}{(2\pi)^{n/2}}\int_{\R^n} e^{-ix\xi} f(x)\,dx\quad(f\in \mathscr{S}(\R^n;E)).
\]
As usual, we extend it to $\mathscr{S}'(\R^n;E)$ by $[\mathscr{F}u](f):=u(\mathscr{F}f)$ for $u\in \mathscr{S}'(\R^n;E)$ and $f\in \mathscr{S}(\R^n)$. Sometimes, we also use the Fourier transform $\mathscr{F}'$ which only acts on the tangential directions, i.e.
\[
 (\mathscr{F}'f)(\xi',x_n):=\frac{1}{(2\pi)^{(n-1)/2}}\int_{\R^n} e^{-ix'\xi'} f(x',x_n)\,dx'\quad(f\in \mathscr{S}(\R^n;E)).
\]
By $\sigma(T)$ and $\rho(T)$ we denote the spectrum and the resolvent set, respectively, of a linear operator $T\colon E\supset D(T)\to E$ defined on the domain $D(T)$. We write $\mathcal{B}(E_0,E_1)$ for the set of all bounded linear operators from the Banach space $E_0$ to the Banach space $E_1$ and we set $\mathcal{B}(E):=\mathcal{B}(E,E)$.\\
If $f,g\colon M\to \R$ map some parameter set $M$ to the reals, then we occasionally write $f\lesssim g$ if there is a constant $C>0$ such that $f(x)\leq C g(x)$ for all $x\in M$. If $f\lesssim g$ and $g\lesssim f$, we also write $f\eqsim g$. We mainly use this notation in longer computations.\\
Now we formulate our assumptions on the operators $A(D),B_1(D),\ldots,B_m(D)$: Let
\[
 A(D)=\sum_{|\alpha|=2m}a_{\alpha}D^{\alpha},\quad B_j(D)=\sum_{|\beta|=m_j} b^j_{\beta}D^{\beta}\quad(j=1,\ldots,m)
\]
for some $m,m_1,\ldots,m_m\in\N$ with $m_j<2m$ $(j=1,\ldots,m)$ and $a_{\alpha},b^j_{\beta}\in\mathcal{B}(E)$. 
\begin{assumption}[Ellipticity and Lopatinskii-Shapiro condition] \label{Assump:ELS} There is a $\phi'\in(0,\pi]$ such that
 \begin{enumerate}[(a)]
  \item $\rho(A(\xi))\subset\Sigma_{\phi'}$ for all $\xi\in\R^n\setminus\{0\}$.
  \item The equation
    \begin{align*}
     \lambda u(x_n)-A(\xi',D_n)u(x_n)&=0\quad\,\,(x_n>0),\\
     B_j(\xi',D_n)u(0)&=g_j\quad(j=1,\ldots,m)
    \end{align*}
has a unique continuous solution $u$ with $\lim_{x\to\infty} u(x)=0$ for all $(\lambda,\xi')\in\Sigma_{\phi'}\times\R^{n-1}$ and all $g=(g_1,\ldots,g_m)\in E^m$.
 \end{enumerate}
We take $\phi\in(0,\phi')$. If time-dependent equations are considered, we assume that $\phi>\pi/2$.
\end{assumption}
Assumption \ref{Assump:ELS} will be a global assumption which we assume to hold true without explicitly mentioning this every time. As we also consider mixed scales in this paper, there will be a lot of different choices of the precise spaces. Moreover, for the Bessel potential scale we will need different assumptions on the weights and the Banach space $E$ than for the Besov scale, Triebel-Lizorkin scale, or their dual scales. Thus, it will be convenient to introduce a notation which covers all these different cases. Some of the notation and notions in the following assumption will be introduced later in Section \ref{Section:Preliminaries}. For the moment, we just mention that $H$ denotes the Bessel potential scale, $B$ the Besov scale, $\mathcal{B}$ its dual scale, $F$ the Triebel-Lizorkin scale and $\mathcal{F}$ its dual scale.

\begin{assumption}\label{Assump:Spaces} Let $E$ be a Banach space, $s_0,s_1,s_2\in\R$, $p_0,p_1,p_2\in[1,\infty)$ and $q_0,q_1,q_2\in[1,\infty]$. Let further $w_0,w_1,w_2$ be weights and $I_{x_n},J_{t}\subset\R$ intervals. In the following $\bullet$ is a placeholder for any suitable choice of parameters. Moreover, by writing $J_{t}$, $I_{x_n}$ and $\R^{n-1}_{x'}$ we indicate with respect to which variable the spaces should be understood. Here, $t$ denotes the time, $x_n$ the normal direction and $x'$ the tangential directions.
\begin{enumerate}[(a)]
 \item\label{SpaceTangential} We take \begin{align*}
  \mathscr{A}^\bullet\in\{H^\bullet_{p_0}(\R^{n-1}_{x'},&w_0;E), B^\bullet_{p_0,q_0}(\R^{n-1}_{x'},w_0;E),F^\bullet_{p_0,q_0}(\R^{n-1}_{x'},w_0;E),\\
  &\mathcal{B}^\bullet_{p_0,q_0}(\R^{n-1}_{x'},w_0;E),\mathcal{F}^\bullet_{p_0,q_0}(\R^{n-1}_{x'},w_0;E)\}.
   \end{align*}
   If $\mathscr{A}^\bullet$ belongs to the Bessel potential scale, we assume that $p_0\in(1,\infty)$, that $E$ is a UMD space and that $w_0$ is an $A_p(\R^{n-1})$ weight. If $\mathscr{A}^\bullet$ belongs to the Besov or Triebel-Lizorkin scale, we assume that $w_0$ is an $A_{\infty}(\R^{n-1})$ weight. If $\mathscr{A}^\bullet$ belongs to the dual scale of Besov or Triebel-Lizorkin scale, we assume that $w_0$ is an $[A_{\infty}(\R^{n-1})]_p'$ weight, $p_0,q_0\in(1,\infty)$ and that $E$ is a UMD space.
   \item We take  \begin{align*}
  \mathscr{B}^{\bullet}(I_{x_n};\mathscr{A}^\bullet)\in\{H^\bullet_{p_1}(I_{x_n},&w_1;\mathscr{A}^\bullet), B^\bullet_{p_1,q_1}(I_{x_n},w_1;\mathscr{A}^\bullet),F^\bullet_{p_1,q_1}(I_{x_n},w_1;\mathscr{A}^\bullet),\\&\mathcal{B}^\bullet_{p_1,q_1}(I_{x_n},w_1;\mathscr{A}^\bullet),\mathcal{F}^\bullet_{p_1,q_1}(I_{x_n},w_1;\mathscr{A}^\bullet)\}.
 \end{align*}
   We impose conditions on $w_1,p_1,q_1$ and $E$ which are analogous to the ones for $w_0,p_0,q_0$ and $E$ in part \eqref{SpaceTangential}.
      \item We take  \begin{align*}
  \mathscr{C}^{\bullet}(J_t;\mathscr{A}^\bullet)\in\{H^\bullet_{p_2}(J_t,&w_2;\mathscr{A}^\bullet), B^\bullet_{p_2,q_2}(J_t,w_2;\mathscr{A}^\bullet),F^\bullet_{p_2,q_2}(J_t,w_2;\mathscr{A}^\bullet),\\&\mathcal{B}^\bullet_{p_2,q_2}(J_t,w_2;\mathscr{A}^\bullet),\mathcal{F}^\bullet_{p_2,q_2}(J_t,w_2;\mathscr{A}^\bullet)\}.
 \end{align*}
   We impose conditions on $w_2,p_2,q_2$ and $E$ which are analogous to the ones for $w_0,p_0,q_0$ and $E$ in part \eqref{SpaceTangential}.
   \item We take  \begin{align*}
  \mathscr{C}^{\bullet}(J_{t};\mathscr{B}^{\bullet}(I_{x_n};\mathscr{A}^\bullet))\in\{H^\bullet_{p_2}(J_{t},w_2&;\mathscr{B}^{\bullet}(I_{x_n};\mathscr{A}^\bullet)), B^\bullet_{p_2,q_2}(J_{t},w_2;\mathscr{B}^{\bullet}(I_{x_n};\mathscr{A}^\bullet)),\\&F^\bullet_{p_2,q_2}(J_{t},w_2;\mathscr{B}^{\bullet}(I_{x_n};\mathscr{A}^\bullet)),
    \}.
 \end{align*}
   We impose conditions on $w_2,p_2,q_2$ and $E$ which are analogous to the ones for $w_0,p_0,q_0$ and $E$ in part \eqref{SpaceTangential}.
\end{enumerate}
Most of the time, we just write $\mathscr{A}^s$, $\mathscr{B}^k(\mathscr{A}^s)$ and $\mathscr{C}^l(\mathscr{B}^k(\mathscr{A}^s))$ instead of $\mathscr{A}^s_{p_0,q_0}(\R^{n-1}_{x'},w_0;E)$, $\mathscr{B}^k_{p_1,q_1}(I_{x_n},w_1;\mathscr{A}^s(\R^{n-1}_{x'},w_0;E))$ and $\mathscr{C}^l_{p_2,q_2}(J_t,w_2;\mathscr{B}^k_{p_1,q_1}(I_{x_n},w_1;\mathscr{A}^s(\R^{n-1}_{x'},w_0;E))$. We mainly do this in order to keep notations shorter. Moreover, most of the time we only work with the smoothness parameter so that adding the other parameters to the notation would be distracting. However, at some places we will still add some of the other parameters if more clarity is needed.
\end{assumption}

Also Assumption \ref{Assump:Spaces} will be global and we use this notation throughout the paper.

\begin{remark}\phantomsection
    \begin{enumerate}[(a)]
        \item  Assumption \ref{Assump:Spaces} is formulated in a way such that we can always apply Mikhlin's theorem, Theorem \ref{Thm:Mikhlin}, and its iterated versions Proposition \ref{Prop:IteratedMikhlin} and Proposition \ref{Prop:IteratedMikhlinRbounded}. If $E$ has to satisfy Pisier's property $(\alpha)$ for some results, we will explicitly mention it.
        \item Note that every $f\in\mathscr{S}'(\R^{n-1})$ is contained in one of the spaces $\mathscr{A}^{\bullet}$ with certain parameters, see for example \cite[Proposition 1]{Kabanava_2008}.
    \end{enumerate}
\end{remark}

\section{Preliminaries}\label{Section:Preliminaries}

\subsection{Some Notions from the Geometry of Banach Spaces} If one wants to transfer theorems from a scalar-valued to a vector-valued situation, then this is oftentimes only possible if one imposes additional geometric assumptions on the Banach space. And since the iterated spaces we introduced in Assumption~\ref{Assump:Spaces} are vector-valued even if we take $E=\R$ or $E=\C$, it should not come as a surprise that we have to introduce some of these geometric notions. We refer the reader to \cite{HvNVW_2016,HvNVW_2017} for an extensive treatment of the notions in this subsection.
\subsubsection{UMD spaces}
The importance of Banach spaces with the property of unconditional martingale differences (UMD spaces) lies in the fact Mikhlin's Fourier multiplier theorem has only been generalized for operator-valued symbols if the underlying Banach spaces are UMD spaces. Therefore, if one wants to work with Fourier multipliers on vector-valued $L_p$-spaces, one is forced to impose this geometric condition.\\
A Banach space $E$ is called UMD space if for all $p\in(1,\infty)$ there is a constant $C>0$ such that for all probability spaces $(\Omega,\mathcal{F},\mathbb{P})$, all $N\in\N$, all $\epsilon_1,\ldots,\epsilon_N\in\C$ with $|\epsilon_1|=\ldots=|\epsilon_N|=1$, all filtrations $(\mathcal{F}_k)_{k=0}^N$ and all martingales $(f_k)_{k=0}^N$ in $L_p(\Omega;E)$ it holds that
\begin{align*}
	\bigg\|\sum_{k=1}^N \epsilon_k (f_k-f_{k-1})\bigg\|_{L_p(\Omega;E)}\leq C \bigg\|\sum_{k=1}^N  f_k-f_{k-1}\bigg\|_{L_p(\Omega;E)}.
\end{align*}
 This is equivalent to $E$ being a Banach space of class $\mathcal{HT}$, which is defined by the boundedness of the Hilbert transform on $L_p(\R;E)$. UMD spaces are always reflexive. Some important examples of UMD spaces are:
\begin{itemize}
 \item Hilbert spaces, in particular the scalar fields $\R,\C$,
 \item the space $L_p(S;E)$ for $p\in(1,\infty)$, a $\sigma$-finite measure space $(S,\mathcal{A},\mu)$ and a UMD space $E$,
 \item the classical function spaces such as Bessel potential spaces $H^{s}_p$, Besov spaces $B^{s}_{p,q}$ and Triebel-Lizorkin spaces $F^{s}_{p,q}$ in the reflexive range as well as their $E$-valued versions if $E$ is a UMD space.
\end{itemize}
\subsubsection{Cotype}
In this work, Banach spaces satisfying a finite cotype assumption could be considered as merely a technical notion that is needed to derive Proposition~\ref{Chapter3::Proposition::PoissonMappingPropertySobolevRBounded} which is a sharper version of Proposition~\ref{Chapter3::Proposition::PoissonMappingPropertySobolevRBoundedEpsilon}. The latter does not need a finite cotype assumption, while we show that it seems to be necessary to derive the former in Proposition~\ref{Prop:CounterExampleEpsilon}. The main reason why we need finite cotype assumptions is that they allow us to use a version of Kahane's contraction principle with function coefficients, see Proposition~\ref{Prop:HytoenenVeraar}.\\
Let $(\Omega,\mathcal{F},\mathbb{P})$ be a probability space. A sequence of random variables $(\epsilon_k)_{k\in\N}$ is called Rademacher sequence if it is an i.i.d. sequence with $\mathbb{P}(\epsilon_k=1)=\mathbb{P}(\epsilon_k=-1)=\frac{1}{2}$ for $k\in\N$. A Banach space $E$ is said to have cotype $q\in[2,\infty]$ if there is a constant $C>0$ such that for all choices of $N\in\N$ and $x_1,\ldots,x_N\in E$ the estimate
\[
 \bigg(\sum_{k=1}^N \|x_k\|^q\bigg)^{1/q}\leq C\bigg(\E\big\|\sum_{k=1}^N\epsilon_kx_k\big\|^q\bigg)^{1/q}
\]
holds with the usual modification for $q=\infty$. We want to remark the following
\begin{itemize}
 \item Every Banach space has cotype $\infty$.
 \item If a Banach space has cotype $q\in[2,\infty)$, then it also has cotype $\tilde{q}\in[q,\infty]$.
 \item No nontrivial Banach space can have cotype $q\in[1,2)$ since even the scalar fields $\R$ and $\C$ do not satisfy this.
 \item If the Banach space $E$ has cotype $q_E$, then $L_p(S;E)$ has cotype $\max\{p,q_E\}$ for every measure space $(S,\mathcal{A},\mu)$.
 \item If the Banach space $E$ has cotype $q_E$, then $H^{s}_p(\R^n;E)$ has cotype $\max\{p,q_E\}$ and $B^{s}_{p,q}(\R^n;E)$ and $F^{s}_{p,q}(\R^n;E)$ have cotype $\max\{p,q,q_E\}$. The same also holds for the weighted variants we introduce later.
\end{itemize}
\subsubsection{Pisier's property $(\alpha)$}
Finally, we also need Pisier's property $(\alpha)$ at some places in this paper. This condition is usually needed if one wants to derive $\mathcal{R}$-boundedness from Mikhlin's multiplier theorem. If one has a set of $\mathcal{R}$-bounded operator-valued symbols, then one needs Pisier's property $(\alpha)$ in order to obtain the $\mathcal{R}$-boundedness of the resulting operator family.\\
A Banach space $E$ has Pisier's property $(\alpha)$ if Kahane's contraction principle also holds for double random sums, i.e. if for two Rademacher sequences $(\epsilon'_i)_{i\in\N}$, $(\epsilon''_j)_{j\in\N}$ on the probability spaces $(\Omega',\mathcal{F}',\mathbb{P}')$ and $(\Omega'',\mathcal{F}'',\mathbb{P}'')$, respectively, there is a constant $C>0$ such that for all $M,N\in\N$, all $(a_{ij})_{1\leq i\leq M,1\leq j\leq N}\subset\C$ with $|a_{ij}|\leq1$ and all $(x_{ij})_{1\leq i\leq M,1\leq j\leq N}\subset E$ the estimate
\[
 \E_{\mathbb{P'}}\E_{\mathbb{P''}}\bigg\| \sum_{i=1}^M\sum_{j=1}^N a_{ij}\epsilon_{i}\epsilon_{j}x_{ij} \bigg\|^2\leq C^2 \E_{\mathbb{P'}}\E_{\mathbb{P''}}\bigg\| \sum_{i=1}^M\sum_{j=1}^N \epsilon_{i}\epsilon_{j}x_{ij} \bigg\|^2
\]
holds. Even though Pisier's property $(\alpha)$ is independent of the UMD property, the examples of spaces with Pisier's property $(\alpha)$ we have in mind are similar:
\begin{itemize}
 \item Hilbert spaces, in particular the scalar fields $\R,\C$,
 \item the space $L_p(S;E)$ for $p\in[1,\infty)$, a measure space $(S,\mathcal{A},\mu)$ and a Banach space $E$ with Pisier's property $(\alpha)$,
 \item the classical function spaces such as Bessel potential spaces $H^{s}_p$, Besov spaces $B^{s}_{p,q}$ and Triebel-Lizorkin spaces $F^{s}_{p,q}$ in the reflexive range as well as their  $E$-valued versions if $E$ has Pisier's property $(\alpha)$.
\end{itemize}

\subsection{\texorpdfstring{$\mathcal{R}$}--bounded Operator Families}
We refer the reader to \cite{DHP03,HvNVW_2017} for introductions to $\mathcal{R}$-bounded operator families. The notion of $\mathcal{R}$--boundedness is frequently needed if one works with vector-valued function spaces. As UMD spaces, it is essential for vector-valued generalizations of Mikhlin's multiplier theorem. But perhaps more importantly, it can be used the derive a necessary and sufficient condition for a closed linear operator $A\colon E\supset D(A)\to E$ on the UMD space $E$ to have the property of maximal regularity. This is the case if and only if it is $\mathcal{R}$--sectorial, i.e. if and only if the set
\[
	\{\lambda(\lambda-A)^{-1}:\lambda\in\C,\,\arg \lambda < \phi\}
\]
for some $\phi>\pi/2$ is $\mathcal{R}$--bounded, see \cite[Theorem 4.2]{Weis_2001}. Here, an operator $A$ is said to have the property of maximal regularity on $[0,T)$, $0<T<\infty$, if the mapping
\[
	W^1_p([0,T);X)\cap L_p([0,T);D(A))\to L_p([0,T);X)\times I_p(A),\;u\mapsto\begin{pmatrix}\partial_t u-Au \\ \gamma_0 u \end{pmatrix}
\]
is an isomorphism of Banach spaces, where $\gamma_0u:= u(0)$ denotes the temporal trace and $I_p(A)$ is the space of admissible initial conditions. It can be described as a real interpolation space of $X$ and $D(A)$ by the relation $I_p(A):=(X,D(A))_{1-1/p.p}$. The above isomorphy is very useful for the treatment of nonlinear parabolic equations, as it allows for the efficient use of fixed point iterations. This approach to nonlinear equations has already been applied many times in the literature.\\
Let us now define $\mathcal{R}$--boundedness: Let $E_0,E_1$ be Banach spaces. A family of operators $\mathcal{T}\subset\mathcal{B}(E_0,E_1)$ is called $\mathcal{R}$-bounded if there is a constant $C>0$ and $p\in[1,\infty)$ such that for a Rademacher sequence $(\epsilon_k)_{k\in\N}$ on a probability space $(\Omega,\mathcal{F},\mathbb{P})$ and all $N\in\N$, $x_1,\ldots,x_N\in E_0$ and $T_1,\ldots,T_N\in\mathcal{T}$ the estimate
\[
 \left\| \sum_{k=1}^N \epsilon_k T_k x_k \right\|_{L_p(\Omega;E_1)}\leq C\left\| \sum_{k=1}^N \epsilon_k  x_k \right\|_{L_p(\Omega;E_0)}
\]
holds. The least admissible constant such that this estimate holds will be denoted by $\mathcal{R}(\mathcal{T})$ or, if we want to emphasize the dependence on the Banach spaces, by $\mathcal{R}_{\mathcal{B}(E_0,E_1)}(\mathcal{T})$. By the Kahane-Khintchine inequalities, the notion of $\mathcal{R}$-boundedness does not depend on $p$. $\mathcal{R}$-boundedness trivially implies uniform boundedness, but the converse does not hold true in general. For Hilbert spaces however, both notions coincide.\\
An equivalent characterization of $\mathcal{R}$-boundedness can be given by using the $\operatorname{Rad}_p(E)$-spaces. They are defined as the space of all sequences $(x_k)_{k\in\N}\subset E$ such that $\sum_{k=1}^{\infty} \epsilon_k  x_k$ converges in $L_p(\Omega;E)$. $\operatorname{Rad}_p(E)$-spaces are endowed with the norm
\[
 \|(x_k)_{k\in\N}\|_{\operatorname{Rad}_p(E)}=\sup_{N\in\N}\left\| \sum_{k=1}^N \epsilon_k  x_k \right\|_{L_p(\Omega;E)}.
\]
Given  $T_1,\ldots,T_N\in\mathcal{B}(E_0,E_1)$ we define
\[
 \operatorname{diag}(T_1,\ldots,T_N)\colon \operatorname{Rad}_p(E_0)\to \operatorname{Rad}_p(E_1),\,(x_k)_{k\in\N}\to (T_kx_k)_{k\in\N}
\]
where $T_k:=0$ for $k>N$. Then, a family of operators $\mathcal{T}\subset\mathcal{B}(E_0,E_1)$ is $\mathcal{R}$-bounded if and only if 
\[
 \{\operatorname{diag}(T_1,\ldots,T_N): N\in\N, T_1,\ldots,T_N\in\mathcal{T}\}\subset \mathcal{B}(\operatorname{Rad}_p(E_0),\operatorname{Rad}_p(E_1))
\]
is uniformly bounded.\\
Let us now collect some results concerning $\mathcal{R}$-boundedness which will be useful in this paper.
\begin{proposition}\label{Prop:HytoenenVeraar}
 Let $E$ be a Banach space with cotype $q\in[2,\infty)$, $(\epsilon_k)_{k\in\N}$ a Rademacher sequence on the probability space $(\Omega,\mathcal{F},\mathbb{P})$ and let $(S,\mathcal{A},\mu)$ be a $\sigma$-finite measure space. For all $\tilde{q}\in(q,\infty]$ there exists a constant $C>0$ such that for all $N\in\N$, $f_1,\ldots,f_N\in L_{\tilde{q}}(S)$ and $x_1,\ldots,x_N\in E$ it holds that
 \begin{align}\label{Eq:Contraction_Principle_Functions}
  \left\|\sum_{k=1}^N \epsilon_k f_k x_k \right\|_{L_{\tilde{q}}(S;L_2(\Omega;E))}\leq C \sup_{1\leq k\leq N} \|f_k\|_{L_{\tilde{q}}(S)}\left\|\sum_{k=1}^N \epsilon_kx_k\right\|_{L_2(\Omega;E)}.
 \end{align}
If $q\in\{2,\infty\}$, then we can also take $\tilde{q}=q$.
\end{proposition}
\begin{proof}
 This is one of the statements in \cite[Lemma 3.1]{Hytoenen_Veraar_2009}.
\end{proof}

\begin{remark}\label{Rem:HytoenenVeraar}
\begin{enumerate}[(a)]
 \item It was already observed in \cite[Remark 3.3]{Hytoenen_Veraar_2009} that if $\tilde{q}<\infty$, then Proposition \ref{Prop:HytoenenVeraar} can also be formulated as follows: The image of the unit ball $B_{L^{\tilde{q}}(S)}(0,1)$ in $L^{\tilde{q}}(S)$ under the embedding $L^{\tilde{q}}(S)\hookrightarrow \mathcal{B}(E,L^{\tilde{q}}(S;E)),f\mapsto f\otimes(\,\cdot\,)$ is an $\mathcal{R}$-bounded subset of $\mathcal{B}(E,L^{\tilde{q}}(S;E))$.
 \item If $(S,\mathcal{A},\mu)$ is nonatomic and if \eqref{Eq:Contraction_Principle_Functions} holds for all $N\in\N$, $f_1,\ldots,f_N\in L_{\tilde{q}}(S)$ and $x_1,\ldots,x_N\in E$, then $E$ has cotype $\tilde{q}$. This is follows from the statements in \cite[Lemma 3.1]{Hytoenen_Veraar_2009}.
 \item Let $(A,\Sigma,\nu)$ be a $\sigma$-finite measure space. Let further $2\leq \bar{q}<q<\infty$ and let $\bar{E}$ be a Banach space with cotype $\bar{q}$. If $E=L_q(A;\bar{E})$, then \eqref{Eq:Contraction_Principle_Functions} also holds with $\tilde{q}=q$. This was shown in \cite[Remark 3.4]{Hytoenen_Veraar_2009}.
\end{enumerate}
\end{remark}

\begin{proposition}\label{Prop:Rbounded_Interpolation}
Let $(E_0,E_1)$ and $(F_0,F_1)$ be interpolation couples of UMD-spaces, $\Sigma\subset\C$ and $f\colon\Sigma\to\C$. Let further $(T(\lambda))_{\lambda\in\Sigma}\subset \mathcal{B}(E_0+E_1,F_0+F_1)$ be a collection of operators such that
\[
 \mathcal{R}_{\mathcal{B}(E_0,F_0)}(\{T(\lambda):\lambda\in\Sigma\})<M_0,\quad \mathcal{R}_{\mathcal{B}(E_1,F_1)}(\{f(\lambda)T(\lambda):\lambda\in\Sigma\})<M_1
\]
for some $M_0,M_1>0$. We write $E_{\theta}=[E_0,E_1]_{\theta}$ and $F_{\theta}=[F_0,F_1]_{\theta}$ with $\theta\in(0,1)$ for the complex interpolation spaces. Then there is a constant $C>0$ such that
\[
 \mathcal{R}_{\mathcal{B}(E_\theta,F_\theta)}(\{f(\lambda)^{\theta}T(\lambda):\lambda\in\Sigma\})<CM_0^{1-\theta}M_1^{\theta}.
\]
\end{proposition}
\begin{proof}
In order to avoid possible ambiguities with complex exponentials, we assume that $f$ takes values in $(0,\infty)$. As a consequence of Kahane's contraction principle (\cite[Theorem 6.1.13]{HvNVW_2017}) we may do this without loss of generality. It suffices to show that
 \[
  \{\operatorname{diag}(f(\lambda_1)^{\theta}T(\lambda_1),\ldots,f(\lambda_N)^{\theta}T(\lambda_N):N\in\N,\lambda_1,\ldots,\lambda_N\in\Sigma\}
 \]
 is a bounded family in $\mathcal{B}(\operatorname{Rad}_p(E_{\theta}),\operatorname{Rad}_p(F_{\theta}))$. Let
 \[
  S:=\{z\in\C: 0\leq\Re z\leq 1\}.
 \]
For fixed $N\in\N$ and  $\lambda_1,\ldots,\lambda_N\in\Sigma$ we define
\begin{align*}
 \mathscr{T}_{\lambda_1,\ldots,\lambda_N}\colon S \to \mathcal{B}(\operatorname{Rad}_p(E_{0})&\cap \operatorname{Rad}_p(E_{1}), \operatorname{Rad}_p(F_{0})+ \operatorname{Rad}_p(F_{1})),\\
 &\,z\mapsto \operatorname{diag}( f(\lambda_1)^z T(\lambda_1),\ldots,f(\lambda_N)^z T(\lambda_N)).
\end{align*}
For fixed $(x_k)_{k\in\N}\in \operatorname{Rad}_p(E_{0})\cap \operatorname{Rad}_p(E_{1})$ the mapping
\begin{align*}
 \mathscr{T}_{\lambda_1,\ldots,\lambda_N}(\,\cdot\,)(x_k)_{k\in\N}\colon S \to \operatorname{Rad}_p(F_{0})+ \operatorname{Rad}_p(F_{1}),
 &\,z\mapsto (f(\lambda_k)^zT(\lambda_k)x_k)_{k\in\N}.
\end{align*}
is continuous, bounded and analytic in the interior of $S$. Again, we used the convention $T(\lambda_k)=0$ for $k>N$. Moreover, by assumption we have that
\[
 \sup_{t\in\R}\|\mathscr{T}_{\lambda_1,\ldots,\lambda_N}(j+it)\|_{\mathcal{B}(\operatorname{Rad}_p(E_{j}),\operatorname{Rad}_p(F_{j}))}<M_j\quad(j\in\{0,1\}).
\]
Thus, it follows from abstract Stein interpolation (see \cite[Theorem 2.1]{Voigt_1992}) that
\[
 \|\mathscr{T}_{\lambda_1,\ldots,\lambda_N}(\theta)\|_{\mathcal{B}(\operatorname{Rad}_p^{\theta}(E_{0},E_1),\operatorname{Rad}_p^{\theta}(F_{0},F_1))}<M_0^{1-\theta}M_1^{\theta},
\]
where we used the shorter notation $\operatorname{Rad}_p^{\theta}(E_{0},E_1)=[\operatorname{Rad}_p(E_{0}),\operatorname{Rad}_p(E_{1})]_{\theta}$ in the subscript. But it was shown in \cite[Corollary 3.16]{Kaip_Saal_2012} that
 \[
  [\operatorname{Rad}_p(E_0),\operatorname{Rad}_p(E_1)]_{\theta}=\operatorname{Rad}_p(E_{\theta})
 \]
with equivalence of norms so that there is a constant $C>0$ such that
\[
 \|\mathscr{T}_{\lambda_1,\ldots,\lambda_N}(\theta)\|_{\mathcal{B}(\operatorname{Rad}_p(E_{\theta}),\operatorname{Rad}_p(F_{\theta}))}<CM_0^{1-\theta}M_1^{\theta}.
\]
Since $N\in\N$ and $\lambda_1,\ldots,\lambda_N\in\Sigma$ were arbitrary, we obtain the assertion.
\end{proof}

\begin{remark}
\begin{enumerate}[(a)]
 \item The proof of Proposition~\ref{Prop:Rbounded_Interpolation} was inspired by the proof of \cite[Lemma 6.9]{HHK_2006}. Note that \cite[Example 6.13]{HHK_2006} shows that Proposition~\ref{Prop:Rbounded_Interpolation}  does not hold true if the complex interpolation functor is replaced by the real one.
 \item In Proposition~\ref{Prop:Rbounded_Interpolation} we only use the UMD assumption in order to show that the interpolation space of two Rademacher spaces coincides with the Rademacher space of the interpolation space of the underlying Banach spaces. This holds more generally for K-convex Banach spaces (see \cite[Theorem 7.4.16]{HvNVW_2017}). We refrain from introducing K-convexity in order not to overload this paper with geometric notions. Note however that UMD spaces are K-convex, see \cite[Example 7.4.8]{HvNVW_2017}.
 \end{enumerate}
\end{remark}

\begin{lemma}\label{Lemma:HolomorphicSectorRBounded}
	Let $\psi\in(0,\pi)$ and let $E_0,E_1$ be Banach spaces. Let further $N\colon\overline{\Sigma}_{\psi}\to\mathcal{B}(E_0,E_1)$ be holomorphic and bounded on $\Sigma_{\psi}$ and suppose that $N|_{\partial\Sigma_\psi}$ has $\mathcal{R}$-bounded range. Then the set
	\[
		\{\lambda^k\big(\tfrac{d}{d\lambda}\big)^kN(\lambda):\lambda\in \overline{\Sigma}_{\psi'}\}
	\]
	is $\mathcal{R}$-bounded for all $\psi'<\psi$ and all $k\in\N_0$.
\end{lemma}
\begin{proof}
	For $k=0$ and $k=1$ the proof is contained in \cite[Example 2.16]{Kunstmann_Weis_2004}. Other values of $k$ can then be obtained by iteration. Note that the boundedness of $N$ is necessary since Poisson's formula, which is used for $k=0$, only holds for bounded functions.
\end{proof}

\begin{definition}
Let $(X,d)$ be a metric space and  $E_0,E_1$ be Banach spaces. Let further $U\subset \R^n$ be open and $k\in\N_0$.
\begin{enumerate}[(a)]
 \item We say that a function $f\colon X\to\mathcal{B}(E_0,E_1)$ is $\mathcal{R}$-continuous if for all $x\in X$ and all $\epsilon>0$ there is a $\delta>0$ such that we have
 \[
  \mathcal{R}_{\mathcal{B}(E_0,E_1)}(\{f(y)-f(x): y\in B(x,\delta)\})<\epsilon.
 \]
We write $C_{\mathcal{R}B}(X,\mathcal{B}(E_0,E_1))$ for the space of all $\mathcal{R}$-continuous functions $f\colon X\to\mathcal{B}(E_0,E_1)$ with $\mathcal{R}$-bounded range.
 \item We say that a function $f\colon X\to\mathcal{B}(E_0,E_1)$ is uniformly $\mathcal{R}$-continuous if for all $\epsilon>0$ there is a $\delta>0$ such that for all $x\in X$ we have
 \[
  \mathcal{R}_{\mathcal{B}(E_0,E_1)}(\{f(y)-f(x): y\in B(x,\delta)\})<\epsilon.
 \]
We write $BUC_{\mathcal{R}}(X,\mathcal{B}(E_0,E_1))$ for the space of all uniformly $\mathcal{R}$-continuous functions $f\colon X\to\mathcal{B}(E_0,E_1)$ with $\mathcal{R}$-bounded range.
\item We write $C_{\mathcal{R}B}^k(U,\mathcal{B}(E_0,E_1))$ for the space of all $f\in C^k(U,\mathcal{B}(E_0,E_1))$ such that $\partial^{\alpha}f\in C_{\mathcal{R}B}(U,\mathcal{B}(E_0,E_1))$ for all $\alpha\in\N_0^n$, $|\alpha|\leq k$. Analogously, we write $BUC_{\mathcal{R}}^k(U,\mathcal{B}(E_0,E_1))$ for the space of all $f\in C^k(U,\mathcal{B}(E_0,E_1))$ such that $\partial^{\alpha}f\in BUC_{\mathcal{R}}(U\mathcal{B}(E_0,E_1))$ for all $\alpha\in\N_0^n$, $|\alpha|\leq k$.
\end{enumerate}
\end{definition}

\begin{definition} Let $U\subset \R^n$ be open, $E_0,E_1$ Banach spaces and $k\in\N_0$. Let further $\mathcal{T}\subset C_{\mathcal{R}B}^k(X,\mathcal{B}(E_0,E_1))$ or $\mathcal{T}\subset BUC_{\mathcal{R}}^k(X,\mathcal{B}(E_0,E_1))$.
 \begin{enumerate}[(a)]
  \item We say that $\mathcal{T}$ is bounded if $$\sup_{f\in\mathcal{T}}\mathcal{R}(f^{(j)}(U):j\in\{0,\ldots,k\})<\infty.$$
\item We say that $\mathcal{T}$  is $\mathcal{R}$-bounded if 
\[
 \mathcal{R}(\{f^{(j)}(x):x\in U, f\in\mathcal{T}, j\in\{0,\ldots,k\}\})<\infty.
\]
 \end{enumerate}
\end{definition}

\subsection{Weighted Function Spaces}
Weights are an important tool to weaken the regularity assumptions on the data which are needed in order to derive well-posedness and a priori estimates for elliptic and parabolic boundary value problems, see for example \cite{Alos_Bonaccorsi_2002, Brzezniak_et_al_2015, Hummel_Lindemulder_2019, Lindemulder_2018, Lindemulder_Veraar_2018}. Power weights, i.e. weights of the form $w_\gamma(x):=\operatorname{dist}(x,\partial\mathcal{O})^\gamma$ which measure the distance to the boundary of the domain $\mathcal{O}\subset\R^n$, are particularly useful for this purpose. Roughly speaking, the larger the value of $\gamma$, the larger may the difference between regularity in the interior and regularity on the boundary be. This way, one can obtain arbitrary regularity in the interior while the regularity of the boundary data may be close to $0$. However, there is an important borderline: If $\gamma\in (-1,p-1)$, where $p$ denotes the integrability parameter of the underlying function space, then $w_\gamma$ is a so-called $A_p$ weight. If $\gamma\geq p-1$, then it is only an $A_{\infty }$ weight. $A_p$ weights are an important class of weights. These weights are exactly the weights $w$ for which the Hardy-Littlewood maximal operator is bounded on $L_p(\mathcal{O},w)$. Consequently, the whole Fourier analytic toolbox can still be used and many results can directly be transferred to the weighted setting. In the $A_{\infty}$-range however, this does not hold any longer. But in order to obtain more flexibility for the regularity of the boundary data which can be considered, one would like to go beyond the borderline and also work with $A_{\infty}$ weights. This is possible if one works with weighted Besov- or Triebel-Lizorkin spaces. As we will explain later, these scales of function spaces allow for a combination of $A_{\infty}$ weights and Fourier multiplier methods. In our analysis, we want to include both cases: We treat the more classical situation with the Bessel potential scale and $A_p$ weights, which include the classical Sobolev spaces, as well as the more flexible situation with Besov and Triebel-Lizorkin scales and $A_{\infty}$ weights.\\
Let us now give the precise definitions: Let $\mathcal{O}\subset\R^n$ be a domain. A weight $w$ on $\mathcal{O}$ is a function $w\colon\mathcal{O}\to[0,\infty]$ which takes values in $(0,\infty)$ almost everywhere with respect to the Lebesgue measure. We mainly work with the classes $A_p$ $(p\in(1,\infty])$. A weight $w$ on $\R^n$ is an element of $A_p$ for $p\in(1,\infty)$ if and only if
\[
 [w]_{A_p}:= \sup_{Q\text{ cube in }\R^n}\bigg(\frac{1}{\lambda(Q)}\int_Q w(x)\,dx\bigg)\bigg(\frac{1}{\lambda(Q)}\int_Q w(x)^{-\frac{1}{p-1}}\,dx\bigg)^{p-1}<\infty.
\]
The quantity $[w]_{A_p}$ is called $A_p$ Muckenhoupt characteristic constant of $w$. We define $A_{\infty}:=\bigcup_{1<p<\infty} A_p$. Moreover, we write $[A_{\infty}]_p'$ for the space of all weights $w$ such that the $p$-dual weight $w^{-\frac{1}{p-1}}$ is in  $A_{\infty}$. We refer to \cite[Chapter 9]{Grafakos_2009} for an introduction to these classes of weights.\\
For $p\in[1,\infty)$, a domain $\mathcal{O}\subset\R^n$, a weight $w$ and a Banach space $E$ the weighted Lebesgue-Bochner space $L_p(\mathcal{O},w;E)$ is defined as the space of all strongly measurable functions $f\colon\mathcal{O}\to E$ such that
\[
 \|f\|_{L_p(\mathcal{O},w;E)}:=\bigg(\int_{\mathcal{O}} \|f(x)\|_{E}^p w(x)\,dx\bigg)^{1/p}<\infty.
\]
We further set $L_{\infty}(\mathcal{O},w;E):=L_{\infty}(\mathcal{O};E)$. In addition, let $ L_1^{loc}(\mathcal{O};E)$ be the space of all locally integrable functions, i.e. strongly measurable functions $f\colon  \mathcal{O}\to E$ such that
\[
	\int_K \|f(x)\|_E\,dx<\infty
\]
 for all compact $K\subset \mathcal{O}$. As usual, functions which coincide on a set of measure $0$ are considered as equal in these spaces. \\
One has to be cautious with the definition of weighted Sobolev spaces. One would like to define them as spaces of distributions such that derivatives up to a certain order can be represented by functions in $L_p(\mathcal{O},w;E)$. But for some weights, the elements of $L_p(\mathcal{O},w;E)$ might not be locally integrable and thus, taking distributional derivatives might not be possible. H\"older's inequality shows that $L_p(\mathcal{O},w;E)\subset L_1^{loc}(\mathcal{O},E)$ if $w^{-\frac{1}{p-1}}\in L_1^{loc}(\mathcal{O})$. We refer to \cite{Kufner_Opic_1984} for further thoughts in this direction.
\begin{definition}
 Let $\mathcal{O}\subset\R^n$ be a domain, $E$ a Banach space, $m\in\N_0$, $p\in[1,\infty)$ and $w$ a weight on $\mathcal{O}$ such that $w^{-\frac{1}{p-1}}\in L_1^{loc}(\mathcal{O})$.
 \begin{enumerate}[(a)]
  \item We define the weighted Sobolev space $W^m_p(\mathcal{O},w;E)$ by
 \[
  W^m_p(\mathcal{O},w;E):=\{f\in L_p(\mathcal{O},w;E)\,|\,\forall\alpha\in\N_0^n,\,|\alpha|\leq m : \partial^\alpha f\in L_p(\mathcal{O},w;E)\}
 \]
and endow with the norm $\|f\|_{W^m_p(\mathcal{O},w;E)}:=\big(\sum_{|\alpha|\leq m} \|f\|_{L_p(\mathcal{O},w;E)}^p\big)^{1/p}$. With the usual modifications, we can also define $W^m_\infty(\mathcal{O},w;E)$.
\item As usual, we define $W^m_{p,0}(\mathcal{O},w;E)$ to be the closure of the space of test functions $\mathscr{D}(\mathcal{O};E)$ in $W^m_{p}(\mathcal{O},w;E)$.
\item Let $E$ be reflexive, $w\in A_p$ and $p,p'\in(1,\infty)$ conjugated H\"older indices, i.e. they satisfy $1=\frac{1}{p}+\frac{1}{p'}$. Then we define the dual scale $W^{-m}_p(\mathcal{O},w;E):=(W^m_{p',0}(\mathcal{O},w^{-\frac{1}{p-1}};E'))'$.
 \end{enumerate}
\end{definition}

We further define weighted Bessel potential, Besov and Triebel-Lizorkin spaces. Since we use the Fourier analytic approach, we already define them as subsets of tempered distributions.

\begin{definition}\label{def:bessel_potential_space}
 Let $E$ be a Banach space, $s\in\R$, $p\in[1,\infty]$ and $w$ a weight on $\R^n$ such that $w^{-\frac{1}{p-1}}\in L_1^{loc}(\R^n)$. Then we define the weighted Bessel potential space $H^s_p(\R^n,w;E)$ by
 \[
  H^s_p(\R^n,w;E):=\{f\in\mathscr{S}'(\R^n;E)\,|\,\langle D \rangle^s f\in L_p(\R^n,w;E)\}
 \]
and endow it with the norm $\|f\|_{H^s_p(\R^n,w;E)}:=\|\langle D\rangle^s f\|_{L_p(\R^n,w;E)}$.
\end{definition}

\begin{definition}
\begin{enumerate}[(a)]
   \item Let $\varphi_0\in\mathscr{D}(\R^n)$ be a smooth function with compact support such that $0\leq\varphi_0\leq1$ and
   \[
     \varphi_0(\xi)=1\quad\text{if } |\xi|\leq1,\qquad  \varphi_0(\xi)=0\quad\text{if }|\xi|\geq 3/2.
   \]
   For $\xi\in\R^n$ and $k\in\N$ let further
   \begin{align*}
    \varphi(\xi)&:=\varphi_0(\xi)-\varphi_0(2\xi),\\
    \varphi_k(\xi)&:=\varphi(2^{-k}\xi).
   \end{align*}
   We call such a sequence $(\varphi_k)_{k\in\N_0}$ smooth dyadic resolution of unity.
   \item Let $E$ be a Banach space and let $(\varphi_k)_{k\in\N_0}$ be a smooth dyadic resolution of unity. On the space of $E$-valued tempered distributions $\mathscr{S}'(\R^n;E)$ we define the sequence of operators $(S_k)_{k\in\N_0}$ by means of
   \[
    S_kf:=\mathscr{F}^{-1}\varphi_k\mathscr{F} f\quad(f\in\mathscr{S}'(\R^n;E)).
   \]
     The sequence $(S_k f)_{k\in\N_0}$ is called dyadic decomposition of $f$.
\end{enumerate}
\end{definition}
By construction, we have that $\mathscr{F}(S_k f)$ has compact support so that $S_k f$ is an analytic function by the Paley-Wiener theorem, see \cite[Theorem 2.3.21]{Grafakos_2008}. Moreover, it holds that $\sum_{k\in\N_0} \varphi_k=1$ so that we have $f=\sum_{k\in\N_0} S_kf$, i.e. $f$ is the limit of a sequence of analytic functions where the limit is taken in the space of tempered distributions. Elements of Besov- and Triebel-Lizorkin spaces even have convergence in a stronger sense, as their definition shows:

\begin{definition}\label{def:besov_triebel_lizorkin_space}
 Let $(\varphi_k)_{k\in\N_0}$ be a smooth dyadic resolution of unity. Let further $E$ be a Banach space, $w$ a weight, $s\in\R$, $p\in[1,\infty)$ and $q\in[1,\infty]$.
 \begin{enumerate}[(a)]
  \item We define the weighted Besov space $B^s_{p,q}(\R^n,w;E)$ by 
  \[
   B^s_{p,q}(\R^n,w;E):=\{f\in\mathscr{S}'(\R^n,E): \|f\|_{B^s_{p,q}(\R^n,w;E)}<\infty\}
  \] where
  \[
      \|f\|_{B^s_{p,q}(\R^n,w;E)}:=\|(2^{sk}\mathscr{F}^{-1}\varphi_k\mathscr{F}f)_{k\in\N_0}\|_{\ell^q(L_p(\R^n,w;E))}.
  \]
  \item We define the weighted Triebel-Lizorkin space $F^s_{p,q}(\R^n,w;E)$ by $$F^s_{p,q}(\R^n,w;E):=\{f\in\mathscr{S}'(\R^n;E): \|f\|_{F^s_{p,q}(\R^n,w;E)}<\infty\}$$ where
  \[
      \|f\|_{F^s_{p,q}(\R^n,w;E)}:=\|(2^{sk}\mathscr{F}^{-1}\varphi_k\mathscr{F}f)_{k\in\N_0}\|_{L_p(\R^n,w;\ell^q(E)))}.
  \]
 \end{enumerate}
\end{definition}
It is well-known, that these spaces do not depend on the choice of the dyadic resolution of unity if $w$ is an $A_{\infty}$-weight. In this case, different choices lead to equivalent norms, see for example \cite[Proposition 3.4]{Meyries_Veraar_2012}. In fact, the condition on the weight can be weakened: In \cite{Rychkov_2001} it was shown that one also obtains the independence of the dyadic resolution of unity in the case of so-called $A_{\infty}^{loc}$ weights. 

\begin{definition}\label{def:dual_besov_triebel_lizorkin_space}
 Let $E$ be a reflexive Banach space, $w\in [A_{\infty}]_p'$, $s\in\R$ and $p,q\in(1,\infty)$. We define the dual scales of Besov and Triebel-Lizorkin scale by
 \[
 	\mathcal{B}^{s}_{p,q}(\R^n,w;E):=(B^{-s}_{p',q'}(\R^n,w^{-\frac{1}{p-1}};E'))',\quad \mathcal{F}^{s}_{p,q}(\R^n,w;E):=(F^{-s}_{p',q'}(\R^n,w^{-\frac{1}{p-1}};E'))',
 \]
 where $p',q'$ denote the conjugated H\"older indices.
\end{definition}

\begin{remark}
The main reason for us to include the dual scales in our considerations is the following: If $w$ is additionally an admissible weight in the sense of \cite[Section 1.4.1.]{Schmeisser_Triebel_1987}, then we have $\mathcal{B}^{s}_{p,q}(\R^n,w)=B^{s}_{p,q}(\R^n,w)$ and $\mathcal{F}^{s}_{p,q}(\R^n,w)=F^{s}_{p,q}(\R^n,w)$. Therefore, we can also treat weighted Besov- and Triebel-Lizorkin spaces with weights that are outside the $A_{\infty}$ range. Formulating this in terms of dual scales allows us to transfer Fourier multiplier theorems without any additional effort just by duality. The main example we have in mind will be $w(x)=\langle x\rangle^{d}$ with arbitrary $d\in\R$. We will make use of this in a forthcoming paper on equations with boundary noise.
\end{remark}

\begin{proposition}\label{Prop:Contraction_Principle_Functions}
 Recall that Assumption \ref{Assump:Spaces} holds true and suppose that $E$ has cotype $q_E\in[2,\infty)$. Let further $(S,\Sigma,\mu)$ be a $\sigma$-finite measure space, $(\epsilon_k)_{k\in\N}$ a Rademacher sequence on the probability space $(\Omega,\mathcal{F},\mathbb{P})$, $s\in\R$ and $p_0,q_0\in(1,\infty)$. Consider one of the following cases:
 \begin{enumerate}[(a)]
  \item $\mathscr{A}^\bullet$ stands for the Bessel potential scale and $p\in(\max\{q_E,p_0\},\infty)$. Moreover, we allow $p=\max\{q_E,p_0\}$ if $q_E<p_0$ or if $E$ is a Hilbert space and $p_0=2$.
  \item $\mathscr{A}^\bullet$ stands for the Besov scale and $p\in(\max\{q_E,p_0,q_0\},\infty)$. Moreover, we allow $p=\max\{q_E,p_0,q_0\}$ if $q_E< p_0\leq q_0$ or if $E$ is a Hilbert space and $p_0=q_0=2$.
  \item $\mathscr{A}^\bullet$ stands for the Triebel-Lizorkin scale and $p\in(\max\{q_E,p_0,q_0\},\infty)$. Moreover, we allow $p=\max\{q_E,p_0,q_0\}$ if $q_E< q_0\leq p_0$ or if $E$ is a Hilbert space and $p_0=q_0=2$.
 \end{enumerate}
Then the images of balls with finite radius in $L_p(S)$ under the embedding
\[
 L_p(S)\hookrightarrow \mathcal{B}(\mathscr{A}^s,L_p(S;\mathscr{A}^s)), f\mapsto f\otimes(\,\cdot\,)
\]
are $\mathcal{R}$-bounded. More precisely, there is a constant $C>0$ such that for all $N\in\N$, $g_1,\ldots,g_N\in \mathscr{A}^s$ and all $f_1,\ldots,f_N\in L_p(S)$ we have the estimate
\[
 \left\| \sum_{k=1}^N \epsilon_k f_k \otimes g_k\right\|_{L_p(\Omega;L_p(S;\mathscr{A}^s))}\leq C\sup_{k=1,\ldots,n}\|f\|_{L_p(S)}\left\| \sum_{k=1}^N \epsilon_k  g_k\right\|_{L_p(\Omega;\mathscr{A}^s)}.
\]
\end{proposition}
\begin{proof}
 The cases $p\in(\max\{q_E,p_0\},\infty)$ in the Bessel potential case and $p\in(\max\{q_E,p_0,q_0\},\infty)$ in the Besov and Triebel-Lizorkin case follow from the result by Hyt\"onen and Veraar, Proposition \ref{Prop:HytoenenVeraar}, as in these cases $\mathscr{A}^s$ has cotype $\max\{q_E,p_0\}$ and $\max\{q_E,p_0,q_0\}$, respectively, see for example \cite[Proposition 7.1.4]{HvNVW_2017}. The Hilbert space cases follow directly since uniform boundedness and $\mathcal{R}$-boundedness coincide. The other cases in which $p=\max\{q_E,p_0\}$ or $p=\max\{q_E,p_0,q_0\}$ are allowed follow by Fubini's theorem together with the Kahane-Khintchine inequalities as in \cite[Remark 3.4]{Hytoenen_Veraar_2009}.
\end{proof}

For the mapping properties we derive later on, it is essential that we can use Mikhlin's multiplier theorem. There are many versions of this theorem available. For our purposes, the following will be sufficient.

\begin{theorem}\phantomsection \label{Thm:Mikhlin}\begin{enumerate}[(a)]
 \item\label{Thm:MikhlinA_p} Let $E$ be a UMD space, $p\in(1,\infty)$, $s\in\R$ and let $w$ be an $A_p$ weight. Let $m\in C^n(\R^n\setminus\{0\};\mathcal{B}(E))$ such that
 \[
  \kappa_m:=\mathcal{R}\big(\{|\xi|^{|\alpha|}D^{\alpha}m(\xi):\xi\in\R^n\setminus\{0\},|\alpha|\leq n\}\big)<\infty.
 \]
Then we have that
\[
 \|\mathscr{F}^{-1} m \mathscr{F} \|_{\mathcal{B}(H^s_p(\R^n,w;E))}\leq C\kappa_m
\]
with a constant $C>0$ only depending on $n$, $p$ and $E$.
\item\label{Thm:MikhlinA_pAlpha} Suppose that $E$ is a UMD space with Pisier's property $(\alpha)$. Let $p\in(1,\infty)$ and $w\in A_p(\R^n)$. Let further $\mathcal{T}\subset C^{n}(\R^n\setminus\{0\},\mathcal{B}(E))$. Then there is a constant $C>0$ independent of $\mathcal{T}$ such that
\[
 \mathcal{R}_{\mathcal{B}(H^s_p(\R^n,w;E))}(\{\mathscr{F}^{-1}m\mathscr{F}:m\in\mathcal{T}\})\leq C\kappa_{\mathcal{T}}
\]
where
\[
 \kappa_{\mathcal{T}}:=\mathcal{R}_{\mathcal{B}(E)}(\{|\xi|^{|\alpha|} D^{\alpha}m(\xi):\xi\in\R^n\setminus\{0\},\alpha\in\N_0^n,|\alpha|\leq n, m\in\mathcal{T}\}).
\]
\item\label{Thm:MikhlinBesovTriebel} Let $E$ be a Banach space, $p\in(1,\infty)$, $q\in[1,\infty]$ and $s\in\R$. Let further $w$ be an $A_{\infty}$ weight, $m\in C^{\infty}(\R^n,\mathcal{B}(E))$ and $\mathscr{A}^s_{p,q}(\R^n,w;E)\in\{B^s_{p,q}(\R^n,w;E),F^s_{p,q}(\R^n,w;E)\}$. Then there is a natural number $N\in\N$ and a constant $C>0$ not depending on $m$ such that
\[
 \|\mathscr{F}^{-1} m \mathscr{F} \|_{\mathcal{B}(\mathscr{A}^s_{p,q}(\R^n,w;E))}\leq C\kappa_m
\]
where
\[
 \kappa_m:=\sup_{|\alpha|\leq N}\sup_{\xi\in\R^n} \|\langle\xi\rangle^{|\alpha|}D^{\alpha}m(\xi)\|_{\mathcal{B}(E)}.
\]
The same holds if $E$ is reflexive, $p,q\in(1,\infty)$, $w\in[A_\infty]_p'$ and $\mathscr{A}^s_{p,q}(\R^n,w;E)\in\{\mathcal{B}^s_{p,q}(\R^n,w;E),\mathcal{F}^s_{p,q}(\R^n,w;E)\}$.
\item\label{Thm:MikhlinBesovTriebelRBounded} Let $E$ be a Banach space, $p\in(1,\infty)$, $q\in[1,\infty]$ and $s\in\R$. Let further $w$ be an $A_{\infty}$ weight, $\mathcal{T}\subset C^{\infty}(\R^n,\mathcal{B}(E))$ and $\mathscr{A}^s_{p,q}(\R^n,w;E)\in\{B^s_{p,q}(\R^n,w;E),F^s_{p,q}(\R^n,w;E)\}$. Then there is an $N\in\N$ and a constant $C>0$ independent of $\mathcal{T}$ such that
\[
 \mathcal{R}(\{\mathscr{F}^{-1} m \mathscr{F} : m\in\mathcal{T}\}\leq C \kappa_{\mathcal{T}}
\]
where
\[
 \kappa_{\mathcal{T}}:=\sup_{|\alpha|\leq N}\mathcal{R}(\{\langle\xi\rangle^{|\alpha|}D^{\alpha}m(\xi):\xi\in\R^n,m\in\mathcal{T}\}).
\]
The same holds if $E$ is a UMD space, $p,q\in(1,\infty)$, $w\in[A_\infty]_p'$ and $\mathscr{A}^s_{p,q}(\R^n,w;E)\in\{\mathcal{B}^s_{p,q}(\R^n,w;E),\mathcal{F}^s_{p,q}(\R^n,w;E)\}$.
\end{enumerate}
\end{theorem}
\begin{proof}
 Part \eqref{Thm:MikhlinA_p} with $s=0$ is contained in \cite[Theorem 1.2]{Fackler_Hytoenen_Lindemulder_2018}. The general case $s\in\R$ follows from $s=0$ by decomposing $m(\xi)=\langle\xi\rangle^{-s}m(\xi)\langle\xi\rangle^{s}$ and by using the definition of Bessel potential spaces. Part \eqref{Thm:MikhlinA_pAlpha} can be derived as \cite[5.2 (b)]{Kunstmann_Weis_2004}. The scalar-valued, unweighted version of part \eqref{Thm:MikhlinBesovTriebel} is contained in \cite[Paragraph 2.3.7]{Triebel_1983}. But the proof therein can be transferred to our situation by using \cite[Proposition 2.4]{Meyries_Veraar_2012}. Part \eqref{Thm:MikhlinBesovTriebelRBounded} is the isotropic version of \cite[Lemma 2.4]{Hummel_Lindemulder_2019}. The statements concerning the dual scales follow by duality. In the $\mathcal{R}$-bounded case we refer the reader to \cite[Proposition 8.4.1]{HvNVW_2017}.
\end{proof}

\begin{remark}
	For the dual scales in Theorem \ref{Thm:Mikhlin} \eqref{Thm:MikhlinBesovTriebelRBounded} it is actually not necessary to assume that $E$ is a UMD space. Instead \cite[Proposition 8.4.1]{HvNVW_2017} shows that K-convexity is good enough. But since we did not introduce K-convexity, we only stated the less general version here.
\end{remark}

\begin{remark}
 Later on, we sometimes want to apply Mikhlin's theorem for $m$ taking values in $\mathcal{B}(E^N,E^M)$ with certain $N,M\in\N$ instead of $\mathcal{B}(E)$. Note however that we can identify $\mathcal{B}(E^N,E^M)\simeq\mathcal{B}(E)^{M\times N}$. Hence, one can apply Mikhlin's theorem for each component and the statements of Theorem \ref{Thm:Mikhlin} transfer to the case in which $m$ takes values in $\mathcal{B}(E^N,E^M)$.
\end{remark}

Later on, we will also use parameter-dependent versions of our function spaces. They are natural to work with in the context of the parameter-dependent Boutet de Monvel calculus. And since we use elements of this parameter-dependent calculus, these spaces are also useful in our setting.
\begin{definition}
  Recall that Assumption \ref{Assump:Spaces} holds true. Let $\mu\in\C$ and $s,s_0\in\R$. Then we define the parameter-dependent weighted spaces
 \begin{align*}
  \mathscr{A}^{s,\mu,s_0}(\R^n,w;E)&:=\langle D,\mu\rangle^{s_0-s} \mathscr{A}^{s_0}(\R^n,w;E),\quad \|\cdot\|_{\mathscr{A}^{s,\mu,s_0}(\R^n,w;E)}:=\|\langle D,\mu\rangle^{s-s_0}\cdot\|_{\mathscr{A}^{s_0}(\R^n,w;E)},
 \end{align*}
where $\langle D,\mu\rangle:=\mathscr{F}^{-1}\langle\xi,\mu\rangle\mathscr{F}=\mathscr{F}^{-1}(1+|\xi|^2+|\mu|^2)^{1/2}\mathscr{F}$.
\end{definition}

\begin{lemma}\label{Lemma:EquivalenceParameterDependentNorms}
Let  $\mu\in\C$ and $s,s_0\in\R$. We have the estimates
\begin{align*}
 \|\cdot\|_{\mathscr{A}^{s,\mu,s_0}}\eqsim \|\cdot\|_{\mathscr{A}^{s}}+\langle\mu\rangle^{s-s_0} \|\cdot\|_{\mathscr{A}^{s_0}}\qquad&\text{if }s-s_0\geq0,\\ 
 \|\cdot\|_{\mathscr{A}^{s,\mu,s_0}} \lesssim \|\cdot\|_{\mathscr{A}^{s}} \lesssim \langle\mu\rangle^{s_0-s}\|\cdot\|_{\mathscr{A}^{s,\mu,s_0}}\qquad&\text{if }s-s_0\leq0.
\end{align*}
\end{lemma}
\begin{proof}
 Assumption \ref{Assump:Spaces} is formulated in a way such that we can apply our versions of the Mikhlin multiplier theorem, Theorem \ref{Thm:Mikhlin} \eqref{Thm:MikhlinA_p} and \eqref{Thm:MikhlinBesovTriebel}. Let first $s\geq s_0$. Note that the function
 \[
  m\colon \R^n\times\C\to \R,\,(\xi,\mu)\mapsto \frac{\langle\xi,\mu\rangle^{s-s_0}}{\langle\xi\rangle^{s-s_0}+\langle\mu\rangle^{s-s_0}}
 \]
 satisfies the condition from Theorem \ref{Thm:Mikhlin} uniformly in $\mu$. Indeed, by induction it follows that $\partial^{\alpha}_{\xi}m(\xi,\mu)$ $(\alpha\in\N_0^n)$ is a linear combination of terms of the form
 \[
  p_{\beta,i,j,k}(\xi,\mu)=\xi^\beta\langle\xi,\mu\rangle^{s-s_0-i}\langle\xi\rangle^{(s-s_0-2)j-k}(\langle\xi\rangle^{s-s_0}+\langle\mu\rangle^{s-s_0})^{-1-j}
 \]
 for some $\beta\in\N_0^n$ and $i,k,j\in\N_0$ such that $|\alpha|=i+2j+k-|\beta|$. But each of these terms satisfies
 \begin{align*}
  \langle\xi\rangle^{|\alpha|}|p_{\beta,i,j,k}(\xi,\mu)|&=m(\xi,\mu)\langle\xi\rangle^{|\alpha|}|\xi^\beta|\langle\xi,\mu\rangle^{-i}\langle\xi\rangle^{(s-s_0-2)j-k}(\langle\xi\rangle^{s-s_0}+\langle\mu\rangle^{s-s_0})^{-j}\\
  &\lesssim  m(\xi,\mu)\langle\xi\rangle^{|\alpha|+|\beta|-i-2j-k+(s-s_0)j-(s-s_0)j}\\
  &\lesssim  m(\xi,\mu)\\
  &\lesssim 1.
 \end{align*}
Hence, $(m(\cdot,\mu))_{\mu\in\C}$ is a bounded family of Fourier multipliers. Using this, we obtain
 \begin{align*}
  \|u\|_{\mathscr{A}^{s,\mu,s_0}}&=\|\langle D,\mu\rangle^{s-s_0}u\|_{\mathscr{A}^{s_0}}=\big\|m(D,\mu)(\langle D\rangle^{s-s_0}+\langle\mu\rangle^{s-s_0})u\big\|_{\mathscr{A}^{s_0}}\lesssim \big\|(\langle D\rangle^{s-s_0}+\langle\mu\rangle^{s-s_0})u\big\|_{\mathscr{A}^{s_0}}\\
  &\leq \|u\|_{\mathscr{A}^{s}}+\langle\mu\rangle^{s-s_0}\|u\|_{\mathscr{A}^{s_0}}=\bigg\|\frac{\langle D\rangle^{s-s_0}}{\langle D,\mu\rangle^{s-s_0}}\langle D,\mu\rangle^{s-s_0}u\bigg\|_{\mathscr{A}^{s_0}}+\bigg\|\frac{\langle \mu\rangle^{s-s_0}}{\langle D,\mu\rangle^{s-s_0}}\langle D,\mu\rangle^{s-s_0}u\bigg\|_{\mathscr{A}^{s_0}}\\
  &\lesssim \|\langle D,\mu\rangle^{s-s_0}u\|_{\mathscr{A}^{s_0}}= \|u\|_{\mathscr{A}^{s,\mu,s_0}}
 \end{align*}
for $s-s_0\geq0$ and
\begin{align*}
 \|u\|_{\mathscr{A}^{s,\mu,s_0}}&=\|\langle D,\mu\rangle^{s-s_0}u\|_{\mathscr{A}^{s_0}}=\bigg\|\frac{\langle D,\mu\rangle^{s-s_0}}{\langle D\rangle^{s-s_0}}\langle D\rangle^{s-s_0}u\bigg\|_{\mathscr{A}^{s_0}}\lesssim\|\langle D\rangle^{s-s_0}u\|_{\mathscr{A}^{s_0}}\\
 &=\|u\|_{\mathscr{A}^{s}}=\bigg\|\frac{\langle D\rangle^{s-s_0}}{\langle D,\mu\rangle^{s-s_0}}\langle D,\mu\rangle^{s-s_0}u\bigg\|_{\mathscr{A}^{s_0}}\lesssim \langle\mu\rangle^{s_0-s}\|u\|_{\mathscr{A}^{s,\mu,s_0}}
\end{align*}
for $s-s_0\leq0$.
\end{proof}

In this paper, we also consider function spaces on open intervals $I$. In this case, we can just define them by restriction.
\begin{definition}
 Let be $I\subset\R$ an open interval. Then we define $(\mathscr{A}^{\bullet}(I,w;E),\|\cdot\|_{\mathscr{A}^{\bullet}(I,w;E)})$ by
 \[
  \mathscr{A}^{\bullet}(I,w;E)=\{f|_I:f\in\mathscr{A}^{\bullet}(I,w;E)\},\quad\|g\|_{\mathscr{A}^{\bullet}(I,w;E)}:=\inf_{f\in\mathscr{A}^{\bullet}(\R,w;E), f|_I=g} \|f\|_{\mathscr{A}^{\bullet}(\R,w;E)}.
 \]
 We use the same definition for $\mathscr{B}^{\bullet}$ and $\mathscr{C}^{\bullet}$.
\end{definition}

\begin{remark}
 Recall that we defined $W^{-m}_p(\mathcal{O},w;E)$ as the dual of $W^m_{p'}(\mathcal{O},w^{-\frac{1}{p-1}};E')$ and not by restriction. In the scalar-valued unweighted setting both definitions coincide, see \cite[Section 2.10.2]{Triebel_1978}. We believe that the same should hold true under suitable assumptions in the weighted vector-valued setting. But since this is not important for this work, we do not investigate this any further.
\end{remark}

\begin{lemma}\label{Lemma:EmbeddingNegitveSmoothness}
 Let $s\in\R$, $p\in(1,\infty)$, $r\in(-1,p-1)$ and $l\in\N$. Suppose that $\mathscr{A}^{s}$ is reflexive. Then we have the continuous embeddings
 \begin{align*}
  L_p(\R_+,|\operatorname{pr}_n|^{r+lp}; &\mathscr{A}^s)\hookrightarrow  W^{-l}_{p}(\R_+,|\operatorname{pr}_n|^{r}; \mathscr{A}^s)
 \end{align*}
\end{lemma}
\begin{proof}
 We should note that almost the same proof was given in \cite{Lindemulder_2018}. By duality, it suffices to prove
 \[
  W^{l}_{p',0}(\R_+,|\operatorname{pr}_n|^{r'};(\mathscr{A}^{s})')\hookrightarrow L_{p'}(\R_+,|\operatorname{pr}_n|^{r'-lp'};(\mathscr{A}^{s})')
 \]
where $r'=-\frac{r}{p-1}$ and $p'=\frac{p}{p-1}$. But this is a special case of \cite[Corollary 3.4]{Lindemulder_Veraar_2018}.
\end{proof}


\section{Pseudo-differential Operators in Mixed Scales} \label{Section:Pseudos}
Now we briefly introduce some notions and notations concerning pseudo-differential operators. Since we only use the $x$-independent case in the following, we could also formulate our results in terms of Fourier multipliers. However, parameter-dependent H\"ormander symbol classes provide a suitable framework for the formulation of our results. In the case of parameter-dependent symbols, we oftentimes consider spaces of smooth functions on an open set $U\subset \R^n\times\C$. In this case, we identify $\C\simeq\R^2$ and understand the differentiability in the real sense. If we want to understand it in the complex sense, we say \textit{holomorphic} instead of \textit{smooth}.

\begin{definition}
 Let $Z$ be a Banach space, $d\in\R$, $\Sigma\subset\C$ open and $\vartheta\colon\Sigma\to(0,\infty)$ a function.
 \begin{enumerate}[(a)]
  \item The space of parameter-independent H\"ormander symbols $S^d(\R^n;Z)$ of order $d$ is the space of all smooth functions $p\in C^{\infty}(\R^n;Z)$ such that
 \[
  \|p\|^{(d)}_k:=\sup_{\xi\in\R^n, \atop\alpha\in\N_0^n, |\alpha|\leq k} \langle\xi\rangle^{-(d-|\alpha|)} \|D^{\alpha}_{\xi}p(\xi)\|_{Z}<\infty
 \]
for all $k\in\N_0$.
  \item The space of parameter-dependent H\"ormander symbols $S^{d,\vartheta}(\R^n\times\Sigma;Z)$ of order $d$ is the space of all smooth functions $p\in C^{\infty}(\R^n\times\Sigma;Z)$ such that
 \[
  \|p\|^{(d,\vartheta)}_k:=\sup_{\alpha\in\N_0^n,\,\gamma\in\N_0^2\atop |\alpha|+|\gamma|\leq k}\sup_{\xi\in\R^n,\mu\in\Sigma} \vartheta(\mu)^{-1}\langle\xi,\mu\rangle^{-(d-|\alpha|_1-|\gamma|_1)} \|D^{\alpha}_{\xi}D_{\mu}^{\gamma}p(\xi,\mu)\|_{Z}<\infty
 \]
for all $k\in\N_0$. If $\vartheta=1$, then we also omit it in the notation.
 \end{enumerate}
\end{definition}
Actually, if one omits the weight function $\vartheta$, then the latter symbol class is the special case of parameter-dependent H\"ormander symbols with regularity $\infty$. Usually, one also includes the regularity parameter $\nu$ in the notation of the symbol class, so that the notation $S^{d,\infty}(\R^n\times\Sigma;Z)$ is more common in the literature. But since the symbols in this paper always have infinite regularity, we omit $\infty$ in the notation.\\
For the Bessel potential case, $\mathcal{R}$-bounded versions of these symbol classes are useful.
\begin{definition}
 Let $E$ be a Banach space, $N,M\in\N$, $d\in\R$, $\Sigma\subset\C$ open and $\vartheta\colon\Sigma\to(0,\infty)$ a function.
 \begin{enumerate}[(a)]
  \item By $S^d_{\mathcal{R}}(\R^n;\mathcal{B}(E^N,E^M))$ we denote the space of all smooth functions $p\in C^{\infty}(\R^n;\mathcal{B}(E^N,E^M))$ such that
 \[
  \|p\|^{(d)}_{k,\mathcal{R}}:=\mathcal{R}\big\{\langle\xi\rangle^{-(d-|\alpha|_1)}D^{\alpha}_{\xi}p(\xi):\xi\in\R^n, \alpha\in\N_0^n, |\alpha|\leq k \big\}<\infty
 \]
for all $k\in\N_0$.
  \item By $S^{d,\vartheta}_{\mathcal{R}}(\R^n\times\Sigma;\mathcal{B}(E^N,E^M))$ we denote the space of all smooth functions $p\in C^{\infty}(\R^n\times\Sigma;\mathcal{B}(E^N,E^M))$ such that
 \[
  \|p\|^{(d,\vartheta)}_{k,\mathcal{R}}:=\mathcal{R}\big\{\vartheta(\mu)^{-1}\langle\xi,\mu\rangle^{-(d-|\alpha|-|\gamma|)} D^{\alpha}_{\xi}D_{\mu}^{\gamma}p(\xi,\mu): \xi\in\R^n,\mu\in\Sigma,\alpha\in\N_0^n,\gamma\in\N_0^2,|\alpha|+|\gamma|\leq k\big\}<\infty
 \]
for all $k\in\N_0$. If $\vartheta=1$, then we also omit it in the notation.
 \end{enumerate}
\end{definition}

\begin{remark}\phantomsection \label{Rem:SymbolClasses} 
 \begin{enumerate}[(a)]
  \item It seems like $\mathcal{R}$-bounded versions of the usual H\"ormander symbol classes have first been considered in the Ph.D. thesis of \v{S}trkalj. We also refer to \cite{Portal_Strkalj_2006}.
  \item It was observed in \cite{Denk_Krainer_2007} that also the $\mathcal{R}$-bounded symbol classes are Fr\'echet spaces.
  \item Since uniform bounds can be estimated by $\mathcal{R}$-bounds, we have the continuous embeddings 
  \[
   S^d_{\mathcal{R}}(\R^n;\mathcal{B}(E^N,E^M))\hookrightarrow S^d(\R^n;\mathcal{B}(E^N,E^M)),\quad S^{d,\vartheta}_{\mathcal{R}}(\R^n\times\Sigma;\mathcal{B}(E^N,E^M))\hookrightarrow S^{d,\vartheta}(\R^n\times\Sigma;\mathcal{B}(E^N,E^M)).
  \]
 \item Since uniform boundedness and $\mathcal{R}$-boundedness for a set of scalars are equivalent, we have that
 \[
  S^d(\R^n;\C)\hookrightarrow S^d_{\mathcal{R}}(\R^n;\mathcal{B}(E^N)),\quad S^{d,\vartheta}(\R^n\times\Sigma;\C)\hookrightarrow S^{d,\vartheta}_{\mathcal{R}}(\R^n\times\Sigma;\mathcal{B}(E^N)).
 \]
 \item\label{Rem:SymbolClasses:Product} Given $d_1,d_2\in\R$, $\vartheta_1,\vartheta_2\colon\Sigma\to(0,\infty)$ and $N_1,N_2,N_3\in\N$ we have the continuous bilinear mappings
 \begin{align*}
 &S^{d_2}(\R^n;\mathcal{B}(E^{N_2},E^{N_3}))\times S^{d_1}(\R^n;\mathcal{B}(E^{N_1},E^{N_2}))\to S^{d_1+d_2}(\R^n;\mathcal{B}(E^{N_1},E^{N_3})),\,(p_2,p_1)\mapsto p_2p_1,\\
 &S^{d_2,\vartheta_2}(\R^n\times\Sigma;\mathcal{B}(E^{N_2},E^{N_3}))\times S^{d_1,\vartheta_1}(\R^n\times\Sigma;\mathcal{B}(E^{N_1},E^{N_2}))\\
 &\qquad\qquad\qquad\qquad\qquad\qquad\to S^{d_1+d_2,\vartheta_1\cdot\vartheta_2}(\R^n\times\Sigma;\mathcal{B}(E^{N_1},E^{N_3})),\,(p_2,p_1)\mapsto p_2p_1.
 \end{align*}
 The same properties also hold for the $\mathcal{R}$-bounded versions.
 \item\label{Rem:SymbolClasses:Derivative} The differential operator $\partial^{\alpha}$ with $\alpha\in\N_0^n$ is a continuous linear operator
  \begin{align*}
 S^{d}(\R^n;\mathcal{B}(E^{N},E^{M}))\to S^{d-|\alpha|}(\R^n;\mathcal{B}(E^{N},E^{M})),\,p\mapsto \partial^{\alpha}p,\\
 S^{d,\vartheta}(\R^n\times\Sigma;\mathcal{B}(E^{N},E^{M}))\to S^{d-|\alpha|,\vartheta}(\R^n\times\Sigma;\mathcal{B}(E^{N},E^{M})),\,p\mapsto \partial^{\alpha}p.
 \end{align*}
The same properties also hold for the $\mathcal{R}$-bounded versions.
\item \label{Rem:SymbolClasses:DependentIndependent} One could also view parameter-independent symbol classes as a subset of a parameter-dependent symbol classes with bounded $\Sigma\subset\C$ which consists of those symbols which do not depend on the parameter $\mu$. Hence, the statements we formulate for parameter-dependent symbol classes in the following also hold in a similar way in the parameter-independent case.
 \end{enumerate}
\end{remark}

\begin{definition}
 Let $E$ be a Banach space, $d\in\R$ and $\Sigma\subset\C$ open. Let further $p\in S^d(\R^n\times\Sigma;\mathcal{B}(E^{N},E^{M}))$. Then we define the corresponding pseudo-differential operator by
 \[
  (P_{\mu}f)(x):=(\operatorname{op}[p(\,\cdot\,,\mu)]f)(x):=[\mathscr{F}^{-1}p(\,\cdot\,\mu)\mathscr{F}f](x)=\frac{1}{(2\pi)^{n/2}}\int_{\R^n}e^{ix\xi}p(\xi,\mu)\widehat{f}(\xi)\,d\xi.
 \]
for $f\in\mathscr{S}(\R^n;E^N)$.
\end{definition}

Since we only consider $x$-independent symbols, the mapping properties of such pseudo-differential operators are an easy consequence of Mikhlin's theorem.

\begin{proposition}\label{Prop:PseudoMappingProperties}
 Let $N,M\in\N$, $s,s_0,d\in\R$, $\Sigma\subset\C$ open and $\vartheta\colon\Sigma\to(0,\infty)$ a function. Consider one of the following two cases
  \begin{enumerate}[(a)]
     \item\label{Prop:PseudoMappingProperties:PseudoBessel} $\mathscr{A}^{\bullet}$ belongs to the Bessel potential scale and $S^{d,\vartheta}_{\mathscr{A}}(\R^{n}\times\Sigma;\mathcal{B}(E^N,E^M))=S^{d,\vartheta}_{\mathcal{R}}(\R^{n}\times\Sigma;\mathcal{B}(E^N,E^M))$.
  \item\label{Prop:PseudoMappingProperties:PseudoBesovTriebel} $\mathscr{A}^{\bullet}$ belongs to the Besov or the Triebel-Lizorkin scale and $S^{d,\vartheta}_{\mathscr{A}}(\R^{n}\times\Sigma;\mathcal{B}(E^N,E^M))=S^{d,\vartheta}(\R^{n}\times\Sigma;\mathcal{B}(E^N,E^M))$.
 \end{enumerate}
Then the mapping
 \[
  S^{d,\vartheta}_{\mathscr{A}}(\R^{n}\times\Sigma;\mathcal{B}(E^N,E^M))\times \mathscr{A}^{s+d,\mu,s_0}(\R^{n},w_0,E^N)\to \mathscr{A}^{s,\mu,s_0}(\R^{n},w_0,E^M),\; (p,f)\mapsto \operatorname{op}[p(\,\cdot\,,\mu)]f
 \]
 defined by extension from $\mathscr{S}(\R^{n},E^N)$ to $\mathscr{A}^{s+d,\mu,s_0}(\R^{n},w_0,E^N)$ is bilinear and continuous. Moreover, there is a constant $C>0$ independent of $\vartheta$ such that
 \[
  \| \operatorname{op}[p(\,\cdot\,,\mu)]f\|_{ \mathscr{A}^{s,\mu,s_0}(\R^{n},w_0,E^N)} \leq C \vartheta(\mu)\|p\|^{(d,\vartheta)}_N\|f\|_{ \mathscr{A}^{s+d,\mu,s_0}(\R^{n},w_0,E^M)}
 \]
  for all $\mu\in\Sigma$.
\end{proposition}
\begin{proof}
 It is obvious that the mapping is bilinear. For the continuity, we note that
 \[
  S^{d,\vartheta}_{\mathscr{A}}(\R^n\times\Sigma;\mathcal{B}(E^N,E^M))\to S^{0,\vartheta}_{\mathscr{A}}(\R^n\times\Sigma;\mathcal{B}(E^N,E^M)), p\mapsto[(\xi,\mu)\mapsto p(\xi,\mu)\langle\xi,\mu\rangle^{-d}]
 \]
is continuous. Hence, by Mikhlin's theorem there is an $N'\in\N$ such that
\begin{align*}
 \| \operatorname{op}[p(\,\cdot\,,\mu)]f\|_{ \mathscr{A}^{s,\mu,s_0}(\R^{n},w_0,E^M)}&=\| \operatorname{op}[\langle\cdot,\mu\rangle^{s_0-s}]\operatorname{op}[p(\,\cdot\,,\mu)\langle\cdot,\mu\rangle^{-d}]\operatorname{op}[\langle\cdot,\mu\rangle^{d+s-s_0}]f\|_{ \mathscr{A}^{s,\mu,s_0}(\R^{n},w_0,E^M)}\\
 &= \|\operatorname{op}[p(\,\cdot\,,\mu)\langle\cdot,\mu\rangle^{-d}]\operatorname{op}[\langle\cdot,\mu\rangle^{d+s-s_0}]f\|_{ \mathscr{A}^{s_0}(\R^{n},w_0,E^M)}\\
 &\lesssim \vartheta(\mu)\|p\|^{(d,\vartheta)}_{N',\mathscr{A}}\|\operatorname{op}[\langle\cdot,\mu\rangle^{d+s-s_0}]f\|_{ \mathscr{A}^{s_0}(\R^{n},w_0,E^N)}\\
 &=  \vartheta(\mu)\|p\|^{(d,\vartheta)}_{N',\mathscr{A}}\|f\|_{ \mathscr{A}^{s+d,\mu,s_0}(\R^{n},w_0,E^N)}.
\end{align*}
This also shows the asserted estimate.
\end{proof}

We can also formulate an $\mathcal{R}$-bounded version of Proposition \ref{Prop:PseudoMappingProperties} without the parameter-dependence of the function spaces.
\begin{proposition}\label{Prop:PseudoMappingProperties_R_bounded}
 Let $N,M\in\N$, $s,d\in\R$, $\Sigma\subset\C$ open and $\vartheta\colon\Sigma\to(0,\infty)$ a function. Consider one of the following two cases
  \begin{enumerate}[(a)]
     \item\label{Prop:PseudoMappingProperties_R_bounded:PseudoBessel} $\mathscr{A}^{\bullet}$ belongs to the Bessel potential scale and $E$ satisfies Pisier's property $(\alpha)$ in addition to the assumptions in Assumption \ref{Assump:Spaces}.
  \item\label{Prop:PseudoMappingProperties_R_bounded:PseudoBesovTriebel} $\mathscr{A}^{\bullet}$ belongs to the Besov or the Triebel-Lizorkin scale.
 \end{enumerate}
Then the mapping
 \[
  S^{d,\vartheta}_{\mathcal{R}}(\R^{n}\times\Sigma;\mathcal{B}(E^N,E^M))\times \mathscr{A}^{s+d}(\R^{n},w_0;E^N)\to \mathscr{A}^{s}(\R^{n},w_0;E^M),\; (p,f)\mapsto \operatorname{op}[p(\,\cdot\,,\mu)]f
 \]
 defined by extension from $\mathscr{S}(\R^n,E^N)$ to $\mathscr{A}^{s,\mu,s_0}(\R^{n},w_0;E^N)$ is bilinear and continuous. Moreover, there is a constant $C>0$ independent of $\vartheta$ such that
 \[
  \mathcal{R}_{\mathcal{B}(\mathscr{A}^{s+d}(\R^{n},w_0;E^N),\mathscr{A}^{s}(\R^{n},w_0;E^M))}\big(\{\vartheta(\mu)^{-1}\langle\mu\rangle^{-d_+}\operatorname{op}[p(\,\cdot\,,\mu)]:\mu\in\Sigma\}\big)\leq C \|p\|^{(d,\vartheta)}_N.
 \]
\end{proposition}
\begin{proof}
Note that $m(\,\cdot\,,\mu):=[\xi\mapsto \langle\mu\rangle^{-d_+}\langle\xi,\mu\rangle^d\langle\xi\rangle^{-d}]$ satisfies Mikhlin's condition uniformly in $\mu$. Indeed, by induction on $|\alpha|$ one gets that $\partial^{\alpha}\langle\mu\rangle^{-d_+}\langle\xi,\mu\rangle^d\langle\xi\rangle^{-d}$ is a linear combination of terms of the form
\[
 p_{j,k}(\xi,\mu)=\xi^{\beta}\langle\xi,\mu\rangle^{d-2j}\langle\xi\rangle^{-d-2k}\langle\mu\rangle^{-d_+}
\]
for some $\beta\in\N_0^{n-1}$, $j,k\in\N_0$ with $|\alpha|=2j+2k-|\beta|$. For such a term we obtain
\begin{align*}
 \langle\xi\rangle^{|\alpha|}|p_{j,k}(\xi,\mu)|&=m(\xi,\mu)|\xi^{\beta}|\langle\xi,\mu\rangle^{-2j}\langle\xi\rangle^{|\alpha|-2k}\\
 &\leq m(\xi,\mu)\langle\xi\rangle^{|\alpha|+|\beta|-2j-2k}\\
 &\lesssim 1.
\end{align*}
Hence, by Mikhlin's theorem there is an $N'\in\N$ such that
\begin{align*}
  &\quad\mathcal{R}_{\mathcal{B}(\mathscr{A}^{s+d}(\R^{n-1},w_0;E^N),\mathscr{A}^{s}(\R^{n-1},w_0;E^M))}\big(\{\vartheta(\mu)^{-1}\langle\mu\rangle^{-d_+}\operatorname{op}[p(\,\cdot\,,\mu)]:\mu\in\Sigma\}\big)\\
  &= \mathcal{R}_{\mathcal{B}(\mathscr{A}^{s+d}(\R^{n-1},w_0;E^N),\mathscr{A}^{s}(\R^{n-1},w_0;E^M))}\big(\{\vartheta(\mu)^{-1}\operatorname{op}[p(\,\cdot\,,\mu)\langle\cdot,\mu\rangle^{-d}]\operatorname{op}[\langle\mu\rangle^{-d_+} \langle\cdot,\mu\rangle^{d}]:\mu\in\Sigma\}\big)\\
 &\lesssim\mathcal{R}_{\mathcal{B}(\mathscr{A}^{s}(\R^{n-1},w_0;E^N),\mathscr{A}^{s}(\R^{n-1},w_0;E^M))}\big(\{\vartheta(\mu)^{-1}\operatorname{op}[p(\,\cdot\,,\mu)\langle\cdot,\mu\rangle^{-d}]:\mu\in\Sigma\}\big)\\
 &\lesssim \|p\|^{(d,\vartheta)}_{N'}.
\end{align*}
\end{proof}

\begin{proposition}[Iterated version of Mikhlin's theorem]\label{Prop:IteratedMikhlin}
 Let $s,k\in\R$ and let $E$ be a Banach space. Consider one of the following cases with Assumption \ref{Assump:Spaces} in mind, $\mathscr{B}$ being defined on $I_{x_n}=\R$ and with $m\in L_{\infty}(\R^n;\mathcal{B}(E))$ being smooth enough:
 \begin{enumerate}[(a)]
  \item\label{Prop:IteratedMikhlin:TLTL} Neither $\mathscr{A}$ nor $\mathscr{B}$ stands for the Bessel potential scale. For $N\in\N_0$ we define
  \[
   \kappa_{m,N}:=\sup\big\{\langle\xi'\rangle^{|\alpha'|}\langle\xi_n\rangle^{\alpha_n}\|\partial_{\xi}^{\alpha}m(\xi',\xi_n)\|_{\mathcal{B}(E)}: \alpha\in\N_0^{n},|\alpha|\leq N, \xi\in\R^{n}\big\}.
  \]
  \item\label{Prop:IteratedMikhlin:BTL} $\mathscr{A}$ stands for the Bessel potential scale and $\mathscr{B}$ does not stand for the Bessel potential scale. For $N\in\N_0$ we define
   \[
   \kappa_{m,N}:=\sup_{\xi_n\in\R,\alpha_n\in\N_0,\alpha_n\leq N}\mathcal{R}\big\{|\xi'|^{|\alpha'|}\langle\xi_n\rangle^{\alpha_n}\partial_{\xi}^{\alpha}m(\xi',\xi_n): \alpha'\in\N_0^{n-1},|\alpha'|\leq N, \xi'\in\R^{n-1}\big\}.
  \]
  \item\label{Prop:IteratedMikhlin:TLB} $\mathscr{B}$ stands for the Bessel potential scale and $\mathscr{A}$ does not stand for the Bessel potential scale. For $N\in\N_0$ we define
     \[
   \kappa_{m,N}:=\mathcal{R}\big\{\langle\xi'\rangle^{|\alpha'|}|\xi_n|^{\alpha_n}\partial_{\xi}^{\alpha}m(\xi',\xi_n): \alpha\in\N_0^{n},|\alpha|\leq N, \xi\in\R^{n}\big\}.
  \]
  \item\label{Prop:IteratedMikhlin:BB} Both $\mathscr{A}$ and $\mathscr{B}$ stand for the Bessel potential scale and $E$ satisfies Pisier's property $(\alpha)$. For $N\in\N_0$ we define
 \[
   \kappa_{m,N}:=\mathcal{R}\big\{|\xi'|^{|\alpha'|}|\xi_n|^{\alpha_n}\partial_{\xi}^{\alpha}m(\xi',\xi_n): \alpha\in\N_0^{n},|\alpha|\leq N, \xi\in\R^{n}\big\}.
  \]
 \end{enumerate}
 There is an $N\in\N_0$ and a constant $C>0$ independent of $m$ such that
 \[
  \|\operatorname{op}[m]\|_{\mathcal{B}(\mathscr{B}^{k}(\mathscr{A}^s))}\leq C\kappa_{m,N}.
 \]
\end{proposition}
\begin{proof}
First, we note that $\operatorname{op}[\partial_{\xi_n}^{\alpha_n}m(\,\cdot,\xi_n)]=\partial_{\xi_n}^{\alpha_n}\operatorname{op}[m(\,\cdot,\xi_n)]$, $\alpha_n\in\N$, if $m$ is smooth enough. Indeed, let $\epsilon>0$ be small enough and $h\in (-\epsilon,\epsilon)$. Then we have
\begin{align*}
 &\quad \big\|\operatorname{op}\big[\tfrac{1}{h}(m(\,\cdot,\xi_n+h)-m(\,\cdot,\xi_n))-\partial_nm(\,\cdot,\xi_n) \big]\big\|_{\mathcal{B}(\mathscr{A}^s,\mathscr{A}^s)}\\
 &\leq C\sup_{\alpha'\in\N_0^{n-1},|\alpha'|\leq N'}\sup_{\xi'\in\R^{n-1}}\|\langle\xi'\rangle^{|\alpha'|}\partial_{\xi'}^{\alpha'}[\tfrac{1}{h}(m(\xi',\xi_n+h)-m(\xi',\xi_n))-\partial_nm(\xi',\xi_n) ]\|_{\mathcal{B}(E)}\\
  &= C\sup_{\alpha'\in\N_0^{n-1},|\alpha'|\leq N'}\sup_{\xi'\in\R^{n-1}}\bigg\|\langle\xi'\rangle^{|\alpha'|}\partial_{\xi'}^{\alpha'}\bigg[\int_0^1\partial_n m(\xi',\xi_n+sh)-\partial_n m(\xi',\xi_n)\,ds\bigg]\bigg\|_{\mathcal{B}(E)}\\
 &\leq C\sup_{\alpha'\in\N_0^{n-1},|\alpha'|\leq N'}\sup_{\xi'\in\R^{n-1}}\sup_{s\in[0,1]}\langle\xi'\rangle^{|\alpha'|}\|\partial_{\xi'}^{\alpha'}\partial_n[m(\xi',\xi_n+s h)-m(\xi',\xi_n)]\|_{\mathcal{B}(E)}
\end{align*}
Now we can use the uniform continuity of $$\R^{n-1}\times(-\epsilon,\epsilon)\to\mathcal{B}(E),\,(\xi',h)\mapsto \langle\xi'\rangle^{|\alpha'|}\partial_{\xi'}^{\alpha'}\partial_n m(\xi',\xi_n+h))$$ to see that we have convergence to $0$ as $h\to0$ in the above estimate. The uniform continuity follows from the boundedness of the derivatives (if $m$ is smooth enough). For derivatives of order $\alpha_n\geq2$ we can apply the same argument to $\partial_{\xi_n}^{\alpha_n-1}m$.\\
 The idea is now to apply Miklhin's theorem twice. For example in case \eqref{Prop:IteratedMikhlin:BB} one obtains
 \begin{align*}
 \| \operatorname{op}[m]\|_{\mathcal{B}(\mathscr{B}^{k}(\mathscr{A}^s)))}&\lesssim\mathcal{R}_{\mathcal{B}(\mathscr{A}^s,\mathscr{A}^s)}\big(\{|\xi_n|^{\alpha_n}\partial_{\xi_n}^{\alpha_n}\operatorname{op}[m(\,\cdot,\xi_n)]:\alpha_n\in\N,\alpha_n\leq N_n,\xi_n\in\R\}\big)\\
 &=\mathcal{R}_{\mathcal{B}(\mathscr{A}^s,\mathscr{A}^s)}\big(\{\operatorname{op}[|\xi_n|^{\alpha_n}\partial_{\xi_n}^{\alpha_n}m(\,\cdot,\xi_n)]:\alpha_n\in\N,\alpha_n\leq N_n,\xi_n\in\R\}\big)\\
 &\lesssim \kappa_{m,N}
\end{align*}
by Theorem \ref{Thm:Mikhlin} \eqref{Thm:MikhlinA_pAlpha} for $N_n,N\in\N_0$ large enough. The other cases are obtained analogously.
\end{proof}

There also is an $\mathcal{R}$-bounded version of Proposition \ref{Prop:IteratedMikhlin}
\begin{proposition}[Iterated $\mathcal{R}$-bounded version of Mikhlin's theorem]\label{Prop:IteratedMikhlinRbounded}
 Let $s,k\in\R$ and let $E$ be a Banach space. Consider one of the following cases with Assumption \ref{Assump:Spaces} in mind and with $\mathcal{M}\subset C^{\tilde{N}}(\R^n\setminus\{0\};\mathcal{B}(E))$ with $\tilde{N}\in\N_0$ being large enough:
 \begin{enumerate}[(a)]
  \item\label{Prop:IteratedMikhlinRbounded:TLTL} Neither $\mathscr{A}$ nor $\mathscr{B}$ stands for the Bessel potential scale. For $N\in\N_0$ we define
  \[
   \kappa_{\mathcal{M},N}:=\mathcal{R}\big\{\langle\xi'\rangle^{|\alpha'|}\langle\xi_n\rangle^{\alpha_n}\partial_{\xi}^{\alpha}m(\xi): m\in\mathcal{M},\alpha\in\N_0^{n},|\alpha|\leq N, \xi\in\R^{n}\big\}.
  \]
  \item\label{Prop:IteratedMikhlinRbounded:BTL} $\mathscr{A}$ stands for the Bessel potential scale, $\mathscr{B}$ does not stand for the Bessel potential scale and $E$ satisfies Pisier's property $(\alpha)$. For $N\in\N_0$ we define
   \[
   \kappa_{\mathcal{M},N}:=\mathcal{R}\big\{|\xi'|^{|\alpha'|}\langle\xi_n\rangle^{\alpha_n}\partial_{\xi}^{\alpha}m(\xi): m\in\mathcal{M},\alpha\in\N_0^{n},|\alpha|\leq N, \xi\in\R^{n}\big\}.
  \]
  \item\label{Prop:IteratedMikhlinRbounded:TLB} $\mathscr{B}$ stands for the Bessel potential scale, $\mathscr{A}$ does not stand for the Bessel potential scale and $E$ satisfies Pisier's property $(\alpha)$. For $N\in\N_0$ we define
       \[
   \kappa_{\mathcal{M},N}:=\mathcal{R}\big\{\langle\xi'\rangle^{|\alpha'|}|\xi_n|^{\alpha_n}\partial_{\xi}^{\alpha}m(\xi): m\in\mathcal{M},\alpha\in\N_0^{n},|\alpha|\leq N, \xi\in\R^{n}\big\}.
  \]
  \item\label{Prop:IteratedMikhlinRbounded:BB} Both $\mathscr{A}$ and $\mathscr{B}$ stand for the Bessel potential scale and $E$ satisfies Pisier's property $(\alpha)$. For $N\in\N_0$ we define
 \[
   \kappa_{\mathcal{M},N}:=\mathcal{R}\big\{|\xi'|^{|\alpha'|}|\xi_n|^{\alpha_n}\partial_{\xi}^{\alpha}m(\xi',\xi_n):m\in\mathcal{M}, \alpha\in\N_0^{n},|\alpha|\leq N, \xi\in\R^{n}\big\}.
  \]
 \end{enumerate}
 There is an $N\in\N_0$ and a constant $C>0$ such that
 \[
  \mathcal{R}(\{\operatorname{op}[m]:m\in\mathcal{M}\})\leq C\kappa_m\quad\text{in }\mathcal{B}(\mathscr{B}^{k}(\mathscr{A}^s)).
 \]
\end{proposition}
\begin{proof}
 This follows by the same proof as \ref{Prop:IteratedMikhlin}. One just has to use the $\mathcal{R}$-bounded versions of Mikhlin's theorem.
\end{proof}

\begin{lemma}[Lifting Property for Mixed Scales]\label{Lemma:Lifting_Property}
 Let $s,k,t_0,t_1\in\R$. Then 
 \[
  \langle D_n\rangle^{t_0}\langle D'\rangle^{t_1}\colon \mathscr{B}^{k+t_0}(\mathscr{A}^{s+t_1})\stackrel{\simeq}{\rightarrow} \mathscr{B}^{k}(\mathscr{A}^{s})
 \]
is an isomorphism of Banach spaces
\end{lemma}
\begin{proof}
 If $\mathscr{A}^\bullet$ or $\mathscr{B}^{\bullet}$ belongs to the Bessel potential scale, then it follows from the definition of Bessel potential spaces that
 \[
 \langle D'\rangle^{t_1} \colon\mathscr{A}^{s+t_1}\stackrel{\simeq}{\rightarrow} \mathscr{A}^{s}\quad\text{or} \quad\langle D_n\rangle^{t_0}\colon \mathscr{B}^{k+t_0}(\mathscr{A}^{s})\stackrel{\simeq}{\rightarrow} \mathscr{B}^{k}(\mathscr{A}^{s}),
 \]
respectively. In the other cases, this is the statement of \cite[Proposition 3.9]{Meyries_Veraar_2012}. Composing the two mapings yields the assertion.
\end{proof}

\begin{proposition}\label{Prop:BesselPotentialInMixedScales}
Let $s,k\in\R$ and $t\geq0$. Suppose that $E$ has Pisier's property $(\alpha)$ if both $\mathscr{A}$ and $\mathscr{B}$ belong to the Bessel potential scale. Then
\[
 \langle D\rangle^{t}\colon  \mathscr{B}^{k+t}(\mathscr{A}^{s})\cap \mathscr{B}^{k}(\mathscr{A}^{s+t})\stackrel{\simeq}{\rightarrow}\mathscr{B}^k(\mathscr{A}^s)
\]
is an isomorphism of Banach spaces.
\end{proposition}
\begin{proof}
    By the assumptions we imposed on $\mathscr{B}^{\bullet}(\mathscr{A}^\bullet)$, we can apply Mikhlin's theorem. We define
    \[
     f\colon \R^{n}\to\R,\;\xi\mapsto \frac{\langle\xi\rangle^{t}}{\langle\xi'\rangle^{t}+\langle\xi_n\rangle^{t}}
    \]
which satisfies the assumptions of Proposition \ref{Prop:IteratedMikhlin}. Indeed, by induction we have that $\partial_{\xi}^{\alpha}f$ is a linear combination of terms of the form
\begin{align*}
	\xi^{\beta}\langle\xi\rangle^{t-i'-i_n}\langle\xi'\rangle^{(t-2)j'-k'}\langle\xi_n\rangle^{(t-2)j_n-k_n}(\langle\xi'\rangle^t+\langle\xi_n\rangle^t)^{-1-j'-j_n}
\end{align*}
for some $\beta\in\N_0^n$, $i',i_n,j',j_n,k',k_n\in\N_0$ such that $\alpha_n=i_n+2j_n+k_n-\beta_n$ and $|\alpha'|=i'+2j'+k'-|\beta'|$. But for such a term we have that
\begin{align*}
	&\quad|\langle\xi'\rangle^{|\alpha'|}\langle\xi_n\rangle^{\alpha_n}\xi^{\beta}\langle\xi\rangle^{t-i'-i_n}\langle\xi'\rangle^{(t-2)j'-k'}\langle\xi_n\rangle^{(t-2)j_n-k_n}(\langle\xi'\rangle^t+\langle\xi_n\rangle^t)^{-1-j'-j_n}|\\
	&= |f(\xi)\langle\xi'\rangle^{|\alpha'|}\langle\xi_n\rangle^{\alpha_n}\xi^{\beta}\langle\xi\rangle^{-i'-i_n}\langle\xi'\rangle^{(t-2)j'-k'}\langle\xi_n\rangle^{(t-2)j_n-k_n}(\langle\xi'\rangle^t+\langle\xi_n\rangle^t)^{-j'-j_n}|\\
	&\leq f(\xi)[\langle\xi'\rangle^{|\alpha'|+|\beta'|-i'-2j'-k'}\langle\xi_n\rangle^{\alpha_n+\beta_n-i_n-2j_n-k_n}][\langle\xi'\rangle^{tj'}\langle\xi_n\rangle^{tj_n}(\langle\xi'\rangle^t+\langle\xi_n\rangle^t)^{-j'-j_n}]\\
	&\leq f(\xi)\\
	&\leq 1.
\end{align*}
This shows that $f$ satisfies the assumptions of Proposition \ref{Prop:IteratedMikhlin}. Therefore, we obtain
\begin{align*}
 \|\langle D\rangle^{t}u\|_{\mathscr{B}^k(\mathscr{A}^s)}&=\|f(D)(\langle D'\rangle^{t}+\langle D_n\rangle^{t})u \|_{\mathscr{B}^k(\mathscr{A}^s)}\lesssim  \|(\langle D'\rangle^{t}+\langle D_n\rangle^{t})u \|_{\mathscr{B}^k(\mathscr{A}^s)}\\
 &\lesssim\max\{\|u \|_{\mathscr{B}^{k+t}(\mathscr{A}^{s})},\|u \|_{\mathscr{B}^{k}(\mathscr{A}^{s+t})}\}
\end{align*}
as well as
\begin{align*}
\max\{\|u \|_{\mathscr{B}^{k+t}(\mathscr{A}^{s})},\|u \|_{\mathscr{B}^{k}(\mathscr{A}^{s+t})}\} &\leq  \|\langle D_n\rangle^t u \|_{\mathscr{B}^{k}(\mathscr{A}^{s})}+\|\langle D'\rangle^t u \|_{\mathscr{B}^{k}(\mathscr{A}^{s})}\\
&= \bigg\|\frac{\langle D'\rangle^{t}}{\langle D\rangle^{t}}\langle D\rangle^{t}u \bigg\|_{\mathscr{B}^k(\mathscr{A}^s)}+\bigg\|\frac{\langle D_n\rangle^{t}}{\langle D\rangle^{t}}\langle D\rangle^{t}u \bigg\|_{\mathscr{B}^k(\mathscr{A}^s)}\\
 &\lesssim \|\langle D\rangle^{t}u \|_{\mathscr{B}^k(\mathscr{A}^s)}.
\end{align*}
This proves the assertion.
\end{proof}

\begin{proposition}\label{Prop:Pseudo_Iterated_Mapping_Properties}
 Let $s,k,d\in\R$. Let
 \[
  S^d_{\mathscr{B},\mathscr{A}}(\R^n,\mathcal{B}(E)):=\begin{cases}
                                                     S^d(\R^n,\mathcal{B}(E))&\text{ if neither }\mathscr{A}\text{ nor }\mathscr{B}\text{ stands for the Bessel potential scale},\\
                                                     S^d_{\mathcal{R}}(\R^n,\mathcal{B}(E))&\text{ otherwise}.
                                                    \end{cases}
 \]
Suppose that $E$ has Pisier's property $(\alpha)$ if both $\mathscr{A}$ and $\mathscr{B}$ belong to the Bessel potential scale
 \begin{enumerate}[(a)]
  \item If $d\leq 0$, then
    \[
      S^d_{\mathscr{B},\mathscr{A}}(\R^n,\mathcal{B}(E)) \times \mathscr{B}^k(\mathscr{A}^s) \to \mathscr{B}^{k-d}(\mathscr{A}^s)\cap \mathscr{B}^k(\mathscr{A}^{s-d}),\;(p,u)\mapsto \operatorname{op}[p] u
    \]
 is bilinear and continuous.
 \item If $d\geq0$, then 
     \[
      S^d_{\mathscr{B},\mathscr{A}}(\R^n,\mathcal{B}(E)) \times(\mathscr{B}^{k+d}(\mathscr{A}^s)\cap \mathscr{B}^k(\mathscr{A}^{s+d})) \to \mathscr{B}^k(\mathscr{A}^s),\;(p,u)\mapsto \operatorname{op}[p] u
    \]
    is bilinear and continuous.
 \end{enumerate}
\begin{proof}
 By writing $p(\xi)=\frac{p(\xi)}{\langle\xi\rangle^d}\langle\xi\rangle^d$ and using Proposition \ref{Prop:BesselPotentialInMixedScales} we only have to treat the case $d=0$. But this case is included in the iterated version of Mikhlin's theorem, Proposition \ref{Prop:IteratedMikhlin}. Indeed, for a symbol $p\in S^0(\R^n,\mathcal{B}(E))$ we have
 \[
  \sup_{\xi\in\R^n\atop \alpha\in\N_0^n,|\alpha|_1\leq k} \langle\xi'\rangle^{|\alpha'|}\langle\xi_n\rangle^{|\alpha_n|}\|\partial_{\xi}^{\alpha}p(\xi)\|_{\mathcal{B}(E)}\leq   \sup_{\xi\in\R^n\atop \alpha\in\N_0^n,|\alpha|\leq k} \langle\xi\rangle^{|\alpha|}\|\partial_{\xi}^{\alpha}p(\xi)\|_{\mathcal{B}(E)}<\infty
 \]
for all $k\in\N_0$. If $p\in S^0_{\mathcal{R}}(\R^n,\mathcal{B}(E))$ we can use Kahane's contraction principle in order to obtain
\begin{align*}
 &\mathcal{R}\big\{\langle\xi'\rangle^{|\alpha'|}\langle\xi_n\rangle^{|\alpha_n|}\partial_{\xi}^{\alpha}p(\xi): \alpha\in\N_0^n,|\alpha|\leq k,\xi\in\R^n\big\}\\
&\leq  \mathcal{R}\big\{\langle\xi\rangle^{|\alpha|}\partial_{\xi}^{\alpha}p(\xi): \alpha\in\N_0^n,|\alpha|\leq k,\xi\in\R^n\big\}.
\end{align*}
\end{proof}

\end{proposition}


\section{Poisson Operators in Mixed Scales}\label{Section:Poisson}

Consider equation \eqref{EllipticBVP} with $f=0$, i.e. 
\begin{align}
\begin{aligned}\label{Eq:EllipticBVPf=0}
 \lambda u -A(D)u&=0\quad\;\text{in }\R^n_+,\\
 B_j(D)u&=g_j\quad\text{on }\R^{n-1}.
 \end{aligned}
\end{align}
Recall that we always assume that the ellipticity condition and the Lopatinskii-Shapiro condition are satisfied in the sector $\Sigma_{\phi'}\setminus\{0\}$ with $\phi'\in(0,\pi)$ and that $\phi\in(0,\phi')$. The solution operators of \eqref{Eq:EllipticBVPf=0} which map the boundary data $g=(g_1,\ldots,g_m)$ to the solution $u$ are called Poisson operators. This notion comes from the Boutet de Monvel calculus, where Poisson operators are part of the so-called singular Green operator matrices. These matrices were introduced to extend the idea of pseudo-differential operators to boundary value problems. They allow for a unified treatment of boundary value problems and their solution operators, since both of them are contained in the algebra of singular Green operator matrices. In this work however, we do not need this theory in full generality. Instead, we just focus on Poisson operators.\\
We will use a solution formula for the Poisson operator corresponding to \eqref{Eq:EllipticBVPf=0} which was derived in the classical work \cite[Proposition 6.2]{DHP03} by Denk, Hieber and Pr\"uss. In order to derive this formula, a Fourier transform in the tangential directions of \eqref{Eq:EllipticBVPf=0} is applied. This yields a linear ordinary differential equation of order $2m$ at each point in the frequency space. This ordinary differential equation is then transformed to a linear first order system which can easily be solved by an exponential function if one knows the values of $$U(\xi',0):=(\mathscr{F}'u(\xi',0),\partial_n\mathscr{F}'u(\xi',0),\ldots,\partial_n^{2m-1}\mathscr{F}'u(\xi',0)).$$ The Lopatinskii-Shapiro condition ensures that those vectors $U(\xi',0)$ which yield a stable solution can be uniquely determined from $(\mathscr{F}'g_1,\ldots,\mathscr{F}'g_m)$. The operator which gives this solution is denoted by
\[
	M(\xi',\lambda)\colon E^m\to E^{2m},\,(\mathscr{F}'g_1(\xi'),\ldots,\mathscr{F}'g_m(\xi'))\mapsto U(\xi',0).
\]
Now one just has to take the inverse Fourier transform of the solution and a projection to the first component. The latter is necessary to come back from the solution of the first order system to the solution of the higher order equation.\\
This would already be enough to derive a good solution formula. However, in \cite{DHP03} an additional rescaling was introduced so that compactness arguments can be applied. More precisely, the variables $\rho(\xi',\lambda)=\langle \xi',|\lambda|^{1/2m}\rangle=(1+|\xi'|^2+|\lambda|^{1/m})^{1/2}$, $b=\xi'/\rho$ and $\sigma=\lambda/\rho^{2m}$ are introduced in the Fourier image. The solution formula is then written in terms of $(\rho,b,\sigma)$ instead of $(\xi',\lambda)$. For this reformulation it is crucial that the operators $A,B_1,\ldots,B_m$ are homogeneous, i.e. that there are no lower order terms. And even though this rescaling makes the formulas more involved, the compactness arguments which can be used as a consequence are very useful.\\
After carrying out all these steps, the solution can be represented by
\[
 u(x)=\operatorname{pr}_1[\operatorname{Poi}(\lambda)g](x)
\]
where
\begin{itemize}
 \item $\operatorname{pr}_1\colon E^{2m}\to E,\,w\mapsto w_1$ is the projection onto the first component,
 \item $g=(g_1,\ldots,g_m)^T$ and the operator $\operatorname{Poi}(\lambda)\colon E^m\to E^{2m}$ is given by 
 \begin{align}\label{Eq:FormulaPoisson}
 	\begin{aligned}
 		\left[\operatorname{Poi}(\lambda)g\right](x):=\big[(\mathscr{F}')^{-1} e^{i\rho A_0(b,\sigma)x_n}M(b,\sigma)\hat{g}_{\rho}\big](x').
 	\end{aligned}
 \end{align}
 \item $\mathscr{F}'$ is the Fourier transform along $\R^{n-1}$, i.e. in tangential direction,
 \item $A_0$ is a smooth function with values in $\mathcal{B}(E^{2m},E^{2m})$ which one obtains from $\lambda-A(\xi',D_n)$ after a reduction to a first order system,
 \item $M$ is a smooth function with values in $\mathcal{B}(E^{m},E^{2m})$ which maps the values of the boundary operators applied to the stable solution $v$ to the vector containing its derivatives at $x_n=0$ up to the order $2m-1$, i.e. $$(B_1(D)v(0),\ldots, B_m(D)v(0))^T\mapsto (v(0),\partial_n v(0),\ldots, \partial_n^{2m-1} v(0))^T,$$
 \item $\rho$ is a positive parameter that can be chosen in different ways in dependence of $\xi'$ and $\lambda$. In our case, it will be given by $\rho(\xi',\lambda)=\langle \xi',|\lambda|^{1/2m}\rangle=(1+|\xi'|^2+|\lambda|^{1/m})^{1/2}$,
 \item $b=\xi'/\rho$, $\sigma=\lambda/\rho^{2m}$ and $\hat{g}_{\rho}=((\mathscr{F}'g_1)/\rho^{m_1},\ldots,(\mathscr{F}'g_m)/\rho^{m_m})^T$,
\end{itemize}
Again, we want to emphasize that $b$, $\sigma$, and $\rho$ depend on $\xi'$ and $\lambda$. We only neglect this dependence in the notation for the sake of readability.\\
Another operator that we will use later is the spectral projection $\ms{P}_{-}$ of the matrix $A_0$ to the part of the spectrum that lies above the real line. This spectral projection has the property that $\ms{P}_-(b,\sigma)M(b,\sigma)=M(b,\sigma)$.\\
For our purposes, we will rewrite the above representation in the following way: For $j=1,\ldots,m$ we write
  \[
   M_{\rho,j}(b,\sigma)\hat{g}_j := M(b,\sigma)\frac{\hat{g}_j\otimes e_j}{\rho^{m_j}},
  \]
  where $\hat{g}_j\otimes e_j$ denotes the $m$-tuple whose $j$-th component equals to  $\hat{g}_j$ and whose other components all equal to $0$, as well as
  \[
   [\operatorname{Poi}_j(\lambda)g_j](x):=\big[(\mathscr{F}')^{-1} e^{i\rho A_0(b,\sigma)x_n}M_{\rho,j}(b,\sigma)\hat{g}_j\big](x')
  \]
so that we obtain
\begin{align}\label{Boutet:Equation:SolutionFormulaPoisson}
  u= \operatorname{pr}_1\operatorname{Poi}(\lambda)g=\operatorname{pr}_1\sum_{j=1}^m\operatorname{Poi}_j(\lambda)g_j.
\end{align}
\begin{remark}
	If we look at the formula \eqref{Eq:FormulaPoisson} we can already see that the solution operator is actually just an exponential function in normal direction. As such, it should be arbitrarily smooth. Of course, one has to think about with respect to which topology in tangential direction this smoothness should be understood. It is the aim of this section to analyze this carefully. We treat \eqref{Eq:FormulaPoisson} as a function of $x_n$ with values in the space of pseudo-differential operators in tangential direction. Since \eqref{Eq:FormulaPoisson} is exponentially decaying in $\xi'$ if $x_n>0$, the pseudo-differential operators will have order $-\infty$, i.e. they are smoothing. Hence, the solutions will also be arbitrarily smooth, no matter how rough the boundary data is. However, the exponential decay becomes slower as one approaches $x_n=0$. This will lead to singularities if \eqref{Eq:FormulaPoisson} is considered as a function of $x_n$ with values in the space of pseudo-differential operators of a fixed order. In the following, we will study how strong these singularities are depending on the regularity in normal and tangential directions in which the singularities are studied. The answer will be given in Theorem~\ref{Thm:MainThm1}. Therein, one may choose the regularity in normal direction $k$, the integrability in normal direction $p$, the regularity in tangential direction $t$ of the solution and the regularity $s$ of the boundary data. Then the parameter $r$ in the relation $r-p[t+k-m_j-s]_+>-1$ gives a description of the singularity at the boundary, since it is the power of the power weight which one has to add to the solution space such that the Poisson operator is a well-defined continuous operator between the spaces one has chosen.
\end{remark}
In the following, we oftentimes substitute $\mu=\lambda^{1/2m}$ for homogeneity reasons. If $\lambda$ is above the real line, then we take $\mu$ to be the first of these roots, and if $\lambda$ is below the real line, we take $\mu$ to be the last of these roots. If $\lambda>0$, then we just take the ordinary positive root.

\begin{definition}
 A domain $\mathcal{O}$ is called plump with parameters $R>0$ and $\delta\in(0,1]$ if for all $x\in\mathcal{O}$ and all $r\in(0,R]$ there exists a $z\in \mathcal{O}$ such that
 \[
  B(z,\delta r)\subset B(x,r)\cap \mathcal{O}
 \]

\end{definition}

\begin{lemma}\label{ASymbolOrdnung0}
 Let $E_0,E_1$ be a Banach spaces. As described above, we take $\mu=\lambda^{1/2m}$ so that $\rho(\xi',\mu)=\langle\xi',\mu\rangle$, $b=\xi'/\rho$ and $\sigma=\mu^{2m}/\rho^{2m}$. Let $U\subset\R^{n-1}\times\Sigma_{\phi}$ be a plump and bounded environment of the range of $(b,\sigma)$. Then the mapping
 \[
  BUC^{\infty}(U,\mathcal{B}(E_0,E_1))\to S^0_{\mathcal{R}}(\R^{n-1}\times\Sigma_{\phi/2m},\mathcal{B}(E_0,E_1)), A \mapsto A\circ(b,\sigma)
 \]
 is well-defined and continuous.
\end{lemma}
\begin{proof}
 A similar proof was carried out in \cite[Proposition 4.21]{Hummel_Lindemulder_2019}. We combine this proof with \cite[Theorem 8.5.21]{HvNVW_2017} in order to obtain the $\mathcal{R}$-bounded version.\\
 Let $A\in BUC^{\infty}(U,\mathcal{B}(E_0,E_1))$. By induction on $|\alpha'|+|\gamma|$ we show that $D_{\xi'}^{\alpha'}D_{\mu}^\gamma (A\circ(b,\sigma))$ is a linear combination of terms of the form $(D_{\xi'}^{\tilde{\alpha}'}D_{\mu}^{\tilde{\gamma}}A)\circ(b,\sigma)\cdot f$ with $f\in S^{-|\alpha'|-|\gamma|}_{\mathcal{R}}(\R^{n-1}\times\Sigma_{\phi/2m})$, $\tilde{\alpha}'\in\N_0^{n-1}$ and $\tilde{\gamma}\in\N_0$. It follows from \cite[Theorem 8.5.21]{HvNVW_2017} that this is true for $|\alpha'|+|\gamma|=0$. So let $j\in\{1,\ldots,n-1\}$. By the induction hypothesis, we have that $D_{\xi'}^{\alpha'}D_{\mu}^\gamma (A\circ(b,\sigma))$ is a linear combination of terms of the form $(D_{\xi'}^{\tilde{\alpha}'}D_{\mu}^{\tilde{\gamma}}A)\circ(b,\sigma)\cdot f$ with $f\in S^{-|\alpha'|-|\gamma|}_{\mathcal{R}}(\R^{n-1}\times\Sigma_{\phi/2m})$, $\tilde{\alpha}'\in\N_0^{n-1}$ and $\tilde{\gamma}\in\N_0$. Hence, for $D_{\xi_j}D_{\xi'}^{\alpha'}D_{\mu}^{\gamma}$ it suffices to treat the summands separately, i.e. we consider $D_{\xi_j}((D_{\xi'}^{\tilde{\alpha}'}D_{\mu}^{\tilde{\gamma}}A)\circ(b,\sigma)\cdot f)$. By the product rule and the chain rule, we have
 \begin{align*}
  &\quad\;D_{\xi_j}((D_{\xi'}^{\tilde{\alpha}'}D_{\mu}^{\tilde{\gamma}}A)\circ(b,\sigma)\cdot f)\\
  &=(D_{\xi'}^{\tilde{\alpha}'}D_{\mu}^{\tilde{\gamma}}A)\circ(b,\sigma))(D_jf)+\bigg(\sum_{l=2}^{n}D_{\xi_j}( \tfrac{\xi'_l}{\rho})\cdot f \cdot [(D_l D_{\xi'}^{\widetilde{\alpha}'}D_{\mu}^{\tilde{\gamma}} A)\circ(b,\sigma)]\bigg)\\
  &\quad+D_{\xi_j}( \tfrac{\mu^{2m}}{\rho^{2m}})\cdot f \cdot [(D_l D_{\xi'}^{\widetilde{\alpha}'}D_{\mu}^{\tilde{\gamma}} A)\circ(b,\sigma)]
 \end{align*}
 By the induction hypothesis and Remark \ref{Rem:SymbolClasses} \eqref{Rem:SymbolClasses:Product} and \eqref{Rem:SymbolClasses:Derivative} we have that $$(D_{\xi_j}f),(D_{\xi_j}\frac{\xi_1'}{\rho})f,\ldots,(D_{\xi_j} \frac{\xi'_{n-1}}{\rho})f,(D_{\xi_j} \tfrac{\mu^{2m}}{\rho^{2m}}) f\in S^{-|\alpha'|-|\gamma|-1}_{\mathcal{R}}(\R^{n-1}\times\Sigma_{\phi/2m}).$$
 The same computation for $D_{\mu_1}$ and $D_{\mu_2}$ instead of $D_{\xi_j}$ also shows the desired behavior and hence, the induction is finished.\\
Now we use \cite[Theorem 8.5.21]{HvNVW_2017} again: Since $U$ is plump we have that $A$ and all its derivatives have an $\mathcal{R}$-bounded range on $U$. Therefore, the terms $(D_{\xi'}^{\tilde{\alpha}'}D_{\mu}^{\tilde{\gamma}}A)\circ(b,\sigma)\cdot f$ from above satisfy
\[
	\mathcal{R}_{\mathcal{B}(E_0,E_1)}\{\langle\xi',\mu\rangle^{|\alpha'|+|\gamma|}(D_{\xi'}^{\tilde{\alpha}'}D_{\mu}^{\tilde{\gamma}}A)\circ(b,\sigma)\cdot f(\xi',\mu):(\xi',\mu)\in\R^{n-1}\times\Sigma_{\phi/2m}\}
\]
which shows the assertion.
\end{proof}

\begin{corollary}\phantomsection \label{Cor:R_bounded_Exponential_Weight}
\begin{enumerate}[(a)]
\item \label{Cor:R_bounded_Exponential_Weight_Continuity}There is a constant $c>0$ such that the mapping
\[
 \R_+\to S^{0}_{\mathcal{R}}(\R^{n-1}\times\Sigma_{\phi/2m};\mathcal{B}(E^{2m})), y\mapsto[(\xi',\mu)\mapsto e^{cy}e^{iA_0(b,\sigma)y}\mathscr{P}_{-}(b,\sigma)]
\]
is bounded and and uniformly continuous.
 \item \label{Cor:R_bounded_Exponential_Weight_R_Bounds}There are constants $C,c>0$ such that
 \[
  \mathcal{R}(\{e^{c\rho x_n}e^{i\rho A_0(b,\sigma)x_n}\mathscr{P}_-(b,\sigma):(\xi',\mu)\in\R^{n-1}\times\Sigma_{\phi/2m}\})<C
 \]
 for all $x_n\geq0$
 \end{enumerate}
\end{corollary}
\begin{proof}
\begin{enumerate}[(a)]
\item By Lemma \ref{ASymbolOrdnung0} it suffices to show that there is a plump environment $U$ of the range $(b,\sigma)$ such that
\[
 \R_+\to BUC^{\infty}(U,\mathcal{B}(E^{2m})), y\mapsto [(\xi',\mu)\mapsto e^{cy}e^{iA_0(\xi,\mu)y}\mathscr{P}_{-}(\xi,\mu)]
\]
is bounded and continuous. We can for example take
\[
 U=\big\{\big(\tfrac{\theta \xi'}{\rho},\tfrac{\theta\mu^{2m}}{\rho^{2m}}\big): \xi'\in\R^{n-1},\,\mu\in\Sigma_{\phi/2m},\,\theta\in (\tfrac{1}{2},2)\big\}.
\]
Obviously, this set contains the range of $(b,\sigma)$ and it is smooth and relatively compact. By this compactness, it follows as in \cite[Section 6]{DHP03} (mainly because of the spectral gap (6.11)) that there is a constant $c>0$ such that
\[
 \sup_{y\geq0,\, (\xi',\mu)\in U}\| e^{2cy}e^{i A_0(\xi',\mu)y}\mathscr{P}_-(\xi',\mu) \|_{\mathcal{B}(E^{2m})}<\infty.
\]
Now, we show by induction on $|\alpha'|+|\gamma|$ that $\partial_{\xi'}^{\alpha'}\partial_{\mu}^{\gamma}e^{i A_0(\xi',\mu)y}\mathscr{P}_-(\xi',\mu)$ is a linear combination of terms of the form
\begin{align}\label{eq:derivative_linear_combination}
f(\xi',\mu) e^{i A_0(\xi',\mu)y}\mathscr{P}_-(\xi',\mu) g(\xi',\mu)y^p
\end{align}
where $f,g\colon\R^{n-1}\times\Sigma_{\phi/2m}\to\mathcal{B}(E^{2m})$ are holomorphic and $p\in\N_0$. Obviously this is true for $|\alpha'|+|\gamma|=0$. For the induction step, we can directly use the induction hypothesis and consider a term of the form \eqref{eq:derivative_linear_combination}. Since
\begin{align*}
 e^{i A_0(\xi',\mu)y}\mathscr{P}_-(\xi',\mu)&=e^{i A_0(\xi',\mu)y}\mathscr{P}_-(\xi',\mu)^2\\&=\mathscr{P}_-(\xi',\mu)e^{i A_0(\xi',\mu)y}\mathscr{P}_-(\xi',\mu)=\mathscr{P}_-(\xi',\mu)e^{i A_0(\xi',\mu)y}
\end{align*}
one can directly verify that the derivatives $\partial_{\xi_j}\partial_{\xi'}^{\alpha'}\partial_{\mu}^{\gamma}e^{i A_0(\xi',\mu)y}\mathscr{P}_-(\xi',\mu)$ $(j=1,\ldots,n-1)$ and $\partial_{\mu_i}\partial_{\xi'}^{\alpha'}\partial_{\mu}^{\gamma}e^{i A_0(\xi',\mu)y}\mathscr{P}_-(\xi',\mu)$ $(i\in\{1,2\})$ are again a linear combination of terms of the form \eqref{eq:derivative_linear_combination}. But for such a term we have
\begin{align*}
 \| f e^{i A_0y}\mathscr{P}_- gy^p \|_{BUC(U,\mathcal{B}(E^{2m}))}\leq C e^{-cy}
\end{align*}
for some constant $C>0$ and all $y\geq0$. This shows that
\[
 \big([(\xi',\mu)\mapsto e^{cy}e^{iA_0(\xi,\mu)y}\mathscr{P}_{-}(\xi,\mu)]\big)_{y\geq0}
\]
satisfies the assumptions of Lemma \ref{ASymbolOrdnung0} uniformly in $y$ so that the boundedness follows. The continuity follows from applying the same argument to $$(e^{iA_0(\xi,\lambda)h}-\operatorname{id}_{E^{2m}})e^{iA_0(\xi,\lambda)y}\mathscr{P}_-(\xi',\mu)$$ for small $|h|$, $h\in\R$. Note that $(\xi,\lambda)$ only runs through a relatively compact set again so that $e^{iA_0(\xi,\lambda)h}-\operatorname{id}_{E^{2m}}\to 0$ in $BUC^{\infty}(U)$ as $h\to\infty$.
\item This follows from the first part by substituting $y=\rho x_n$.
\end{enumerate}
\end{proof}

\begin{lemma}\label{Chapter3::Lemma::PoissonSymbolDerivative}
 Let $n_1,n_2\in\R$ and $c>0$. Moreover, let $f_0\in S^{n_1}_{\mathcal{R}}(\R^{n-1}\times\Sigma_{\phi/2m},\mathcal{B}(E^{2m}))$ and $g\in S_{\mathcal{R}}^{n_2}(\R^{n-1}\times\Sigma_{\phi/2m},\mathcal{B}(E,E^{2m}))$. Then for all $\alpha\in\N_0^{n+1}$ we have that  $\partial_{\xi',\mu}^{\alpha}f_0e^{c\rho x_n+i\rho A_0(b,\sigma)x_n}\mathscr{P}_-(b,\sigma)g_0$ is a linear combination of terms of the form $$f_{\alpha}e^{c\rho x_n+i\rho A_0(b,\sigma)x_n}\mathscr{P}_-(b,\sigma)g_{\alpha}x_n^k$$ where $f_{\alpha}\in S_{\mathcal{R}}^{n_1-d_1}(\R^{n-1}\times\Sigma_{\phi/2m},\mathcal{B}(E^{2m}))$, $g_{\alpha}\in S^{n_2-d_2}_{\mathcal{R}}(\R^{n-1}\times\Sigma_{\phi/2m},\mathcal{B}(E,E^{2m}))$ and $k+d_1+d_2=|\alpha|$.
\end{lemma}
\begin{proof}
 This can be shown by induction on $|\alpha|$. Using Lemma \ref{ASymbolOrdnung0}, the proof of \cite[Lemma 4.22]{Hummel_Lindemulder_2019} carries over to our setting.
\end{proof}

\begin{proposition}\label{Chapter3::Proposition::PoissonSymbolOrderSingularity}
 Let $\zeta,\delta\geq0$, $k\in\N_0$ and $\vartheta(c,x_n,\mu)=x_n^{-\delta}|\mu|^{-\zeta} e^{-c|\mu|x_n}$ for $c,x_n,\mu\in\R_+$.
 \begin{enumerate}[(a)]
  \item \label{PSOS::Part1} For all $l\in\N_0$ there are constants $C,c>0$ such that
\begin{align}\label{Eq:Poisson_R_Bounds}
 \| (\xi',\mu)\mapsto D_{x_n}^ke^{c\rho x_n}e^{i\rho A_0(b,\sigma)x_n}M_j(b,\sigma)\tfrac{1}{\rho^{m_j}}\|^{(k+\zeta-m_j-\delta,\vartheta(c/2,x_n,\cdot\,))}_{l,\mathcal{R}}<C
\end{align}
for all $x_n\in\R_+$.
\item \label{PSOS::Part2} The mapping
 \[
  \R_+\to S^{k+\zeta-m_j-\delta}_{\mathcal{R}}(\R^{n-1}\times\Sigma_{\phi/2m};\mathcal{B}(E,E^{2m})), x_n\mapsto [(\xi',\mu)\mapsto x_n^{\delta} e^{c\rho x_n}D_{x_n}^ke^{i\rho A_0(b,\sigma)x_n}M_j(b,\sigma)\tfrac{1}{\rho^{m_j}}]
 \]
is continuous.
\item \label{PSOS::Part3} If $f\in BUC(\R_+,\C)$ with $f(0)=0$, then
 \[
   \overline{\R_+}\to S^{k+\zeta-m_j-\delta}_{\mathcal{R}}(\R^{n-1}\times\Sigma_{\phi/2m};\mathcal{B}(E,E^{2m})), x_n\mapsto [(\xi',\mu)\mapsto f(x_n)x_n^{\delta} e^{c\rho x_n}D_{x_n}^ke^{i\rho A_0(b,\sigma)x_n}M_j(b,\sigma)\tfrac{1}{\rho^{m_j}}]
 \]
 is uniformly continuous.
 \item  \label{PSOS::Part4} If $\epsilon\in(0,\delta]$, then
 \[
   \overline{\R_+}\to S^{k+\zeta+\epsilon-m_j-\delta}_{\mathcal{R}}(\R^{n-1}\times\Sigma_{\phi/2m};\mathcal{B}(E,E^{2m})), x_n\mapsto [(\xi',\mu)\mapsto x_n^{\delta} e^{c\rho x_n}D_{x_n}^ke^{i\rho A_0(b,\sigma)x_n}M_j(b,\sigma)\tfrac{1}{\rho^{m_j}}]
 \]
 is uniformly continuous.
 \end{enumerate} 
\end{proposition}
\begin{proof}
\begin{enumerate}[(a)]
  \item By Lemma \ref{Chapter3::Lemma::PoissonSymbolDerivative}, we have that $D_{\xi'}^{\alpha'}D_{\mu}^{\gamma} e^{i\rho x_n}e^{i\rho A_0(b,\sigma)x_n}M_{j}(b,\sigma)\frac{1}{\rho^{m_j}}$ is a linear combination of terms of the form 
  \[
   f_{\alpha',\gamma}e^{c\rho x_n+i\rho A_0(b,\sigma)x_n}\ms{P}_-(b,\sigma)g_{\alpha',\gamma}x_n^{p} 
  \]
where $f_{\alpha',\gamma}\in S^{-d_1}(\R^{n-1}\times\R_+,\mc{B }( E^{2m},E^{2m}))$, $g_{\alpha',\gamma}\in S^{-m_{j}-d_2}(\R^{n-1}\times\R_+,\mc{B}(E,E^{2m}))$ and $p+d_1+d_2=|\alpha'|+|\gamma|$. But for such a term, we have that

\begin{align*}
 &\quad\; \mathcal{R}\big(\{\vartheta(\tfrac{c}{2},x_n,\mu)^{-1}\rho^{-k-\zeta+m_j+\delta +|\alpha'|+|\gamma|}D_{x_n}^kf_{\alpha',l}(\xi',\mu)e^{c\rho x_n}e^{i\rho A_0(b,\sigma)x_n}\\
 &\qquad\qquad\qquad\qquad\ms{P}_-(b,\sigma)g_{\alpha',l}(\xi',\mu)x_n^p : (\xi',\mu)\in\R^{n-1}\times\Sigma_{\phi/2m}\}\big) \\
  &\leq C  \sum_{\tilde{k}=0}^k \mathcal{R}\big(\{\vartheta(\tfrac{c}{2},x_n,\mu)^{-1}\rho^{-k-\zeta+m_j+\delta +|\alpha'|+|\gamma|}f_{\alpha',l}(\xi',\mu)e^{c\rho x_n}e^{i\rho A_0(b,\sigma)x_n}\\
  &\qquad\qquad\qquad\qquad\ms{P}_-(b,\sigma)[c\rho+i\rho A_0(b,\sigma)]^{k-\tilde{k}}g_{\alpha',l}(\xi',\mu)x_n^{[p-\tilde{k}]_+}: (\xi',\mu)\in\R^{n-1}\times\Sigma_{\phi/2m}\}\big)\\
  &\leq C  \sum_{\tilde{k}=0}^k\mathcal{R}\big(\{\vartheta(\tfrac{c}{2},x_n,\mu)^{-1} \rho^{-\zeta+\delta-\tilde{k}-d_1-d_2+|\alpha'|+|\gamma|}e^{-c\rho x_n}x_n^{[p-\tilde{k}]_+}: (\xi',\mu)\in\R^{n-1}\times\Sigma_{\phi/2m}\}\big)\\
  &\leq C  \mathcal{R}\big(\{\vartheta(\tfrac{c}{2},x_n,\mu)^{-1}x_n^{-\delta}|\mu|^{-\zeta}\rho^{-d_1-d_2-p+|\alpha'|+|\gamma|}e^{-\tfrac{c}{2}\rho x_n}: (\xi',\mu)\in\R^{n-1}\times\Sigma_{\phi/2m}\}\big)\\
  &\leq C   
\end{align*}
From the second to the third line we used Lemma \ref{ASymbolOrdnung0} and Corollary \ref{Cor:R_bounded_Exponential_Weight} \eqref{Cor:R_bounded_Exponential_Weight_R_Bounds}. This gives \eqref{Eq:Poisson_R_Bounds}.
\item Again, we consider a term of the form \[
   f_{\alpha',\gamma}e^{i\rho A_0(b,\sigma)x_n}\ms{P}_-(b,\sigma)g_{\alpha',\gamma}x_n^{p} 
  \]
where $f_{\alpha',\gamma}\in S^{-d_1}(\R^{n-1}\times\R_+,\mc{B }( E^{2m},E^{2m}))$, $g_{\alpha',\gamma}\in S^{-m_{j}-d_2}(\R^{n-1}\times\R_+,\mc{B}(E,E^{2m}))$ and $p+d_1+d_2=|\alpha'|+|\gamma|$. By the same computation as in part \eqref{PSOS::Part1} we obtain
\begin{align*}
 &\quad\; \mathcal{R}\big(\{\vartheta(\tfrac{c}{2},x_n,\mu)^{-1}\rho^{-k-\zeta+m_j+\delta +|\alpha'|+|\gamma|}D_{x_n}^kf_{\alpha',l}(\xi',\mu)[e^{c\rho (x_n+h)+i\rho A_0(b,\sigma)(x_n+h)}-e^{c\rho x_n+i\rho A_0(b,\sigma)x_n}]:\\
 &\qquad\qquad\qquad\qquad\qquad\qquad\qquad\qquad\qquad\qquad\qquad\qquad(\xi',\mu)\in\R^{n-1}\times\Sigma_{\phi/2m}\})\\
  &\leq C  \mathcal{R}\big(\{\vartheta(\tfrac{c}{2},x_n,\mu)^{-1}x_n^{-\delta}|\mu|^{-\zeta}\rho^{-d_1-d_2-p+|\alpha'|+|\gamma|}[e^{c\rho h+i\rho A_0(b,\sigma)h}-\operatorname{id}_{E^{2m}}] e^{-\tfrac{3c}{4}\rho x_n}: (\xi',\mu)\in\R^{n-1}\times\Sigma_{\phi/2m}\}\big)\\
  &= C   \mathcal{R}\big(\{[e^{c\rho h+i\rho A_0(b,\sigma)h}-\operatorname{id}_{E^{2m}}] e^{-\tfrac{c}{4}\rho x_n}: (\xi',\mu)\in\R^{n-1}\times\Sigma_{\phi/2m}\}\big)\\
  &\leq C \mathcal{R}\big(\{[e^{c\rho h+i\rho A_0(b,\sigma)h}-\operatorname{id}_{E^{2m}}] e^{-\tfrac{c}{4}\rho x_n}: (\xi',\mu)\in\R^{n-1}\times\Sigma_{\phi/2m}, \rho\leq \tfrac{1}{\sqrt{h}}\}\big) \\
  &\quad+C \mathcal{R}\big(\{[e^{c\rho h+i\rho A_0(b,\sigma)h}-\operatorname{id}_{E^{2m}}] e^{-\tfrac{cx_n}{4\sqrt{h}}}: (\xi',\mu)\in\R^{n-1}\times\Sigma_{\phi/2m}, \rho\geq \tfrac{1}{\sqrt{h}}\}\big).
\end{align*}
From Corollary~\ref{Cor:R_bounded_Exponential_Weight}~\eqref{Cor:R_bounded_Exponential_Weight_Continuity} it follows that the first $\mathcal{R}$-bound tends to $0$ as $h\to0$. By Corollary~\ref{Cor:R_bounded_Exponential_Weight}~\eqref{Cor:R_bounded_Exponential_Weight_R_Bounds} it holds that 
\[
 \mathcal{R}(\{e^{c\rho h+i\rho A_0(b,\sigma)h}-\operatorname{id}_{E^{2m}}:(\xi',\mu)\in\R^{n-1}\times\Sigma_{\phi/2m}, \rho\geq \tfrac{1}{\sqrt{h}}\})<\infty
\]
and and since $x_n>0$ also the second $\mathcal{R}$-bound tends to $0$ as $h\to0$. This shows the desired continuity.
\item This follows by the same computation as in part \eqref{PSOS::Part2}. However, without $f$ there would be no continuity at $x_n=0$ as the second $\mathcal{R}$-bound
\[
 \mathcal{R}\big(\{[e^{c\rho h+i\rho A_0(b,\sigma)h}-\operatorname{id}_{E^{2m}}] e^{-\tfrac{cx_n}{2\sqrt{h}}}: (\xi',\mu)\in\R^{n-1}\times\Sigma_{\phi/2m}, \rho\geq \tfrac{1}{\sqrt{h}}\}\big)
\]
does not tend to $0$ as $h\to0$ for $x_n=0$. By adding $f$ though, we obtain the desired continuity.
\item This follows from part \eqref{PSOS::Part3} with $f(x_n)=x_n^{\epsilon}$ for $x_n$ close to $0$.
\end{enumerate}
\end{proof}

Given a topological spaces $Z_0,Z_1$ and $z\in Z_0$, we now write
\[
 \operatorname{ev}_{z}\colon C(Z_0;Z_1)\to Z_1, f\mapsto f(z)
\]
for the evaluation map at $z$.

\begin{corollary}\label{Chapter3::Corollary::PoissonEvalEstimate}
 Let $k\in\N_0$, $p_0\in(1,\infty)$, $q_0\in[1,\infty]$, $\zeta\geq0$ and $s_0,s,\tilde{t}\in\R$.
 \begin{enumerate}[(a)]
  \item There are constants $C,c>0$ such that for all $x_n>0$ and all $\lambda\in\Sigma_{\phi}$ we have the parameter-dependent estimate
 \begin{align*}
  \| [D_{x_n}^k\operatorname{Poi}_j(\lambda)& f](\cdot,x_n)\|_{\mathscr{A}^{\tilde{t}+m_j-k-\zeta,|\lambda|^{1/2m},s_0}(\R^{n-1},w,E^{2m})}\\&\leq Cx_n^{-[\tilde{t}-s]_+}|\lambda|^{-\zeta/2m}e^{-c|\lambda|^{1/2m}x_n}\|f\|_{\mathscr{A}^{s,|\lambda|^{1/2m},s_0}(\R^{n-1},w,E)}\quad(f\in\mathscr{S}(\R^{n-1},E)).
 \end{align*}
 \item\label{PEE:Part2} There is a constant $c>0$ such that for all $\lambda\in\Sigma_{\phi}$ we have that
 \[
  K(\lambda):=[x_n\mapsto x_n^{[\tilde{t}-s]_+}|\lambda|^{\frac{\zeta-[-[\tilde{t}-s]_+-m_j+k+\zeta]_+}{2m}}e^{c|\lambda|^{1/2m} x_n}\operatorname{ev}_{x_n}D_{x_n}^k\operatorname{Poi}_j(\lambda)]
 \]
 is an element of $$C_{\mathcal{R}B}\big(\R_+;\mathcal{B}(\mathscr{A}^s(\R^{n-1},w;E^{2m}),\mathscr{A}^{\tilde{t}+m_j-k-\zeta}(\R^{n-1},w;E^{2m}))\big).$$ Moreover, for all $\sigma>0$ we have that the set $\{K(\lambda):\lambda\in\Sigma_{\phi},\,|\lambda|\geq\sigma\}$ is $\mathcal{R}$-bounded in $C_{\mathcal{R}B}\big(\R_+;\mathcal{B}(\mathscr{A}^s(\R^{n-1},w;E^{2m}),\mathscr{A}^{\tilde{t}+m_j-k-\zeta}(\R^{n-1},w;E^{2m}))\big)$.
  \item Let $f\in BUC([0,\infty),\C)$ such that $f(0)=0$. There is a constant $c>0$ such that for all $\lambda\in\Sigma_{\phi}$ we have that
 \[
  K_f(\lambda):=[x_n\mapsto f(x_n)x_n^{[\tilde{t}-s]_+}|\lambda|^{\frac{\zeta-[-[\tilde{t}-s]_+-m_j+k+\zeta]_+}{2m}}e^{c|\lambda|^{1/2m} x_n}\operatorname{ev}_{x_n}D_{x_n}^k\operatorname{Poi}_j(\lambda)]
 \]
 is an element of $$BUC_{\mathcal{R}}\big(\R_+;\mathcal{B}(\mathscr{A}^s(\R^{n-1},w;E^{2m}),\mathscr{A}^{\tilde{t}+m_j-k-\zeta}(\R^{n-1},w;E^{2m}))\big).$$ Moreover, for all $\sigma>0$ we have that the set $\{K_f(\lambda):\lambda\in\Sigma_{\phi},\,|\lambda|\geq\sigma\}$ is $\mathcal{R}$-bounded in $BUC_{\mathcal{R}}\big(\R_+;\mathcal{B}(\mathscr{A}^s(\R^{n-1},w;E^{2m}),\mathscr{A}^{\tilde{t}+m_j-k-\zeta}(\R^{n-1},w;E^{2m}))\big)$.
  \item Let $\epsilon>0$ and let $K(\lambda)$ be defined as in Part \eqref{PEE:Part2}. Then $K(\lambda)$
 is an element of $$BUC_{\mathcal{R}}\big(\R_+;\mathcal{B}(\mathscr{A}^s(\R^{n-1},w;E^{2m}),\mathscr{A}^{\tilde{t}+m_j-\epsilon-k-\zeta}(\R^{n-1},w;E^{2m}))\big).$$ Moreover, for all $\sigma>0$ we have that the set $\{K(\lambda):\lambda\in\Sigma_{\phi},\,|\lambda|\geq\sigma\}$ is $\mathcal{R}$-bounded in $BUC_{\mathcal{R}}\big(\R_+;\mathcal{B}(\mathscr{A}^s(\R^{n-1},w;E^{2m}),\mathscr{A}^{\tilde{t}+m_j-\epsilon-k-\zeta}(\R^{n-1},w;E^{2m}))\big)$.
 \end{enumerate}
\end{corollary}
\begin{proof}
\begin{enumerate}[(a)]
 \item By Proposition \ref{Chapter3::Proposition::PoissonSymbolOrderSingularity} we have that 
\begin{align*}
 \big(D_{x_n}^k e^{i\rho A_0(b,\sigma)x_n}\frac{M_{j}(b,\sigma)}{\rho^{m_j}}\big)_{x_n>0}&\subset S^{\zeta+k-m_{j}-[t-s]_+}_{\mathcal{R}}(\R^{n-1}\times\Sigma_{\phi/2m};\mc{B}(E,E^{2m}))\\
 &\subset S^{\zeta+k-m_{j}-\tilde{t}+s}_{\mathcal{R}}(\R^{n-1}\times\Sigma_{\phi/2m};\mc{B}(E,E^{2m})).
\end{align*}
Therefore, it follows from \eqref{Eq:Poisson_R_Bounds} together with the mapping properties for parameter-dependent pseudo-differential operators, Proposition~\ref{Prop:PseudoMappingProperties}, that $\operatorname{ev}_{x_n}D_{x_n}^k\operatorname{Poi}_j(\lambda)$ maps $\mathscr{A}^{s,|\lambda|^{1/2m},s_0}(\R^{n-1},w,E)$ into $\mathscr{A}^{\tilde{t}+m_j-k-\zeta,|\lambda|^{1/2m},s_0}(\R^{n-1},w,E^{2m})$ with a bound on the operator norms which is given by $Cx_n^{-[\tilde{t}-s]_+}|\lambda|^{-\zeta/2m} e^{-c|\lambda|^{1/2m}x_n}$ for all $\tilde{t},s\in\R$, $x_n>0$ and all $\zeta\geq0$.
\item We use Proposition \ref{Chapter3::Proposition::PoissonSymbolOrderSingularity} \eqref{PSOS::Part1} together with Proposition \ref{Prop:PseudoMappingProperties_R_bounded}. Then we obtain
\begin{align*}
 &\quad\mathcal{R}_{\mathcal{B}(\mathscr{A}^s(\R^{n-1},w;E^{2m}),\mathscr{A}^{\tilde{t}+m_j-k-\zeta}(\R^{n-1},w;E^{2m})}\big(\{x_n^{[\tilde{t}-s]_+}|\lambda|^{\frac{\zeta-[-[\tilde{t}-s]_+-m_j+k+\zeta]_+}{2m}}e^{c|\lambda|^{1/2m}x_n}\\
   &\qquad\qquad\qquad\operatorname{ev}_{x_n}D_{x_n}^k\operatorname{Poi}_j(\lambda):\lambda\in\Sigma_{\phi},\,|\lambda|\geq\sigma\}\big)\\
 &\leq\mathcal{R}_{\mathcal{B}(\mathscr{A}^{\tilde{t}-[\tilde{t}-s]_+}(\R^{n-1},w;E^{2m}),\mathscr{A}^{\tilde{t}+m_j-k-\zeta}(\R^{n-1},w;E^{2m}))}\big(\{x_n^{[\tilde{t}-s]_+}|\lambda|^{\frac{\zeta-[-[\tilde{t}-s]_+-m_j+k+\zeta]_+}{2m}}e^{c|\lambda|^{1/2m}x_n}\\
   &\qquad\qquad\qquad\operatorname{ev}_{x_n}D_{x_n}^k\operatorname{Poi}_j(\lambda):\lambda\in\Sigma_{\phi},\,|\lambda|\geq\sigma\}\big)\\
 &\leq C_{\sigma}.
\end{align*}
This shows that
\[
 \mathcal{R}(\{[K(\lambda)](x_n)\colon x_n\geq0,\,\lambda\in\Sigma_{\phi},\,|\lambda|\geq\sigma\})<\infty
\]
in $\mathcal{B}(\mathscr{A}^s(\R^{n-1},w;E^{2m}),\mathscr{A}^{\tilde{t}+m_j-k-\zeta}(\R^{n-1},w;E^{2m}))$ it remains to show that the $K(\lambda)$ are $\mathcal{R}$-continuous. But this follows from the continuity statement in Proposition \ref{Chapter3::Proposition::PoissonSymbolOrderSingularity} \eqref{PSOS::Part2} together with Proposition \ref{Prop:PseudoMappingProperties_R_bounded}.
\item This follows as Part \eqref{PEE:Part2} but with Proposition \ref{Chapter3::Proposition::PoissonSymbolOrderSingularity} \eqref{PSOS::Part3} instead of Proposition \ref{Chapter3::Proposition::PoissonSymbolOrderSingularity} \eqref{PSOS::Part2}.
\item This follows as Part \eqref{PEE:Part2} but with Proposition \ref{Chapter3::Proposition::PoissonSymbolOrderSingularity} \eqref{PSOS::Part4} instead of Proposition \ref{Chapter3::Proposition::PoissonSymbolOrderSingularity} \eqref{PSOS::Part2}.
\end{enumerate}
\end{proof}

\begin{proposition}\label{Chapter3::Proposition::PoissonMappingPropertySobolev}
 Consider the situation of Corollary \ref{Chapter3::Corollary::PoissonEvalEstimate} and let $p\in[1,\infty)$, $r\in\R$. In order to shorten the formulas, we write $\gamma_1=r-p[\tilde{t}-s]_+$ and $\gamma_2=p\zeta-p[-[\tilde{t}-s]_+-m_j+ k+\zeta]_+$. Suppose that $\gamma_1>-1$. Then, for all $\sigma>0$ and all there is a constant $C>0$ such that for all $\lambda\in\Sigma_{\phi}$  with $|\lambda|\geq\sigma$ and all $f\in \mathscr{A}^{s}_p(\R^{n-1},w;E)$ it holds that
  \begin{align*}
   \|  \operatorname{Poi}_{j}(\lambda)f\|_{W^{k}_{p}(\R_+,|\operatorname{pr}_n|^r,\mathscr{A}^{t+m_j- k-\zeta}(\R^{n-1},w;E^{2m}))}\leq  C|\lambda|^{\frac{-1-\gamma_1-\gamma_2}{2mp}} \|f\|_{\mathscr{A}^{s}_p(\R^{n-1},w;E)}.
 \end{align*}
\end{proposition}
\begin{proof}
  We use Corollary \ref{Chapter3::Corollary::PoissonEvalEstimate} and obtain
  \begin{align*}
 &\quad \|  \operatorname{Poi}_{j}(\lambda)f\|_{W^{k}_{p}(\R_+,|\operatorname{pr}_n|^r,\mathscr{A}^{t+m_j- k-\zeta}(\R^{n-1},w;E^{2m}))}^p \\
  &= \sum_{l=0}^{k}\int_0^{\infty} \|[D_{x_n}^l\operatorname{Poi}_{j}(\lambda)f](\,\cdot\,,x_n) \|_{\mathscr{A}^{t+m_j- k-\zeta}(\R^{n-1},w;E^{2m})}^px_n^r\,dx_n\\
  &  \leq C \|f\|_{\mathscr{A}^{s}_p(\R^{n-1},w;E)}^p\sum_{l=0}^{k}|\lambda|^{\frac{p[-[t-s]_+-m_j+ l-\zeta]_+-p\zeta}{2m}}\int_0^{\infty} x_n^{\gamma_1}  e^{-cp|\lambda|^{1/2m}x_n}\,dx_n\\
  &  \leq C |\lambda|^{\frac{-\gamma_2}{2m}}\|f\|_{\mathscr{A}^{s}_p(\R^{n-1},w;E)}^p\int_0^{\infty} x_n^{\gamma_1}  e^{-c|\lambda|^{1/2m}x_n}\,dx_n\\
    &  \leq C|\lambda|^{\frac{-1-\gamma_1-\gamma_2}{2m}} \|f\|_{\mathscr{A}^{s}_p(\R^{n-1},w;E)}^p\int_{0}^{\infty} y_n^{\gamma_1}  e^{-cy_n}\,dy_n\\
      &  \leq C|\lambda|^{\frac{-1-\gamma_1-\gamma_2}{2m}} \|f\|_{\mathscr{A}^{s}_p(\R^{n-1},w;E)}^p
 \end{align*}
 for all $f\in \mathscr{A}^{s}_p(\R^{n-1},w;E)$.
 \end{proof}

\begin{proposition}\label{Chapter3::Proposition::PoissonMappingPropertySobolevRBoundedEpsilon}
 Consider the situation of Corollary \ref{Chapter3::Corollary::PoissonEvalEstimate} and let $p\in[1,\infty)$, $r\in\R$. Again we write $\gamma_1=r-p[\tilde{t}-s]_+$ as well as $\gamma_2=p\zeta-p[-[\tilde{t}-s]_+-m_j+ k+\zeta]_+$. Suppose that $\gamma_1>-1$ and take $\epsilon\in(0,1+\gamma_1)$. Then, for all $\sigma>0$ and all there is a constant $C>0$ such that for all $\epsilon\geq0$ 
  \begin{align*}
   \mathcal{R}\big(\{|\lambda|^{\frac{1+\gamma_1+\gamma_2-\epsilon}{2mp}}\operatorname{Poi}_{j}(\lambda):\lambda\in\Sigma_{\phi},\,|\lambda|\geq\sigma \}\big)\leq C,
 \end{align*}
 where the $\mathcal{R}$-bounds are taken in $\mathcal{B}(\mathscr{A}^{s}(\R^{n-1},w;E),W^{k}_{p}(\R_+,|\operatorname{pr}_n|^{r};\mathscr{A}^{\tilde{t}+m_j- k-\zeta}(\R^{n-1},w;E^{2m}))).$
\end{proposition}
\begin{proof}
 Let $(\epsilon_l)_{l\in\N}$ be a Rademacher sequence on the probability space $(\Omega,\mathcal{F},\mathbb{P})$ and let $N\in\N$, $\lambda_1,\ldots,\lambda_N\in\Sigma_{\phi}$ and $f_1,\ldots,f_N\in\mathscr{A}^s(\R^{n-1},w;E)$. Using Corollary \ref{Chapter3::Corollary::PoissonEvalEstimate} and Kahane's contraction principle we obtain
 {\allowdisplaybreaks{
  \begin{align*}
   & \;\;\;\,\left\|\sum_{l=1}^N \epsilon_l|\lambda_l|^{\frac{1+\gamma_1+\gamma_2-\epsilon}{2mp}}\operatorname{Poi}_{j}(\lambda_l) f_l \right\|_{L_p(\Omega;W^k_p(\R_+,|\operatorname{pr}_n|^{r},\mathscr{A}^{\tilde{t}+m_j- k-\zeta}(\R^{n-1},w;E^{2m})))}\\
    &\eqsim  \left\|\sum_{l=1}^N \epsilon_l|\lambda_l|^{\frac{1+\gamma_1+\gamma_2-\epsilon}{2mp}}\operatorname{Poi}_{j}(\lambda_l) f_l \right\|_{W^k_p(\R_+,|\operatorname{pr}_n|^{r};L_p(\Omega;\mathscr{A}^{\tilde{t}+m_j- k-\zeta}(\R^{n-1},w;E^{2m})))}\\
    &\eqsim  \sum_{\tilde{k}=0}^k\left\|\sum_{l=1}^N \epsilon_l|\lambda_l|^{\frac{1+\gamma_1+\gamma_2-\epsilon}{2mp}}D_{x_n}^{\tilde{k}}\operatorname{Poi}_{j}(\lambda_l) f_l \right\|_{L_p(\R_+,|\operatorname{pr}_n|^{r};L_p(\Omega;\mathscr{A}^{\tilde{t}+m_j- k-\zeta}(\R^{n-1},w;E^{2m})))}\\
    &\lesssim_{\sigma}\left\|x_n\mapsto\sum_{l=1}^N \epsilon_l|\lambda_l|^{\frac{1+\gamma_1-\epsilon}{2mp}}x_n^{-[\tilde{t}-s]_+} e^{-\frac{c}{2}|\lambda_l|^{1/2m}x_n}f_l \right\|_{L_p(\R_+,|\operatorname{pr}_n|^{r};L_p(\Omega;\mathscr{A}^{s}(\R^{n-1},w;E)))}\\
   &\lesssim\left(\int_{0}^{\infty} \max_{l=1,\ldots,N}\{|\lambda_l|^{\frac{1+\gamma_1-\epsilon}{2m}}e^{-p\frac{c}{2}|\lambda_l|^{1/2m}x_n} x_n^{\gamma_1}\}\,dx_n\right)^{1/p}\left\|\sum_{l=1}^N \epsilon_lf_l \right\|_{L_p(\Omega;\mathscr{A}^{s}(\R^{n-1},w;E))}\\
   &\lesssim\left(\int_{0}^{\infty} \max_{l=1,\ldots,N}\{e^{-p\frac{c}{3}|\lambda_l|^{1/2m}x_n} x_n^{-1+\epsilon}\}\,dx_n\right)^{1/p}\left\|\sum_{l=1}^N \epsilon_lf_l \right\|_{L_p(\Omega;\mathscr{A}^{s}(\R^{n-1},w;E))}\\
   &\leq\left(\int_{0}^{\infty} e^{-p\frac{c}{3}\sigma^{1/2m}x_n} x_n^{-1+\epsilon}\,dx_n\right)^{1/p}\left\|\sum_{l=1}^N \epsilon_lf_l \right\|_{L_p(\Omega;\mathscr{A}^{s}(\R^{n-1},w;E))}\\
   &\leq C\left\|\sum_{l=1}^N \epsilon_lf_l \right\|_{L_p(\Omega;\mathscr{A}^{s}(\R^{n-1},w;E))}
 \end{align*}}}
 for all $N\in\N$, all $\lambda_1,\ldots,\lambda_N\in\Sigma_{\phi}$ and all $f_1,\ldots,f_N\in\mathscr{A}^s(\R^{n-1},w;E)$. This is the desired estimate.
\end{proof}

\begin{remark}
 Comparing Proposition \ref{Chapter3::Proposition::PoissonMappingPropertySobolev} and Proposition \ref{Chapter3::Proposition::PoissonMappingPropertySobolevRBoundedEpsilon} one might wonder if one can omit the $\epsilon$ in Proposition \ref{Chapter3::Proposition::PoissonMappingPropertySobolevRBoundedEpsilon}. After having applied Kahane's contraction principle in the proof of Proposition \ref{Chapter3::Proposition::PoissonMappingPropertySobolevRBoundedEpsilon} it seems like the $\epsilon$ is necessary. Roughly speaking, taking $(\lambda_l)_{l\in\N}$ such that this sequence is dense in $\Sigma_{\phi}\setminus\overline{B(0,\sigma)}$ will cause $\max_{l=1,\ldots,N}\{|\lambda_l|^{\frac{\gamma_1+1}{2m}}e^{-p\frac{c}{2}|\lambda_l|^{1/2m}x_n} x_n^{\gamma_1}\}$ to have a singularity of the form $x_n^{-1}$ at $x_n=0$ if $N\to\infty$. Indeed, taking $|\lambda_l|^{1/2m}$ close to $x_n^{-1}$ yields that $|\lambda_l|^{\frac{\gamma_1+1}{2m}}e^{-p\frac{c}{2}|\lambda_l|^{1/2m}x_n} x_n^{\gamma_1}$ is close to $x_n^{-1} e^{-pc/2}$. Hence, the integral $\left(\int_{0}^{\infty} \max_{l=1,\ldots,N}\{|\lambda_l|^{\frac{\gamma_1+1}{2m}}e^{-p\frac{c}{2}|\lambda_l|^{1/2m}x_n} x_n^{\gamma_1}\}\,dx_n\right)^{1/p}$ will tend to $\infty$ as $N\to\infty$. Thus, if one wants to remove the $\epsilon$, it seems like one should not apply Kahane's contraction principle as it is applied in the proof of Proposition \ref{Chapter3::Proposition::PoissonMappingPropertySobolevRBoundedEpsilon}. This can for example be avoided under a cotype assumption on $E$ together with a restriction on $p$, as Proposition \ref{Chapter3::Proposition::PoissonMappingPropertySobolevRBounded} shows. However, there are some cases in which the $\epsilon$ can not be removed. We will show this in Proposition \ref{Prop:CounterExampleEpsilon}.
\end{remark}

\begin{proposition}\label{Chapter3::Proposition::PoissonMappingPropertySobolevRBounded}
 Consider the situation of Corollary \ref{Chapter3::Corollary::PoissonEvalEstimate} and let $r\in\R$. Suppose that $E$ has finite cotype $q_E$. Suppose that the assumptions of Proposition \ref{Prop:Contraction_Principle_Functions} hold true. Again, we define $\gamma_1=r-p[\tilde{t}-s]_+$ as well as $\gamma_2=p\zeta-p[-[\tilde{t}-s]_+-m_j+ k+\zeta]_+$. Suppose that $\gamma_1>-1$. Then for all $\sigma>0$ there is a constant $C>0$ such that
  \begin{align*}
   \mathcal{R}\big(\{|\lambda|^{\frac{1+\gamma_1+\gamma_2}{2mp}}\operatorname{Poi}_{j}(\lambda):\lambda\in\Sigma_{\phi},\,|\lambda|\geq\sigma \}\big)\leq C,
 \end{align*}
 where $\mathcal{R}$-bounds are taken in $\mathcal{B}(\mathscr{A}^{s}(\R^{n-1},w;E),W^{k}_{p}(\R_+,|\operatorname{pr}_n|^{r};\mathscr{A}^{\tilde{t}+m_j- k-\zeta}(\R^{n-1},w;E^{2m}))).$
\end{proposition}
\begin{proof}
Let $(\epsilon_l)_{l\in\N}$ be a Rademacher sequence on the probability space $(\Omega,\mathcal{F},\mathbb{P})$ and let $N\in\N$, $\lambda_1,\ldots,\lambda_N\in\Sigma_{\phi}$ and $f_1,\ldots,f_N\in\mathscr{A}^s(\R^{n-1},w;E)$. Using Corollary \ref{Chapter3::Corollary::PoissonEvalEstimate} and Proposition \ref{Prop:Contraction_Principle_Functions} we obtain
 {\allowdisplaybreaks{
  \begin{align*}
   & \;\;\;\,\left\|\sum_{l=1}^N \epsilon_l|\lambda_l|^{\frac{1+\gamma_1+\gamma_2}{2mp}}\operatorname{Poi}_{j}(\lambda_l) f_l \right\|_{L_p(\Omega;W^k_p(\R_+,|\operatorname{pr}_n|^{r},\mathscr{A}^{\tilde{t}+m_j- k-\zeta}(\R^{n-1},w;E^{2m})))}\\
    &\eqsim  \left\|\sum_{l=1}^N \epsilon_l|\lambda_l|^{\frac{1+\gamma_1+\gamma_2}{2mp}}\operatorname{Poi}_{j}(\lambda_l) f_l \right\|_{W^k_p(\R_+,|\operatorname{pr}_n|^{r};L_p(\Omega;\mathscr{A}^{\tilde{t}+m_j- k-\zeta}(\R^{n-1},w;E^{2m})))}\\
    &\eqsim  \sum_{\tilde{k}=0}^k\left\|\sum_{l=1}^N \epsilon_l|\lambda_l|^{\frac{1+\gamma_1+\gamma_2}{2mp}}D_{x_n}^{\tilde{k}}\operatorname{Poi}_{j}(\lambda_l) f_l \right\|_{L_p(\R_+,|\operatorname{pr}_n|^{r};L_p(\Omega;\mathscr{A}^{\tilde{t}+m_j- k-\zeta}(\R^{n-1},w;E^{2m})))}\\
    &\lesssim_{\sigma}\left\|x_n\mapsto\sum_{l=1}^N \epsilon_l|\lambda_l|^{\frac{1+\gamma_1}{2mp}}x_n^{-[\tilde{t}-s]_+} e^{-\frac{c}{2}|\lambda_l|^{1/2m}x_n}f_l \right\|_{L_p(\R_+,|\operatorname{pr}_n|^{r};L_p(\Omega;\mathscr{A}^{s}(\R^{n-1},w;E)))}\\
   &\lesssim\max_{l=1,\ldots,N}\left(\int_{0}^{\infty} |\lambda_l|^{\frac{1+\gamma_1}{2m}}e^{-p\frac{c}{2}|\lambda_l|^{1/2m}x_n} x_n^{\gamma_1}\,dx_n\right)^{1/p}\left\|\sum_{l=1}^N \epsilon_lf_l \right\|_{L_p(\Omega;\mathscr{A}^{s}(\R^{n-1},w;E))}\\
   &=\left(\int_{0}^{\infty} e^{-p\frac{c}{2}y} y^{\gamma_1}\,dy\right)^{1/p}\left\|\sum_{l=1}^N \epsilon_lf_l \right\|_{L_p(\Omega;\mathscr{A}^{s}(\R^{n-1},w;E))}\\
   &\leq C\left\|\sum_{l=1}^N \epsilon_lf_l \right\|_{L_p(\Omega;\mathscr{A}^{s}(\R^{n-1},w;E))}
 \end{align*}}}
 for all $N\in\N$, all $\lambda_1,\ldots,\lambda_N\in\Sigma_{\phi}$ and all $f_1,\ldots,f_N\in\mathscr{A}^s(\R^{n-1},w;E)$. This is the desired estimate.
\end{proof}

Let us now see what can happen if the cotype assumption is not satisfied.
\begin{lemma}\label{Lemma:AuxiliarySpace}
 Let $\tilde{\mathscr{A}}^s$ be defined by
    \[
     \tilde{\mathscr{A}}^s:=\{u\in\mathscr{A}^s: \operatorname{supp}\mathscr{F}u\subset \overline{B(0,1)} \}
    \]
where $B(0,1)$ denotes the ball with center $0$ and radius $1$. We endow $\tilde{\mathscr{A}}^s$ with the norm $\|\cdot\|_{\mathscr{A}^s}$. Then $\tilde{\mathscr{A}}^s$ is a Banach space.
\end{lemma}
\begin{proof}
 Let $(u_n)_{n\in\N}\subset\tilde{\mathscr{A}}^s$ be a Cauchy sequence. Since $\mathscr{A}^s$ is a Banach space, we only have to prove that the limit $u:=\lim_{n\to\infty}u_n$ satisfies $\operatorname{supp}\mathscr{F}u\subset \overline{B(0,1)}$. But since
 \[
  \mathscr{F}\colon \mathscr{A}^s\to \mathscr{S}'(\R^n;E)
 \]
 is continuous, it follows that
 \[
  [\mathscr{F}u](f)=\lim_{n\to\infty}[\mathscr{F}u_n](f)=0
 \]
for all $f\in\mathscr{S}(\R^n)$ such that $\operatorname{supp} f\subset \overline{B(0,1)}^c$. This shows the assertion.
\end{proof}

\begin{proposition}\label{Prop:CounterExampleEpsilon}
 Let $\sigma>0$, $r\in\R$ and $p\in[1,2)$. For $\lambda\geq\sigma$ and $g\in \mathscr{A}^s$ let $u_\lambda:=\operatorname{Poi}_{\Delta}(\lambda)g$ be the solution of
 \begin{align}
 \begin{aligned}\label{eq:DirLapInhomogeneous}
  \lambda u_{\lambda}(x)-\Delta u_{\lambda}(x)&=0\quad\quad(x\in\R^n_+),\\
  u_{\lambda}(x',0)&=g(x')\quad (x'\in\R^{n-1})
  \end{aligned}
 \end{align}
which is decaying in normal direction. Then the set of operators
\[
 \{|\lambda|^{\frac{1+r}{2p}}\operatorname{Poi}_{\Delta}(\lambda):\lambda\geq\sigma\}\subset\mathcal{B}(\mathscr{A}^s, L_p(\R_+,|\operatorname{pr}_n|^r;\mathscr{A}^s))
\]
is not $\mathcal{R}$-bounded.
\end{proposition}
\begin{proof}
 Applying Fourier transform in tangential direction to \eqref{eq:DirLapInhomogeneous} we obtain 
 \begin{align*}
  \partial_n^2 \hat{u}(\xi',x_n)&=(\lambda+|\xi'|^2)\hat{u}(\xi',x_n),\\
  &\hat{u}(\xi',0)=\hat{g}(\xi').
 \end{align*}
 The stable solution of this equation is given by $e^{-(\lambda+|\xi'|^2)^{1/2} x_n}\hat{g}(\xi')$ so that the decaying solution of \eqref{eq:DirLapInhomogeneous} is given by 
 \[
  u_{\lambda}(x',x_n)=\operatorname{Poi}_{\Delta}(\lambda)g= [\mathscr{F}_{x'\to\xi'}^{-1}e^{-(\lambda+|\xi'|^2)^{1/2} x_n}\mathscr{F}_{x'\to\xi'}g](x').
 \]
 Let $\chi\subset\mathscr{D}(\R^{n-1})$ be a test function with $\chi(\xi')=1$ for $\xi'\in B(0,1)$ and $\operatorname{supp}\chi\subset B(0,2)$. It holds that $\chi(\xi')e^{((\lambda+|\xi'|^2)^{1/2}-|\lambda|^{1/2})x_n}$ satisfies the Mikhlin condition uniformly in $\lambda\geq\sigma$ and $x_n\leq 1$. Hence, we have that
 \[
  \{\operatorname{op}[\chi(\xi')e^{((\lambda+|\xi'|^2)^{1/2}-|\lambda|^{1/2})x_n}]:\lambda\geq\sigma,x_n\in[0,1]\}\subset\mathcal{B}(\tilde{\mathscr{A}^s})
 \]
is $\mathcal{R}$-bounded, where $\tilde{\mathscr{A}^s}$ is defined as in Lemma~\ref{Lemma:AuxiliarySpace}. Using these observations and assuming the $\mathcal{R}$-boundedness of $\{|\lambda|^{\frac{1+r}{2p}}\operatorname{Poi}_{\Delta}(\lambda):\lambda\geq\sigma\}$, we can carry out the following calculation: Let $(\epsilon_l)_{l\in\N}$ be a Rademacher sequence on the probability space $(\Omega,\mathcal{F},\mathbb{P})$, $\lambda_l=(\sigma2^l)^2$ $(l\in\N)$, $N\in\N$ and $g_1,\ldots,g_N\in\tilde{\mathscr{A}}^s$. Then we obtain
{\allowdisplaybreaks
\begin{align*}
 \left\|\sum_{l=1}^N \epsilon_l g_l\right\|_{L_p(\Omega;\tilde{\mathscr{A}}^s)}&\gtrsim \left(\int_{\Omega}\int_{0}^{\sigma^{-1}} \bigg\| \sum_{l=1}^N \epsilon_l\lambda_l^{\frac{1+r}{2p}}[\operatorname{Poi}_{\Delta}(\lambda_l)g_l](\,\cdot,x_n)\bigg\|^p_{\mathscr{A}^s}x_n^r\,dx_n\,d\mathbb{P}\right)^{1/p}\\
  &= \left(\int_{\Omega}\int_{0}^{\sigma^{-1}} \bigg\| \sum_{l=1}^N \epsilon_l\lambda_l^{\frac{1+r}{2p}}\operatorname{op}[\chi(\xi')e^{-(\lambda+|\xi'|^2)^{1/2} x_n}]g_l\bigg\|^p_{\mathscr{A}^s}x_n^r\,dx_n\,d\mathbb{P}\right)^{1/p}\\
 &\gtrsim \left(\int_{\Omega}\int_{0}^{\sigma^{-1}} \bigg\| \sum_{l=1}^N \epsilon_l\lambda_l^{\frac{1+r}{2p}}e^{-|\lambda_l|^{1/2}x_n}g_l\bigg\|^p_{\mathscr{A}^s}x_n^r\,dx_n\,d\mathbb{P}\right)^{1/p}\\
 &\gtrsim\left(\int_{\Omega}\sum_{m=1}^N\int_{\sigma^{-1}2^{-m}}^{\sigma^{-1}2^{-m+1}} \bigg\| \sum_{l=1}^N \epsilon_l\lambda_l^{\frac{1+r}{2p}}e^{-|\lambda_l|^{1/2}x_n}g_l\bigg\|^p_{\mathscr{A}^s}x_n^r\,dx_n\,d\mathbb{P}\right)^{1/p}\\
 &\gtrsim \left(\int_{\Omega}\sum_{m=1}^N\int_{\sigma^{-1}2^{-m}}^{\sigma^{-1}2^{-m+1}}  \lambda_m^{\frac{1+r}{2}}e^{-p|\lambda_m|^{1/2}x_n}\| g_m\|^p_{\mathscr{A}^s}x_n^r\,dx_n\,d\mathbb{P}\right)^{1/p}\\
 &\gtrsim \left(\sum_{m=1}^N\| g_m\|^p_{\mathscr{A}^s} \right)^{1/p}.
\end{align*}}
This shows that $\tilde{\mathscr{A}}^s$ has cotype $p$. However, $\tilde{\mathscr{A}}^s$ is a nontrivial Banach space by Lemma~\ref{Lemma:AuxiliarySpace} and therefore its cotype must satisfy $p\geq2$. This contradicts $p\in[1,2)$ and hence $\{|\lambda|^{\frac{1+r}{2p}}\operatorname{Poi}_{\Delta}(\lambda):\lambda\geq\sigma\}$ can not be $\mathcal{R}$-bounded.
\end{proof}

\begin{remark} Proposition~\ref{Prop:CounterExampleEpsilon} shows that it is not possible in general to remove the $\epsilon$ in Proposition~\ref{Chapter3::Proposition::PoissonMappingPropertySobolevRBoundedEpsilon}.
Even though we only treat the Laplacian with Dirichlet boundary conditions in Proposition~\ref{Prop:CounterExampleEpsilon} it seems like the integrability parameter in normal direction may not be smaller than the cotype of the space in tangential directions in order to obtain the sharp estimate of Proposition~\ref{Chapter3::Proposition::PoissonMappingPropertySobolevRBounded}.
\end{remark}

\begin{remark}\label{Rem:Simplifications}
 Depending on what one aims for, it can also be better to substitute $t=\tilde{t}+m_j-k-\zeta$ in Proposition \ref{Chapter3::Proposition::PoissonMappingPropertySobolev}, Proposition \ref{Chapter3::Proposition::PoissonMappingPropertySobolevRBoundedEpsilon} or Proposition \ref{Chapter3::Proposition::PoissonMappingPropertySobolevRBounded}. In this case, we obtain the estimates
  \begin{align*}
   \|\operatorname{Poi}_{j}(\lambda)\|&\leq C|\lambda|^{\frac{-1-\gamma_1-\gamma_2}{2mp}},\quad(\text{Proposition \ref{Chapter3::Proposition::PoissonMappingPropertySobolev}}),\\
    \mathcal{R}\big(\{|\lambda|^{\frac{1+\gamma_1+\gamma_2-\epsilon}{2mp}}\operatorname{Poi}_{j}(\lambda):\lambda\in\Sigma_{\phi},\,|\lambda|\geq\sigma \}\big)&\leq C,\quad(\text{Proposition \ref{Chapter3::Proposition::PoissonMappingPropertySobolevRBoundedEpsilon}}),\\
    \mathcal{R}\big(\{|\lambda|^{\frac{1+\gamma_1+\gamma_2}{2mp}}\operatorname{Poi}_{j}(\lambda):\lambda\in\Sigma_{\phi},\,|\lambda|\geq\sigma \}\big)&\leq C,\quad(\text{Proposition \ref{Chapter3::Proposition::PoissonMappingPropertySobolevRBounded}}),
 \end{align*}
 where
 \[
  \gamma_1=r-p[t+k+\zeta-m_j-s]_+,\quad\gamma_2=p\zeta-p[-[t+k+\zeta-m_j-s]_+-m_j+k+\zeta]_+.
 \]
 and where the operator norms and the $\mathcal{R}$-bounds are taken in 
 $$\mathcal{B}(\mathscr{A}^{s}(\R^{n-1},w;E),W^{k}_{p}((\epsilon,\infty),|\operatorname{pr}_n|^{r};\mathscr{A}^{t}(\R^{n-1},w;E^{2m}))).$$
 If we now choose $\zeta:=[m_j+s-k-t]_+$, then we obtain
\[
 \gamma_1=r-p[t+k-m_j-s]_+,\quad\gamma_2=p[m_j+s-k-t]_+-p[s-t]_+.
\]
From this it follows that
\[
 -\gamma_1-\gamma_2=-r+p(k-m_j)+p([s-t]_++t-s)=-r+p(k-m_j)+p[t-s]_+
\]
This yields the following result:
\end{remark}

\begin{theorem}\label{Thm:MainThm1}
 Recall Assumption \ref{Assump:ELS} and Assumption \ref{Assump:Spaces}. Let $k\in\N_0$, $r,s,t\in\R$ and $p\in[1,\infty)$. Suppose that $r-p[t+k-m_j-s]_+>-1$.
 \begin{enumerate}[(a)]
  \item \label{Thm:MainThm1:UniformBounds} For all $\sigma>0$ there is a constant $C>0$ such that
 \begin{align*}
   \|\operatorname{Poi}_{j}(\lambda)\|\leq C|\lambda|^{\frac{-1-r+p(k-m_j)+p[t-s]_+}{2mp}}
 \end{align*}
 for all $\lambda\in\Sigma_{\phi}$ such that $|\lambda|\geq\sigma$ where the operator norms are taken in the space $\mathcal{B}(\mathscr{A}^{s}(\R^{n-1},w;E),W^{k}_{p}(\R_+,|\operatorname{pr}_n|^{r};\mathscr{A}^{t}(\R^{n-1},w;E^{2m})))$.
 \item Let $\epsilon\in(0,\gamma_1+1)$. Then for all $\sigma>0$ there is a constant $C>0$ such that
 \begin{align*}
   \mathcal{R}\big(\{|\lambda|^{\frac{1+r-\epsilon-p(k-m_j)-p[t-s]_+}{2mp}}\operatorname{Poi}_{j}(\lambda):\lambda\in\Sigma_{\phi},\,|\lambda|\geq\sigma \}\big)\leq C
 \end{align*}
 where the $\mathcal{R}$-bounds are taken in $\mathcal{B}(\mathscr{A}^{s}(\R^{n-1},w;E),W^{k}_{p}(\R_+,|\operatorname{pr}_n|^{r};\mathscr{A}^{t}(\R^{n-1},w;E^{2m})))$.
 \item Suppose that the assumptions of Proposition \ref{Prop:Contraction_Principle_Functions} hold true. Then for all $\sigma>0$ there is a constant $C>0$ such that
 \begin{align*}
   \mathcal{R}\big(\{|\lambda|^{\frac{1+r-p(k-m_j)-p[t-s]_+}{2mp}}\operatorname{Poi}_{j}(\lambda):\lambda\in\Sigma_{\phi},\,|\lambda|\geq\sigma \}\big)\leq C
 \end{align*}
 where the $\mathcal{R}$-bounds are taken in $\mathcal{B}(\mathscr{A}^{s}(\R^{n-1},w;E),W^{k}_{p}(\R_+,|\operatorname{pr}_n|^{r};\mathscr{A}^{t}(\R^{n-1},w;E^{2m})))$.
 \end{enumerate}
\end{theorem}
\begin{proof}
 This follows from Proposition \ref{Chapter3::Proposition::PoissonMappingPropertySobolev}, Proposition \ref{Chapter3::Proposition::PoissonMappingPropertySobolevRBoundedEpsilon} and Proposition \ref{Chapter3::Proposition::PoissonMappingPropertySobolevRBounded} together with the observations in Remark \ref{Rem:Simplifications}. 
\end{proof}

\begin{corollary}\label{Cor:MainThm1_negative}
Let $k\in\Z$, $s,t\in\R$, $p\in(1,\infty)$ and $r\in(-1,p-1)$. Suppose that $r-p[t+k-m_j-s]_+>-1$ and that $\mathscr{A}^s$ is reflexive.
\begin{enumerate}[(a)]
  \item For all $\sigma>0$ there is a constant $C>0$ such that
 \begin{align*}
   \|\operatorname{Poi}_{j}(\lambda)\|\leq C|\lambda|^{\frac{-1-r+p(k-m_j)+p[t-s]_+}{2mp}}
 \end{align*}
 for all $\lambda\in\Sigma_{\phi}$ such that $|\lambda|\geq\sigma$ where the operator norms are taken in the space $\mathcal{B}(\mathscr{A}^{s}(\R^{n-1},w;E),W^{k}_{p}(\R_+,|\operatorname{pr}_n|^{r};\mathscr{A}^{t}(\R^{n-1},w;E^{2m})))$.
 \item Let $\epsilon\in(0,\gamma_1+1)$. Then for all $\sigma>0$ there is a constant $C>0$ such that
 \begin{align*}
   \mathcal{R}\big(\{|\lambda|^{\frac{1+r-\epsilon-p(k-m_j)-p[t-s]_+}{2mp}}\operatorname{Poi}_{j}(\lambda):\lambda\in\Sigma_{\phi},\,|\lambda|\geq\sigma \}\big)\leq C
 \end{align*}
 where the $\mathcal{R}$-bounds are taken in $\mathcal{B}(\mathscr{A}^{s}(\R^{n-1},w;E),W^{k}_{p}(\R_+,|\operatorname{pr}_n|^{r};\mathscr{A}^{t}(\R^{n-1},w;E^{2m})))$.
 \item Suppose that the assumptions of Proposition \ref{Prop:Contraction_Principle_Functions} hold true. Then for all $\sigma>0$ there is a constant $C>0$ such that
 \begin{align*}
   \mathcal{R}\big(\{|\lambda|^{\frac{1+r-p(k-m_j)-p[t-s]_+}{2mp}}\operatorname{Poi}_{j}(\lambda):\lambda\in\Sigma_{\phi},\,|\lambda|\geq\sigma \}\big)\leq C
 \end{align*}
 where the $\mathcal{R}$-bounds are taken in $\mathcal{B}(\mathscr{A}^{s}(\R^{n-1},w;E),W^{k}_{p}(\R_+,|\operatorname{pr}_n|^{r};\mathscr{A}^{t}(\R^{n-1},w;E^{2m})))$.
 \end{enumerate}
\end{corollary}
\begin{proof}
	The case $k\in\N_0$ is already contained in Theorem \ref{Thm:MainThm1}. Hence, we only treat the case $k<0$. In this case it holds that
	\[
		(r-pk)-p[t-m_j-s]_+\geq r-p[t+k-m_j-s]_+>-1.
	\]
	Hence, Theorem \ref{Thm:MainThm1} holds with a weight of the power $r-pk$ and smoothness $0$ in normal direction. Combining this with Lemma \ref{Lemma:EmbeddingNegitveSmoothness} yields the assertion.
\end{proof}

\begin{corollary}\label{Cor:MainThm1_Corollary}
Let $s,t\in\R$, $k\in(0,\infty)\setminus\N$, $p\in(1,\infty)$, $r\in(-1,p-1)$ and $q\in[1,\infty]$. We write $k=\overline{k}-\theta$ with $\overline{k}\in\N_0$ and $\theta\in[0,1)$. Suppose that $r-p[t+\overline{k}-m_j-s]_+>-1$.
 \begin{enumerate}[(a)]
  \item For all $\sigma>0$ there is a constant $C>0$ such that for all $\lambda\in\Sigma_{\phi}$ with $|\lambda|\geq\sigma$ we have the estimate
 \begin{align*}
   \|\operatorname{Poi}_{j}(\lambda)\|\leq C|\lambda|^{\frac{-1-r+p(k-m_j)+p[t-s]_+}{2mp}}
 \end{align*}
  where the norm is taken in $\mathcal{B}(\mathscr{A}^{s}(\R^{n-1},w;E),H^{k}_{p}(\R_+,|\operatorname{pr}_n|^{r};\mathscr{A}^{t}(\R^{n-1},w;E^{2m})))$ or in $\mathcal{B}(\mathscr{A}^{s}(\R^{n-1},w;E),B^{k}_{p,q}(\R_+,|\operatorname{pr}_n|^{r};\mathscr{A}^{t}(\R^{n-1},w;E^{2m}))).$
 \item Let $\epsilon\in(0,\gamma_1+1)$ and let $E$ be a UMD space. Then for all $\sigma>0$ there is a constant $C>0$ such that
 \begin{align*}
   \mathcal{R}\big(\{|\lambda|^{\frac{1+r-\epsilon-p(k-m_j)-p[t-s]_+}{2mp}}\operatorname{Poi}_{j}(\lambda):\lambda\in\Sigma_{\phi},\,|\lambda|\geq\sigma \}\big)\leq C
 \end{align*}
 where the $\mathcal{R}$-bounds are taken in $\mathcal{B}(\mathscr{A}^{s}(\R^{n-1},w;E),H^{k}_{p}(\R_+,|\operatorname{pr}_n|^{r};\mathscr{A}^{t}(\R^{n-1},w;E^{2m})))$.
 \item Suppose that the assumptions of Proposition \ref{Prop:Contraction_Principle_Functions} hold true and let $E$ be a UMD space. Then for all $\sigma>0$ there is a constant $C>0$ such that
 \begin{align*}
   \mathcal{R}\big(\{|\lambda|^{\frac{1+r-p(k-m_j)-p[t-s]_+}{2mp}}\operatorname{Poi}_{j}(\lambda):\lambda\in\Sigma_{\phi},\,|\lambda|\geq\sigma \}\big)\leq C
 \end{align*}
 where the $\mathcal{R}$-bounds are taken in $\mathcal{B}(\mathscr{A}^{s}(\R^{n-1},w;E),H^{k}_{p}(\R_+,|\operatorname{pr}_n|^{r};\mathscr{A}^{t}(\R^{n-1},w;E^{2m})))$.
 \end{enumerate}
\end{corollary}
\begin{proof}
 This follows from Theorem \ref{Thm:MainThm1} together with real and complex interpolation, see \cite[Proposition 6.1, $(6.4)$]{Meyries_Veraar_2012} together with a retraction-coretraction argument, \cite[Proposition 5.6]{Lindemulder_Meyries_Veraar_2018} and Proposition \ref{Prop:Rbounded_Interpolation}.  Note that the power weight $|\operatorname{pr}_n|^{r}$ is an $A_p$ weight, since $r\in(-1,p-1)$, see \cite[Example 9.1.7]{Grafakos_2009}.\\
\end{proof}

\noindent
\begin{lemma}\label{Lemma:Hilbert_Operator}
	Let $p\in(1,\infty)$, $r\in (-1,p-1)$ and $w_r(x):=x^r$ for $x\in\R_+$. Then the linear operator
	\[
		T\colon L_p(\R_+,w_r;\R)\to L_p(\R_+,w_r;\R),\;f\mapsto \int_{0}^{\infty} \frac{f(y)}{x+y}\,dy
	\]
	is bounded.
\end{lemma}
\begin{proof}
	In \cite[Appendix I.3]{Grafakos_2008} this was shown for $r=0$ using Schur's Lemma. We adjust the same proof to the weighted setting.\\
	Let $K(x,y):=\frac{1}{y^r(x+y)}$. Then we may write
	\[
		(Tf)(x)=\int_0^\infty K(x,y)f(y) y^r\,dy.
	\]
	We further define the transpose operator
		\[
		(T^tf)(y)=\int_0^\infty K(x,y)f(x) x^r\,dx=\frac{1}{y^r}\int_0^\infty \frac{f(x)}{x+y}x^r\,dx.
	\]
	By the lemma in \cite[Appendix I.2]{Grafakos_2008} it is sufficient to find $C>0$ and $u,v\colon\R_+\to(0,\infty)$ such that
	\[
		T(u^{p'})\leq Cv^{p'}\quad\text{and}\quad T^t(v^{p})\leq Cu^{p},
	\]
	where $1=\frac{1}{p}+\frac{1}{p'}$. Similar to \cite[Appendix I.3]{Grafakos_2008} we choose $$u(x):=v(x):=x^{-\frac{1+r}{pp'}}$$
	and
	\[
		C:=\max\left\{\int_0^\infty\frac{t^{-\frac{1+r}{p}}}{1+t}\,dt,\int_0^\infty \frac{t^{r-\frac{1+r}{p'}}}{1+t}\,dt \right\}.
	\]
	Note that $r\in(-1,p-1)$ ensures that both integrals are finite since
	\begin{align*}
		-\frac{1+r}{p}\in(-1,0)\Longleftrightarrow r\in(-1,p-1)
	\end{align*}
	and 
		\begin{align*}
		r-\frac{1+r}{p'}\in(-1,0)\Longleftrightarrow r\in(-1,p-1).
	\end{align*}
With this choice we obtain
	\begin{align*}
		(T u^{p'})(x)=\int_0^\infty \frac{y^{-\frac{1+r}{p}}}{x+y}\,dy=x^{-\frac{1+r}{p}}\int_0^\infty\frac{t^{-\frac{1+r}{p}}}{1+t}\,dt\leq C v(x)^{p'}
	\end{align*}
	and
	\begin{align*}
	(T^t v^{p})(y)=\frac{1}{y^r}\int_0^\infty \frac{x^{r-\frac{1+r}{p'}}}{x+y}\,dx=y^{-\frac{1+r}{p'}}\int_0^\infty \frac{t^{r-\frac{1+r}{p'}}}{1+t}\,dt\leq Cu(y)^p.
	\end{align*}
	This shows the assertion.
\end{proof}

From now on, we use the notation
\begin{align}
 \begin{aligned}\label{Eq:MoreSpaces}
 D^{k,\tilde{k},s}_{r}(I):= H_p^{k}(I,|\operatorname{pr}_n|^r,\mathscr{A}^{s+\tilde{k}}) &\cap H^{k+\tilde{k}}_p(I,|\operatorname{pr}_n|^r,\mathscr{A}^{s}),\\
 D^{k,2m,s}_{r,B}(I):= \{u\in H_p^{k}(I,|\operatorname{pr}_n|^r,\mathscr{A}^{s+2m}) &\cap H^{k+2m}_p(I,|\operatorname{pr}_n|^r,\mathscr{A}^{s}):\\&\operatorname{tr}_{x_n=0} B_j(D)u=0\text{ for all }j=1,\ldots,m\}
 \end{aligned}
\end{align}
for $p\in(1,\infty)$ $k,\tilde{k}\in[0,\infty)$, $s\in\R$, $r\in(-1,p-1)$ and $I\in\{\R_+,\R\}$. Moreover, we endow both spaces with the norm
\begin{align*}
 \|u\|_{D^{k,\tilde{k},s}_{r}(I)}=\max\{\|u\|_{H_p^{k}(I,|\operatorname{pr}_n|^r,\mathscr{A}^{s+\tilde{k}})}, \|u\|_{H^{k+\tilde{k}}_p(I,|\operatorname{pr}_n|^r,\mathscr{A}^{s})} \},\\
  \|u\|_{D^{k,2m,s}_{r,B}(I)}=\max\{\|u\|_{H_p^{k}(I,|\operatorname{pr}_n|^r,\mathscr{A}^{s+2m})}, \|u\|_{H^{k+2m}_p(I,|\operatorname{pr}_n|^r,\mathscr{A}^{s})} \},
\end{align*}
respectively, so that $(D^{k,\tilde{k},s}_{r}(I),\|\cdot\|_{D^{k,\tilde{k},s}_{r}(I)})$ and $(D^{k,2m,s}_{r,B}(I),\|\cdot\|_{D^{k,2m,s}_{r,B}(I)})$ are a Banach spaces.

\begin{proposition}\label{Prop:PoissonAfterTrace}
 Let $s\in\R$, $p\in(1,\infty)$, $r\in(-1,p-1)$ and $k\in\N_0$ such that $k\leq\min\{\beta_n:\beta\in\N_0^n,|\beta|=m_j,b_\beta^j\neq0\}$. Let further $u\in D^{k,2m,s}_{r}(\R_+)$ and $\theta\in[0,1]$ such that $2m\theta\in\N_0$. Then for all $\sigma>0$ there is a constant $C>0$ such that we have the estimate
 \begin{align*}
  \mathcal{R}(\{\lambda^{\theta}\operatorname{Poi}_j(\lambda)\operatorname{tr}_{x_n}B_j(D):\lambda\in\Sigma_{\phi},\,|\lambda|\geq\sigma\})\leq C
 \end{align*}
 where the $\mathcal{R}$-bound is taken in
 \[
  \mathcal{B}(D^{k,2m,s}_{r}(\R_+),D^{k,2m(1-\theta),s}_{r}(\R_+)).
 \]
\end{proposition}
\begin{proof}
 The proof uses an approach which is sometimes referred to as Volevich-trick. This approached is already standard in the treatment of parameter-elliptic and parabolic boundary value problems in classical Sobolev spaces, see for example Lemma 7.1 in \cite{DHP03} and how it is used to obtain the results therein. The idea is to use the fundamental theorem of calculus in normal directions and to apply the boundedness of the operator
 \[
T\colon L_p(\R_+,w_r;\R)\to L_p(\R_+,w_r;\R),\;f\mapsto \int_{0}^{\infty} \frac{f(y)}{x+y}\,dy
 \]
from Lemma~\ref{Lemma:Hilbert_Operator}. Using these ideas in connection with Corollary~\ref{Chapter3::Corollary::PoissonEvalEstimate} we can carry out the following computation: Let $(\epsilon_n)_{n\in\N}$ be a Rademacher sequence on the probability space $(\Omega,\mathcal{F},\mathbb{P})$, $N\in\N$, $\lambda_1,\ldots,\lambda_N$ and $u_1,\ldots,u_N\in H^{k}(\R_+,|\operatorname{pr}_n|^r;\mathscr{A}^{s+2m})\cap H^{k+2m}(\R_+,|\operatorname{pr}_n|^r;\mathscr{A}^{s})$. Then we obtain
{\allowdisplaybreaks
\begin{align*}
   &\quad  \left\| \sum_{l=1}^N \epsilon_l\lambda_l^{\theta}\operatorname{Poi}_j\operatorname{tr}_{x_n=0}B_j(D) u_l \right\|_{L_p(\Omega;D^{k,2m(1-\theta),s}_{r}(\R_+))}\\
    \lesssim&\sum_{\tilde{k}=0}^k\left\| \sum_{l=1}^N \epsilon_l\lambda_l^{\theta}D_{x_n}^{\tilde{k}}\operatorname{Poi}_j\operatorname{tr}_{x_n=0}B_j(D) u_l \right\|_{L_p(\Omega;L_p(\R_+,|\operatorname{pr}_n|^r;\mathscr{A}^{s+2m(1-\theta)})}\\
    &+ \sum_{\tilde{k}=0}^{k+(1-\theta)2m}\left\| \sum_{l=1}^N \epsilon_l\lambda_l^{\theta}D_{x_n}^{\tilde{k}}\operatorname{Poi}_j\operatorname{tr}_{x_n=0}B_j(D) u_l \right\|_{L_p(\Omega;L_p(\R_+,|\operatorname{pr}_n|^r;\mathscr{A}^{s})}\\
    \leq &\sum_{\tilde{k}=0}^k\left(\int_{\Omega}\int_{\R_+}\bigg\| \int_{\R_+}\sum_{l=1}^N \epsilon_l\lambda_l^{\theta}[\partial_{y_n}D_{x_n}^{\tilde{k}}\operatorname{ev}_{x_n+y_n}\operatorname{Poi}_j][B_j(D) u_l](\,\cdot\,,y_n)\,dy_n\bigg\|_{\mathscr{A}^{s+2m(1-\theta)}}^px_n^r\,dx_n\,d\mathbb{P}\right)^{1/p}\\
    &+\sum_{\tilde{k}=0}^k\left(\int_{\Omega}\int_{\R_+}\bigg\| \int_{\R_+}\sum_{l=1}^N \epsilon_l\lambda_l^{\theta}[D_{x_n}^{\tilde{k}}\operatorname{ev}_{x_n+y_n}\operatorname{Poi}_j][\partial_{y_n}B_j(D) u_l](\,\cdot\,,y_n)\,dy_n\bigg\|_{\mathscr{A}^{s+2m(1-\theta)}}^px_n^r\,dx_n\,d\mathbb{P}\right)^{1/p}\\
    &+\sum_{\tilde{k}=0}^{k+(1-\theta)2m}\left(\int_{\Omega}\int_{\R_+}\bigg\| \int_{\R_+}\sum_{l=1}^N \epsilon_l\lambda_l^{\theta}[\partial_{y_n}D_{x_n}^{\tilde{k}}\operatorname{ev}_{x_n+y_n}\operatorname{Poi}_j][B_j(D) u_l](\,\cdot\,,y_n)\,dy_n\bigg\|_{\mathscr{A}^{s}}^px_n^r\,dx_n\,d\mathbb{P}\right)^{1/p}\\
    &+ \sum_{\tilde{k}=0}^{k+(1-\theta)2m}\left(\int_{\Omega}\int_{\R_+}\bigg\| \int_{\R_+}\sum_{l=1}^N \epsilon_l\lambda_l^{\theta}[D_{x_n}^{\tilde{k}}\operatorname{ev}_{x_n+y_n}\operatorname{Poi}_j][\partial_{y_n}B_j(D) u_l](\,\cdot\,,y_n)\,dy_n\bigg\|_{\mathscr{A}^{s}}^px_n^r\,dx_n\,d\mathbb{P}\right)^{1/p}.
\end{align*}}
In order to keep the notation shorter, we continue the computation with just the first of the four terms. The steps we would have to carry out for the other three terms, are almost exactly the same with just minor changes on the parameters. We obtain
{\allowdisplaybreaks
\begin{align*}
 &\quad\sum_{\tilde{k}=0}^k\left(\int_{\Omega}\int_{\R_+}\bigg\| \int_{\R_+}\sum_{l=1}^N \epsilon_l\lambda_l^{\theta}[\partial_{y_n}D_{x_n}^{\tilde{k}}\operatorname{ev}_{x_n+y_n}\operatorname{Poi}_j][B_j(D) u_l](\,\cdot\,,y_n)\,dy_n\bigg\|_{\mathscr{A}^{s+2m(1-\theta)}}^px_n^r\,dx_n\,d\mathbb{P}\right)^{1/p}\\
\leq & \sum_{\tilde{k}=0}^k\left(\int_{\Omega}\int_{\R_+}\bigg( \int_{\R_+}\bigg\|\sum_{l=1}^N \epsilon_l\lambda_l^{\theta}[\partial_{y_n}D_{x_n}^{\tilde{k}}\operatorname{ev}_{x_n+y_n}\operatorname{Poi}_j][B_j(D) u_l](\,\cdot\,,y_n)\bigg\|_{\mathscr{A}^{s+2m(1-\theta)}}\,dy_n\bigg)^px_n^r\,dx_n\,d\mathbb{P}\right)^{1/p}\\
 \lesssim& \left(\int_{\Omega}\int_{\R_+}\bigg( \int_{\R_+}\bigg\|\sum_{l=1}^N \epsilon_l\frac{1}{x_n+y_n}[B_j(D) u_l](\,\cdot\,,y_n)\bigg\|_{\mathscr{A}^{s+k+2m-m_j}}\,dy_n\bigg)^px_n^r\,dx_n\,d\mathbb{P}\right)^{1/p}\\
 \lesssim & \left(\int_{\Omega}\int_{\R_+}\bigg\|\sum_{l=1}^N \epsilon_l[B_j(D) u_l](\,\cdot\,,x_n)\bigg\|_{\mathscr{A}^{s+k+2m-m_j}}^px_n^r\,dx_n\,d\mathbb{P}\right)^{1/p}\\
 \leq & \sum_{|\beta|=m_j}\bigg\| b_{\beta}^j\partial_n^{\beta_n}\partial_{x'}^{\beta'}\sum_{l=1}^N \epsilon_l   u_l\bigg\|_{L_p(\Omega;L_p(\R_+,|\operatorname{pr}_n|^r;\mathscr{A}^{s+k+2m-m_j}))}\\
  \lesssim & \sum_{\beta_n=k}^{m_j}\bigg\| \sum_{l=1}^N \epsilon_l   u_l\bigg\|_{L_p(\Omega;H^{\beta_n}_p(\R_+,|\operatorname{pr}_n|^r;\mathscr{A}^{s+k+2m-\beta_n}))}.
\end{align*}}
From the second to the third line we used Corollary~\ref{Chapter3::Corollary::PoissonEvalEstimate}, from the third to the fourth line we used Lemma~\ref{Lemma:Hilbert_Operator} and in the last step we used that $k\leq\min\{\beta_n:\beta\in\N_0^n,|\beta|=m_j,b_\beta^j\neq0\}$. The other three terms above can either also be estimated by 
\[
 \sum_{\beta_n=k}^{m_j}\bigg\| \sum_{l=1}^N \epsilon_l   u_l\bigg\|_{L_p(\Omega;H^{\beta_n}_p(\R_+,|\operatorname{pr}_n|^r;\mathscr{A}^{s+k+2m-\beta_n}))}
\]
or by 
\[
 \sum_{\beta_n=k}^{m_j}\bigg\| \sum_{l=1}^N \epsilon_l   u_l\bigg\|_{L_p(\Omega;H^{\beta_n+1}_p(\R_+,|\operatorname{pr}_n|^r;\mathscr{A}^{s+k+2m-\beta_n-1}))}
\]
if the derivative $\partial_{y_n}$ is taken of $g_j$ instead of $\operatorname{Poi}_j$. Since $m_j< 2m$, we obtain the estimate
\[
  \left\| \sum_{l=1}^N \epsilon_l\lambda_l^{\theta}\operatorname{Poi}_j\operatorname{tr}_{x_n=0}B_j(D) u_l \right\|_{L_p(\Omega;D^{k,2m(1-\theta),s}_{r}(\R_+))}\lesssim  \left\| \sum_{l=1}^N \epsilon_l u_l \right\|_{L_p(\Omega;D^{k,2m,s}_{r}(\R_+))}
\]
\end{proof}


\section{Resolvent Estimates}\label{Section:Resolvent}

Now we study the resolvent problem, i.e. \eqref{EllipticBVP} with $g_j=0$. We show that the corresponding operator is $\mathcal{R}$-sectorial and thus has the property of maximal regularity in the UMD case. But first, we prove the $\mathcal{R}$-sectoriality in $\R^n$.
\begin{theorem}\label{Thm:Rboundedness_Rn}
 Let $k,s\in\R$. Suppose that $E$ satisfies Pisier's property $(\alpha)$ if one of the scales $\mathscr{A}, \mathscr{B}$ belongs to the Bessel potential scale. Then for all $\sigma>0$ the realization of $A(D)-\sigma$ in $\mathscr{B}^k(\mathscr{A}^s)$ given by
 \[
  A(D)-\sigma\colon \mathscr{B}^k(\mathscr{A}^s) \supset \mathscr{B}^{k+2m}(\mathscr{A}^s)\cap \mathscr{B}^k(\mathscr{A}^{s+2m})\to  \mathscr{B}^k(\mathscr{A}^s),\, u\mapsto A(D)u-\sigma u
 \]
is $\mathcal{R}$-sectorial in $\Sigma_{\phi}$ and there is a constant $C>0$ such that the estimate
\begin{align}\label{Eq:EstimateDomain_Rn}
 \|u\|_{\mathscr{B}^{k+2m}(\mathscr{A}^s)\cap \mathscr{B}^k(\mathscr{A}^{s+2m})}\leq C\|(\lambda+\sigma-A(D)u\|_{\mathscr{B}^k(\mathscr{A}^s)} 
\end{align}
holds for all $\lambda\in\Sigma_{\phi}$.
\end{theorem}
\begin{proof}
 It was shown in \cite[Lemma 5.10]{Hummel_Lindemulder_2019} that
 \begin{align}\label{Eq:Rbounded_aus_Paper_mit_Nick}
  \mathcal{R}(\{\langle\xi\rangle^{|\alpha|}D^{\alpha}_{\xi} (s_1+s_2\lambda+s_3|\xi|^{2m})(\lambda+1-A(\xi))^{-1}:\lambda\in\Sigma_{\phi},\,\xi\in\R^n\})<\infty\quad\text{in }\mathcal{B}(E)
 \end{align}
holds for all $\alpha\in\mathcal{B}(E)$ and all $(s_1,s_2,s_3)\in\R^3$. Note that the authors of \cite{Hummel_Lindemulder_2019} use a different convention concerning the sign of $A$. Taking $(s_1,s_2,s_3)=(0,1,0)$ shows together with the iterated $\mathcal{R}$-bounded versions of Mihklin's theorem, Proposition \ref{Prop:IteratedMikhlinRbounded}, show that
\[
 \mathcal{R}(\{\lambda(\lambda+1-A(D))^{-1}:\lambda\in\Sigma_{\phi}\})<\infty.
\]
Thus, it only remains to prove that $\mathscr{B}^{k+2m}(\mathscr{A}^s)\cap \mathscr{B}^k(\mathscr{A}^{s+2m})$ is the right domain and that \eqref{Eq:EstimateDomain_Rn} holds. But \eqref{Eq:Rbounded_aus_Paper_mit_Nick} with $(s_1,s_2,s_3)=(1,0,1)$ shows that
\[
 [\xi\mapsto (1+|\xi|^{2m})(\lambda+1-A(\xi))^{-1}]\in S^0_{\mathcal{R}}(\R^n;\mathcal{B}(E))
\]
so that
\[
  [\xi\mapsto (\lambda+1-A(\xi))^{-1}]\in S^{-2m}_{\mathcal{R}}(\R^n;\mathcal{B}(E))
\]
uniformly in $\lambda\in\Sigma_{\phi}$. Now the assertion follows from Proposition \ref{Prop:Pseudo_Iterated_Mapping_Properties}.
\end{proof}

\begin{remark}\phantomsection \label{Rem:Rboundedness_Rn}
 \begin{enumerate}[(a)]
  \item If both $\mathscr{A}$ and $\mathscr{B}$ belong to the Bessel potential scale, then Theorem \ref{Thm:Rboundedness_Rn} can be improved in the following way: Lemma \ref{Lemma:Lifting_Property} together with Fubini's theorem yields that
  \[
   \langle D' \rangle^{-s}\langle D_n\rangle^{-k} L_p(\R^n_x,w_0\otimes w_1;E)\stackrel{\eqsim}{\to} H^k_p(\R_{x_n},w_1;H^{s}_p(\R^{n-1}_{x'},w_0,E)).
  \]
Moreover, we have
\begin{align*}
 \langle D' \rangle^{-s}\langle D_n\rangle^{-k} &H^{2m}_p(\R^n_x,w_0\otimes w_1;E)\\
 &=H^{k+2m}_p(\R_{x_n},w_1;H^{s}_p(\R^{n-1}_{x'},w_0,E))\cap H^{k}_p(\R_{x_n},w_1;H^{s+2m}_p(\R^{n-1}_{x'},w_0,E)).
\end{align*}
But it is well-known that the realization of $A(D)$ even admits a bounded $\mathcal{H}^{\infty}$-calculus in $ L_p(\R^n_x,w_0\otimes w_1;E)$ with domain $H^{2m}_p(\R^n_x,w_0\otimes w_1;E)$ no matter whether Pisier's property $(\alpha)$ is satisfied or not (recall that the weights in the Bessel potential case are in $A_p$). This can be derived by using the weighted versions of Mihklin's theorem  in the proof of \cite[Theorem 5.5]{DHP03}. Since $\langle D' \rangle^{-s}\langle D_n\rangle^{-k}$ is an isomorphism, $A(D)$ also admits a bounded $\mathcal{H}^{\infty}$-calculus in $H^k_p(\R_{x_n},w_1;H^{s}_p(\R^{n-1}_{x'},w_0,E))$ with domain $$H^{k+2m}_p(\R_{x_n},w_1;H^{s}_p(\R^{n-1}_{x'},w_0,E))\cap H^{k}_p(\R_{x_n},w_1;H^{s+2m}_p(\R^{n-1}_{x'},w_0,E)),$$
even if Pisier's property $(\alpha)$ is not satisfied.
\item \label{Rem:Rboundedness_Rn:Sec} In the proof of Theorem \ref{Thm:Rboundedness_Rn}, one can also use Proposition \ref{Prop:IteratedMikhlin} instead of Proposition \ref{Prop:IteratedMikhlinRbounded} if one only needs sectoriality. In this case, we can again drop the assumption that $E$ has to satisfy Pisier's property $(\alpha)$.
 \end{enumerate}
\end{remark}

\begin{remark}\label{Rem:k_max_explanation}
	For the $\mathcal{R}$-sectoriality of the boundary value problem, which we are going to derive in Theorem~\ref{Thm:RSectorial_Rnplus}, we have a restriction on the regularity in normal direction. It may not be larger than $k_{max}\in\N_0$ which we define by 
	 	\[
 		k_{\max}:=\min\{\beta_n|\,\exists j\in\{1,\ldots,m\}\exists\beta\in\N_0^n,|\beta|=m_j: b^j_{\beta}\neq 0\},
 	\]
 	i.e. $k_{\max}$ is the minimal order in normal direction of all non-zero differential operators which appear in any of the boundary operators
 	\[
 		B_j(D)=\sum_{|\beta|=m_j} b^j_{\beta}D^{\beta}\quad(j=1,\ldots,m).
 	\]
Therefore, if there is a non-zero term with no normal derivatives in one of the $B_1,\ldots, B_n$, then $k_{\max}=0$. In particular, it holds that $k_{\max}=0$ if one of the operators $B_1,\ldots, B_n$ corresponds to the Dirichlet trace at the boundary. This includes the case of the Dirichlet Laplacian. On the other hand, for the Neumann Laplacian we have $k_{\max}=1$. In this sense, our results will be analogous to the usual isotropic case: We will be able to derive $\mathcal{R}$-sectoriality of the Neumann Laplacian in $L_{p}(\R_+,;\mathscr{A}^s)$ and $H_{p}^1(\R_+,;\mathscr{A}^s)$, but for the Dirichlet Laplacian we can only derive it in  $L_{p}(\R_+,;\mathscr{A}^s)$.
\end{remark}

\begin{theorem}\label{Thm:RSectorial_Rnplus}
 Recall Assumption \ref{Assump:ELS} and Assumption \ref{Assump:Spaces}. Suppose that $E$ satisfies Pisier's property $(\alpha)$. Let $k\in[0,k_{\max}]\cap \N_0$, $p\in(1,\infty)$, $r\in(-1,p-1)$ and $s\in\R$. We define the operator
 \[
  A_B\colon H_{p}^k(\R_+,|\operatorname{pr}_n|^r;\mathscr{A}^s)\supset D(A_B) \to H_{p}^k(\R_+,|\operatorname{pr}_n|^r;\mathscr{A}^s),\,u\mapsto A(D)u
 \]
on the domain
\begin{align*}
 D(A_B):=\{ u\in H_{p}^{k}(\R_+,|\operatorname{pr}_n|^r;\mathscr{A}^s): &A_B u\in H_{p}^{k}(\R_+,|\operatorname{pr}_n|^r;\mathscr{A}^s)\\&\operatorname{tr}_{x_n=0} B_j(D)u=0\text{ for all }j=1,\ldots,m\}
\end{align*}
Then for all $\sigma>0$ we have that $A_B-\sigma$ is $\mathcal{R}$-sectorial in $\Sigma_{\phi}$. Moreover, there is a constant $C$ such that for all $\lambda\in\Sigma_\phi$ with $|\lambda|\geq\sigma$ we have the estimate
\begin{align}\label{Eq:InclusionDomain}
  \|u\|_{D^{k,2m,s}_{r,B}(\R_+)} \leq C \| (\lambda-A(D))u\|_{H_{p}^k(\R_+,|\operatorname{pr}_n|^r;\mathscr{A}^s)}.
\end{align}
In particular, it holds that $D^{k,2m,s}_{r,B}(\R_+)=D(A_B)$.
\end{theorem}
\begin{proof}
 Consider
 \begin{align}\label{Eq:Resolvent_Solution_Formula2}
  R(\lambda)f=r_+(\lambda-A(D))_{\R^n}^{-1} \mathscr{E} f - \sum_{j=1}^m\operatorname{pr}_{1}\operatorname{Poi}_j(\lambda)\operatorname{tr}_{x_n=0}B_j(D)(\lambda-A(D))_{\R^n}^{-1} \mathscr{E} f,
 \end{align}
where $\lambda\in\Sigma_{\phi}$, $f\in H_{p}^k(\R_+,|\operatorname{pr}_n|^r;\mathscr{A}^s)$, $(\lambda-A(D))_{\R^n}^{-1}$ denotes the resolvent on $\R^n$ as in Theorem \ref{Thm:Rboundedness_Rn}, $r_+$ denotes the restriction of a distribution on $\R^n$ to $\R^n_+$ and $\mathscr{E}$ denotes an extension operator mapping $H_{p}^t(\R_+,|\operatorname{pr}_n|^r;\mathscr{A}^s)$ into $H_{p}^t(\R,|\operatorname{pr}_n|^r;\mathscr{A}^s)$ for arbitrary $t\in\R$. $\mathscr{E}$ can for example be chosen to be Seeley's extension, see \cite{Seeley_1964}.\\
Combining Proposition \ref{Prop:PoissonAfterTrace} with $\theta=1$ and Theorem \ref{Thm:Rboundedness_Rn} yields that the set
\begin{align}\label{Eq:RSec_Estimate}
 \{\lambda R(\lambda):\lambda\in\Sigma_{\phi},|\lambda|\geq\sigma\}\subset \mathcal{B}(H_{p}^k(\R_+,|\operatorname{pr}_n|^r;\mathscr{A}^s))
\end{align}
is $\mathcal{R}$-bounded. Next, we show that $R(\lambda)$ is indeed the resolvent so that we obtain $\mathcal{R}$-sectoriality. To this end we show that
\[
 R(\lambda)\colon H_{p}^k(\R_+,|\operatorname{pr}_n|^r;\mathscr{A}^s) \to D(A_B)
\]
is a bijection with inverse $\lambda-A_B$. Let $f\in H_{p}^k(\R_+,|\operatorname{pr}_n|^r;\mathscr{A}^s)$. Since $$\operatorname{tr}_{x_n=0}B_k(D)\operatorname{pr}_1\operatorname{Poi}_j(\lambda)=\delta_{k,j}\operatorname{id}_{\mathscr{A}}$$ by construction, it follows from applying $B_j(D)$ to \eqref{Eq:Resolvent_Solution_Formula2} that $\operatorname{tr}_{x_n=0}B_j(D)R(\lambda)f=0$ for all $j=1,\ldots,m$. Moreover, we have $(\lambda-A(D))\operatorname{pr}_1\operatorname{Poi}_j(\lambda)=0$ by the definition of $\operatorname{Poi}_j(\lambda)$. This shows that 
\begin{align}\label{Eq:Left_Inverse}
 (\lambda-A(D))R(\lambda)=\operatorname{id}_{H_{p}^k(\R_+,|\operatorname{pr}_n|^r;\mathscr{A}^s)}
\end{align}
and therefore
\[
 A(D)R(\lambda)f = \lambda R(\lambda)f - (\lambda-A(D))R(\lambda)f=\lambda R(\lambda)f-f.
\]
But it is already contained in \eqref{Eq:RSec_Estimate} that $ \lambda R(\lambda)f\in H_{p}^k(\R_+,|\operatorname{pr}_n|^r;\mathscr{A}^s)$. This shows that $R(\lambda)$ maps $H_{p}^k(\R_+,|\operatorname{pr}_n|^r;\mathscr{A}^s)$ into $D(A_B)$. In addition, \eqref{Eq:Left_Inverse} shows the injectivity of $R(\lambda)$. But also 
\[
 (\lambda-A(D))\colon D(A_B)\to H_{p}^k(\R_+,|\operatorname{pr}_n|^r;\mathscr{A}^s)
\]
is injective as a consequence of the Lopatinskii-Shapiro condition. Hence, there is a mapping
\[
 T(\lambda)\colon H_{p}^k(\R_+,|\operatorname{pr}_n|^r;\mathscr{A}^s)\to D(A_B)
\]
such that $T(\lambda)(\lambda-A(D))=\id_{D(A_B)}$. But from this we obtain
\[
 T(\lambda)=T(\lambda)(\lambda-A_B)R(\lambda)=R(\lambda)
\]
so that
\[
 R(\lambda)(\lambda-A_B)=\id_{D(A_B)},\quad(\lambda-A_B)R(\lambda)=\operatorname{id}_{H_{p}^k(\R_+,|\operatorname{pr}_n|^r;\mathscr{A}^s)},
\]
i.e. $R(\lambda)=(\lambda-A_B)^{-1}$ is indeed the resolvent and we obtain the $\mathcal{R}$-sectoriality.\\
It remains to show that the estimate \eqref{Eq:InclusionDomain} holds. To this end, we can again use the formula for the resolvent \eqref{Eq:Resolvent_Solution_Formula2} in connection with Proposition \ref{Prop:PoissonAfterTrace} ($\theta=0$) and Theorem \ref{Thm:Rboundedness_Rn}. Then we obtain for $u\in D(A_B)$ that
\begin{align*}
 \|u\|_{D^{k,2m,s}_{r,B}(\R_+)}&\leq \|r_+ (\lambda-A(D))_{\R^n}^{-1}\mathscr{E}(\lambda-A_B)u\|_{D^{k,2m,s}_{r,B}(\R_+)}
 \\&\qquad+\sum_{j=1}^m\|\operatorname{pr}_{1}\operatorname{Poi}_j(\lambda)\operatorname{tr}_{x_n=0}B_j(D)(\lambda-A(D))_{\R^n}^{-1} \mathscr{E} (\lambda-A_B)u\|_{D^{k,2m,s}_{r,B}(\R_+)}\\
 &\lesssim\|(\lambda-A_B)u\|_{H^k(\R_+,|\operatorname{pr}_n|^r\mathscr{A}^s)}+\sum_{j=1}^m\|r_+ (\lambda-A(D))_{\R^n}^{-1} \mathscr{E} (\lambda-A_B)u\|_{D^{k,2m,s}_{r,B}(\R_+)}\\
 &\lesssim\|(\lambda-A_B)u\|_{H^k(\R_+,|\operatorname{pr}_n|^r\mathscr{A}^s)}
\end{align*}
This also implies that $D(A_B)=D^{k,2m,s}_{r,B}(\R_+)$. Indeed, it follows from Proposition \ref{Prop:Pseudo_Iterated_Mapping_Properties} that
\begin{align*}
 \|(\lambda-A_B)u\|_{H^k(\R_+,|\operatorname{pr}_n|^r\mathscr{A}^s)}\lesssim \|(\lambda-A(D))\mathscr{E}u\|_{H^k(\R,|\operatorname{pr}_n|^r\mathscr{A}^s)}
 \lesssim \|\mathscr{E}u\|_{D^{k,2m,s}_{r,B}(\R)}\lesssim\|u\|_{D^{k,2m,s}_{r,B}(\R_+)}
\end{align*}
for $u\in D^{k,2m,s}_{r,B}(\R_+)$. Hence, we have
\[
	D^{k,2m,s}_{r,B}(\R_+)\hookrightarrow D(A_B) \hookrightarrow D^{k,2m,s}_{r,B}(\R_+).
\]
\end{proof}
\begin{remark}
	Since $E$ is a UMD space, the results of Theorem~\ref{Thm:RSectorial_Rnplus} also hold for $k\in[0,k_{\max}]$, i.e. $k$ does not have to be an integer. This follows from complex interpolation, see Proposition~\ref{Prop:Rbounded_Interpolation} and \cite[Proposition 5.6]{Lindemulder_Meyries_Veraar_2018}. Note that unlike in Proposition~\ref{Prop:Rbounded_Interpolation} we can not replace the UMD space $E$ by a K-convex Banach space here, since the UMD property is needed for the complex interpolation of Bessel potential spaces in \cite[Proposition 5.6]{Lindemulder_Meyries_Veraar_2018}. Moreover, in Assumption~\ref{Assump:Spaces} we require $E$ to be a UMD space if one of the spaces in tangential or normal direction belongs to the Bessel potential scale.
\end{remark}

Two canonical applications of Theorem \ref{Thm:RSectorial_Rnplus} are Dirichlet and Neumann Laplacian.
\begin{corollary} Let $E=\C$, $p\in(1,\infty)$, $r\in(1,p-1)$ and $s\in\R$.
 \begin{enumerate}[(a)]
  \item We consider the Laplacian with Dirichlet boundary conditions
  \[
   \Delta_D\colon L_p(\R_+,|\operatorname{pr}_n|^{r};\mathscr{A}^s)\supset D(\Delta_D)\to  L_p(\R_+,|\operatorname{pr}_n|^{r};\mathscr{A}^s)
  \]
on the domain $D(\Delta_D)$ given by
\[
 D(\Delta_D):=\{u\in H^{2}_p(\R_+,|\operatorname{pr}_n|^{r};\mathscr{A}^s)\cap L_p(\R_+,|\operatorname{pr}_n|^{r};\mathscr{A}^{s+2m}): \operatorname{tr}_{x_n=0}u=0\}.
\]
For all $\sigma>0$ it holds that $\Delta_D-\sigma$ is $\mathcal{R}$-sectorial in any sector $\Sigma_{\psi}$ with $\psi\in(0,\pi)$.
\item Let $k\in\{0,1\}$. We consider the Laplacian with Neumann boundary conditions
 \[
   \Delta_N\colon H^k_p(\R_+,|\operatorname{pr}_n|^{r};\mathscr{A}^s)\supset D(\Delta_D)\to  H^k_p(\R_+,|\operatorname{pr}_n|^{r};\mathscr{A}^s)
  \]
on the domain $D(\Delta_N)$ given by
\[
 D(\Delta_N):=\{u\in H^{k+2}_p(\R_+,|\operatorname{pr}_n|^{r};\mathscr{A}^s)\cap H^k_p(\R_+,|\operatorname{pr}_n|^{r};\mathscr{A}^{s+2m}): \operatorname{tr}_{x_n=0}\partial_n u=0\}.
\]
For all $\sigma>0$ it holds that $\Delta_N-\sigma$ is $\mathcal{R}$-sectorial in any sector $\Sigma_{\psi}$ with $\psi\in(0,\pi)$.
 \end{enumerate}
\end{corollary}
\begin{proof}
 Both statements follow directly from Theorem \ref{Thm:RSectorial_Rnplus}.
\end{proof}


\section{Application to Boundary Value Problems}\label{Section:BVP}
 
\begin{theorem}\label{Thm:EllipticBVP1}
 Let $s_1,\ldots,s_m\in\R$ and $g_j\in\mathscr{A}^{s_j}$ $(j=1,\ldots,m)$. Then the equation
 \begin{align*}
  \lambda u-A(D)u&=0\quad\;\text{in }\R^n_+,\\
  B_j(D)u&=g_j\quad\text{on }\R^{n-1}
 \end{align*}
 has a unique solution $u\in\mathscr{S}'(\R^n_+;E)$ for all $\lambda\in\Sigma_{\phi}$. This solution satisfies 
 \[
  u\in\sum_{j=1}^m\bigcap_{r,t\in\R,\,k\in\N_0,\,p\in[1,\infty)\atop r-p[t+k-m_j-s_j]_+>-1} W_p^k(\R_+,|\operatorname{pr}_n|^r;\mathscr{A}^t).
 \]
Moreover, for all $\sigma>0$, $t,r\in\R$, $p\in[1,\infty)$ and $k\in\N_0$ such that $r-p[t+k-m_j-s_j]_+>-1$ for all $j=1,\ldots,m$ there is a constant $C>0$ such that
\[
 \|u\|_{W_p^k(\R_+,|\operatorname{pr}_n|^r;\mathscr{A}^t)}\leq C\sum_{j=1}^m |\lambda|^{\frac{-1-r+p(k-m_j)+p[t-s_j]_+}{2mp}}\|g_j\|_{\mathscr{A}^{s_j}}
\]
for all $\lambda\in\Sigma_{\phi}$ with $|\lambda|\geq\sigma$
\end{theorem}
\begin{proof}
 All the assertions follow directly from Theorem \ref{Thm:MainThm1} (\ref{Thm:MainThm1:UniformBounds}).
\end{proof}

\begin{remark} \phantomsection \label{Remark:EllipticBVP1}
 \begin{enumerate}[(a)]
  \item Note that the smoothness parameters $k$ and $t$ of the solution in Theorem \ref{Thm:EllipticBVP1} can be chosen arbitrarily large if one accepts a strong singularity at the boundary. On the other hand, if $t$ is chosen small enough, then the singularity can be removed.
  \item In Theorem \ref{Thm:EllipticBVP1} we can take $k=1+\max_{j=1,\ldots,m}m_j$, $r=0$ and $t$ such that $r-p[t+k-m_j-s_j]_+>-1$ for all $j=1,\ldots,m$. This means that the boundary conditions $B_j(D)u=g_j$ can be understood in a classical sense. Indeed, \cite[Proposition 7.4]{Meyries_Veraar_2012} in connection with \cite[Proposition 3.12]{Meyries_Veraar_2012} shows that 
  \[
  	W^k_p(\R_+;\mathscr{A}^t)\hookrightarrow BUC^{k-1}(\R_+;\mathscr{A}^t).
  \]
  Hence, $\operatorname{tr}_{x_n=0}B_j(D)u$ can be defined in the classical sense.
  \item One can again use interpolation techniques or one can directly work with Corollary \ref{Cor:MainThm1_Corollary} in order to obtain results for the Bessel potential or the Besov scale in normal direction. Note however that this comes with some restrictions on the weight $|\operatorname{pr}_n|^r$.
 \end{enumerate}
\end{remark}

\begin{theorem}\label{Thm:EllipticBVP2}
As defined in Remark~\ref{Rem:k_max_explanation} we set
	 	\[
 		k_{\max}:=\min\{\beta_n|\,\exists j\in\{1,\ldots,m\}\exists\beta\in\N_0^n,|\beta|= m_j: b^j_{\beta}\neq 0\}.
 	\]
 	Let $s\in\R$, $p\in(1,\infty)$, $r\in(-1,p-1)$, $k\in[0,k_{\max}]\cap\N_0$ and $f\in W^k_p(\R_+,|\operatorname{pr}_n|^r;\mathscr{A}^s)$. Let further $s_j\in(s+2m+k-m_j-\frac{1+r}{p},\infty)$ and $g_j\in \mathscr{A}^{s_j}$ $(j=1,\ldots,m)$. Then the equation
 \begin{align*}
  \lambda u-A(D)u&=f\quad\;\text{in }\R^n_+,\\
  B_j(D)u&=g_j\quad\text{on }\R^{n-1}
 \end{align*}
 has a unique solution
 \[
  u\in W^{k+2m}_p(\R_+,|\operatorname{pr}_n|^r;\mathscr{A}^s) \cap W^{k}_p(\R_+,|\operatorname{pr}_n|^r;\mathscr{A}^{s+2m})
 \]
 and for all $\sigma>0$ there is a constant $C>0$ such that for all $\lambda\in\Sigma_{\phi}$ with $|\lambda|\geq\sigma$ we have the estimate
 \begin{align*}
  \|u\|_{W^{k+2m}_p(\R_+,|\operatorname{pr}_n|^r;\mathscr{A}^s)}+&\|u\|_{W^{k}_p(\R_+,|\operatorname{pr}_n|^r;\mathscr{A}^{s+2m})}+|\lambda|\,\|u\|_{W^{k}_p(\R_+,|\operatorname{pr}_n|^r;\mathscr{A}^s)}\\
  &\leq C\left(\|f\|_{W^k_p(\R_+,|\operatorname{pr}_n|^r;\mathscr{A}^s)}+\sum_{j=1}^m|\lambda|^{\frac{-1-r+p(k+2m-m_j)}{2mp}}\|g\|_{\mathscr{A}^{s_j}}\right).
 \end{align*}
\end{theorem}
\begin{proof}
 By Theorem \ref{Thm:RSectorial_Rnplus} we have a unique solution $$u_1\in W^{k+2m}_p(\R_+,|\operatorname{pr}_n|^r;\mathscr{A}^s) \cap W^{k}_p(\R_+,|\operatorname{pr}_n|^r;\mathscr{A}^{s+2m})$$ to the equation 
  \begin{align*}
  \lambda u_1-A(D)u_1&=f\quad\text{in }\R^n_+,\\
  B_j(D)u_1&=0\quad\text{on }\R^{n-1}
 \end{align*}
 which satisfies the estimate
  \begin{align*}
  \|u_1\|_{W^{k+2m}_p(\R_+,|\operatorname{pr}_n|^r;\mathscr{A}^s)}+\|u_1\|_{W^{k}_p(\R_+,|\operatorname{pr}_n|^r;\mathscr{A}^{s+2m})}+|\lambda|\,\|u_1\|_{W^{k}_p(\R_+,|\operatorname{pr}_n|^r;\mathscr{A}^s)}\leq C\|f\|_{W^k_p(\R_+,|\operatorname{pr}_n|^r;\mathscr{A}^s)}.
 \end{align*}
 By Remark \ref{Rem:Rboundedness_Rn} \eqref{Rem:Rboundedness_Rn:Sec}, we do not need Pisier's property $(\alpha)$ for this. Moreover, by Theorem \ref{Thm:EllipticBVP1} the unique solution $u_2$ to the equation
  \begin{align*}
  \lambda u_2-A(D)u_2&=0\quad\;\text{in }\R^n_+,\\
  B_j(D)u_2&=g_j\quad\text{on }\R^{n-1}
 \end{align*}
 satisfies the estimates
 {\allowdisplaybreaks
 \begin{align*}
  \|u_2\|_{W^{k+2m}_p(\R_+,|\operatorname{pr}_n|^r;\mathscr{A}^s)}&\leq C \sum_{j=1}^m|\lambda|^{\frac{-1-r+p(k+2m-m_j)}{2mp}}\|g\|_{\mathscr{A}^{s_j}},\\
  \|u_2\|_{W^{k}_p(\R_+,|\operatorname{pr}_n|^r;\mathscr{A}^{s+2m})}&\leq C \sum_{j=1}^m|\lambda|^{\frac{-1-r+p(k-m_j)}{2mp}}\|g\|_{\mathscr{A}^{s_j}},\\
  \|u_2\|_{W^{k}_p(\R_+,|\operatorname{pr}_n|^r;\mathscr{A}^{s})}&\leq C \sum_{j=1}^m|\lambda|^{\frac{-1-r+p(k-m_j)}{2mp}}\|g\|_{\mathscr{A}^{s_j}}.
 \end{align*}}
 Note that by our choice of $s_j$, we have
 \[
  r-[s+2m+k-m_j-s_j]_+>r-(s+2m+k-m_j-s-2m-k+m_j+\tfrac{1+r}{p})=-1
 \]
 for $s+2m+k-m_j-\frac{1+r}{p}<s_j\leq s+2m+k-m_j$ and
  \[
  r-[s+2m+k-m_j-s_j]_+=r>-1
 \]
 for $s_j\geq s+2m+k-m_j$. The unique solution $u$ of the full system is given by $u=u_1+u_2$ and therefore summing up yields the assertion.
\end{proof}

\begin{theorem}\label{Thm:PBVP}
Recall from Assumption~\ref{Assump:Spaces} that $\mathscr{C}$ stands for the Bessel potential, Besov, Triebel-Lizorkin or one of their dual scales and that we impose some conditions on the corresponding parameters. Let $\sigma>0$, $s_1,\ldots,s_m,l_1,\ldots,l_m\in\R$ and $g_j\in\mathscr{C}^{l_j}(\R_+,w_2;\mathscr{A}^{s_j})$. Let further
 \begin{align*}
    P_j=\{(r,t_0,l,k,p): t_0,l\in\R,&r\in(-1,\infty).k\in\N_0,p\in [1,\infty),\\
    &r-p[t_0+k-m_j-s_j]_+>-1,\\
    &r-2mp(l-l_j)-p(k-m_j)-p[t_0-s_j]_+>-1\}
 \end{align*}
the set of admissible parameters. Then the equation
 \begin{align}
 \begin{aligned}\label{Eq:PBVP}
  \partial_t u +\sigma u- A(D) u &= 0\quad\;\text{in }\R\times\R^n_+,\\
  B_j(D)u&=g_j\quad\text{on }\R_+\times\R^{n-1},
  \end{aligned}
 \end{align}
has a unique solution $u\in\mathscr{S}'(\R\times\R^n_+;E)$. This solution satisfies
\[
 u\in\sum_{j=1}^m\bigcap_{(r,t_0,l,k,p)\in P_j}\mathscr{C}^{l}(\R,w_2;W^{k}_p(\R_+,|\operatorname{pr}_n|^r;\mathscr{A}^{t_0}))
\]
and for all $(r,t_0,l,k,p)\in \bigcap_{j=1}^mP_j$ there is a constant $C>0$ independent of $g_1,\ldots,g_m$ such that
\[
 \|u\|_{\mathscr{C}^{l}(\R,w_2;W^{k}_p(\R_+,|\operatorname{pr}_n|^r;\mathscr{A}^{t_0}))}\leq C\sum_{j=1}^m\|g_j\|_{\mathscr{C}^{l_j}(\R,w_2;\mathscr{A}^{s_j})}.
\]
\end{theorem}
\begin{proof}
 We apply the Fourier transform $\mathscr{F}_{t\mapsto\tau}$ in time to \eqref{Eq:PBVP} and obtain 
  \begin{align}
 \begin{aligned}
  (\sigma+i\tau) \hat{u}- A(D) \hat{u} &= 0\quad\;\text{in }\R\times\R^n_+,\\
  B_j(D)\hat{u}&=\hat{g}_j\quad\text{on }\R_+\times\R^{n-1}.
  \end{aligned}
 \end{align}
 Hence, the solution of \eqref{Eq:PBVP} is given by
 \[
  u(t,x)=\sum_{j=1}^m\mathscr{F}_{t\to\tau}^{-1}\operatorname{Poi}_j(\sigma+i\tau)\mathscr{F}_{t\to\tau}g_j.
 \]
From Theorem \ref{Thm:MainThm1} together with Lemma \ref{Lemma:HolomorphicSectorRBounded} it follows that
\[
 [\tau\mapsto \operatorname{Poi}_j(\sigma+i\tau)]\in S^{\frac{-1-r+p(k-m_j)+p[{t_0}-s_j]_+}{2mp}+\epsilon}_{\mathcal{R}}(\R,\mathcal{B}(\mathscr{A}^{s_j},W^k_p(\R_+,|\operatorname{pr}_n|^r;\mathscr{A}^{t_0})))
\]
for arbitrary $\epsilon>0$ if the parameters satisfy $r-p[{t_0}+k-m_j-s_j]_+>-1$. Hence, the parameter-independent version of Proposition~\ref{Prop:PseudoMappingProperties_R_bounded} (as in Remark \ref{Rem:SymbolClasses} \eqref{Rem:SymbolClasses:DependentIndependent}) yields
\[
 \mathscr{F}_{t\to\tau}^{-1}\operatorname{Poi}_j(\sigma+i\tau)\mathscr{F}_{t\to\tau}g_j\in \mathscr{C}^{l_j-\epsilon+\frac{1+r-p(k-m_j)-p[{t_0}-s_j]_+}{2mp}}(\R,w_2;W^{k}_p(\R_+,|\operatorname{pr}_n|^r;\mathscr{A}^{t_0}))
\]
as well as the estimate
\[
 \|\mathscr{F}_{t\to\tau}^{-1}\operatorname{Poi}_j(\sigma+i\tau)\mathscr{F}_{t\to\tau}g_j\|_{\mathscr{C}^{l_j-\epsilon+\frac{1+r-p(k-m_j)-p[{t_0}-s_j]_+}{2mp}}(\R,w_2;W^{k}_p(\R_+,|\operatorname{pr}_n|^r;\mathscr{A}^{t_0})}\leq C\|g_j\|_{\mathscr{C}^{l_j}(\R,w_2;\mathscr{A}^{s_j})}.
\]
But the condition
\begin{align}\label{Eq:AParameterEstimate1}
 r-2mp(l-l_j)-p(k-m_j)-p[{t_0}-s_j]_+>-1
\end{align}
implies
\begin{align}\label{Eq:AParameterEstimate2}
 l\leq l_j-\epsilon+\frac{1+r-p(k-m_j)-p[{t_0}-s_j]_+}{2mp}.
\end{align}
if $\epsilon>0$ is chosen small enough. Therefore, we obtain
\[
 \mathscr{F}_{t\to\tau}^{-1}\operatorname{Poi}_j(\sigma+i\tau)\mathscr{F}_{t\to\tau}g_j\in \mathscr{C}^{l}(\R,w_2;W^{k}_p(\R_+,|\operatorname{pr}_n|^r;\mathscr{A}^{t_0}))
\]
and the estimate
\[
 \|\mathscr{F}_{t\to\tau}^{-1}\operatorname{Poi}_j(\sigma+i\tau)\mathscr{F}_{t\to\tau}g_j\|_{\mathscr{C}^{l}(\R,w_2;W^{k}_p(\R_+,|\operatorname{pr}_n|^r;\mathscr{A}^{t_0})}\leq C\|g_j\|_{\mathscr{C}^{l_j}(\R,w_2;\mathscr{A}^{s_j})}.
\]
if $(r,{t_0},l,k,p)\in P_j$. Taking the sum over all $j=1,\ldots,m$ yields the assertion.
\end{proof}

\begin{remark}
 \begin{enumerate}[(a)]
  \item If $\mathscr{C}$ does not stand for the Bessel potential scale or if $p>\max\{p_0,q_0,q_E\}$ where $q_E$ denotes the cotype of $E$, then the parameter set $P_j$ in Theorem \ref{Thm:PBVP} can potentially be chosen slightly larger, namely 
   \begin{align*}
    P_j=\{(r,{t_0},l,k,p): {t_0},l\in\R,&r\in(-1,\infty).k\in\N_0,p\in [1,\infty),\\
    &r-p[{t_0}+k-m_j-s_j]_+>-1,\\
    &r-2mp(l-l_j)-p(k-m_j)-p[{t_0}-s_j]_+\geq-1\}.
 \end{align*}
 Indeed, if $p>\max\{p_0,q_0,q_E\}$, then
\[
 [\tau\mapsto \operatorname{Poi}_j(\sigma+i\tau)]\in S^{\frac{-1-r+p(k-m_j)+p[{t_0}-s_j]_+}{2mp}}_{\mathcal{R}}(\R,\mathcal{B}(\mathscr{A}^{s_j},W^k_p(\R_+,|\operatorname{pr}_n|^r;\mathscr{A}^{t_0})))
\]
by Theorem \ref{Thm:MainThm1}. If one continues the proof of Theorem \ref{Thm:PBVP} with this information, then one will find that the $\epsilon$ in \eqref{Eq:AParameterEstimate2} can be removed so that the inequality \eqref{Eq:AParameterEstimate1} does not have to be strict. The same holds for Besov and Triebe-Lizorkin scale, as in this case
\[
 [\tau\mapsto \operatorname{Poi}_j(\sigma+i\tau)]\in S^{\frac{-1-r+p(k-m_j)+p[{t_0}-s_j]_+}{2mp}}(\R,\mathcal{B}(\mathscr{A}^{s_j},W^k_p(\R_+,|\operatorname{pr}_n|^r;\mathscr{A}^{t_0})))
\]
is good enough and holds without restriction on $p$.
\item As in Remark \ref{Remark:EllipticBVP1} we can take the trace $\operatorname{tr}_{x_n=0}B_j(D)u$ in the classical sense if $k$ is large enough and if $l$ and ${t_0}$ are small enough.
\item Again, we can use interpolation techniques to extend the result in Theorem \ref{Thm:PBVP} to the case in which the Bessel potential or Besov scale are taken in normal direction. However, this can only be done for $r\in(-1,p-1)$.
 \end{enumerate}
\end{remark}

\begin{theorem}\label{Thm:PIBVP1}
 Let $\alpha\in(0,1)$, $T>0$, $s,t_0\in\R$, $p\in(1,\infty)$, $r\in(-1,p-1)$, $\mu\in(-1,\infty)$, $v_{\mu}(t)=t^{\mu}$ $(t\in(0,T])$ and $s_1,\ldots,s_m,l_1,\ldots,l_m\in\R$.  Assume that $\mu\in (-1,q_2)$ if $\mathscr{C}$ belongs to the Bessel potential scale. Let again $$k_{\max}:=\min\{\beta_n|\,\exists j\in\{1,\ldots,m\}\exists\beta\in\N_0^n,|\beta|= m_j: b^j_{\beta}\neq 0\}.$$ and $k\in[0,k_{\max}]\cap\N_0$.  We further assume that
 \begin{align}\label{Eq:RestrictionTimeRegularity}
  l_j>\frac{1+\mu}{q_2}-\frac{1+r}{2mp}+\frac{k-m_j+[t_0-s_j]_+}{2m}\quad\text{and}\quad s_j>t_0+k-m_j-\frac{1+r}{p}
 \end{align}
for all $j=1,\ldots,m$. Suppose that $E$ satisfies Pisier's property $(\alpha)$.\\
Then for all $u_0\in H_p^k(\R_+,|\operatorname{pr}_n|^{r};\mathscr{A}^{t_0})$, all $\alpha$-H\"older continuous $f\in C^{\alpha}((0,T);H_p^k(\R_+,|\operatorname{pr}_n|^{r};\mathscr{A}^{t_0}))$ with $\alpha(0,1)$ and $g_j\in\mathscr{C}^{l_j}([0,T],v_{\mu};\mathscr{A}^{s_j})$ there is a unique solution $u$ of the equation 
\begin{align}
    \begin{aligned}\label{Eq.PIBVP1}
    \partial_t u- A(D) u&=f\quad\;\text{in }(0,T]\times\R^{n}_+,\\
    B_j(D)u&=g_j\quad\text{on }(0,T]\times\R^{n-1},\\
    u(0,\,\cdot\,)&=u_0
    \end{aligned}
\end{align}
which satisfies
\begin{align*}
 u&\in C([0,T];H^{k}_{p}(\R_{+,x_n},|\operatorname{pr}_n|^{r};\mathscr{A}^{t_0})),\\
 u&\in \mathscr{C}^{l^*}((0,T],v_{\mu};H^{k+2m}_p(\R_{+,x_n},|\operatorname{pr}_n|^{r};\mathscr{A}^{t_0-2m})),\\
 u&\in C^1((0,T];H^{k}_{p}([\delta,\infty)_{x_n},|\operatorname{pr}_n|^{r};\mathscr{A}^{t_0})),\\
 u&\in C((0,T];H^{k+2m}_p([\delta,\infty)_{x_n},|\operatorname{pr}_n|^{r};\mathscr{A}^{t_0})\cap H^k_p([\delta,\infty)_{x_n},|\operatorname{pr}_n|^{r};\mathscr{A}^{t_0+2m}))
\end{align*}
for all $\delta>0$ and some $l^*\in\R$.
\end{theorem}
\begin{proof}
 First, we substitute $v(t,\,\cdot\,)=e^{-\sigma t}u(t,\,\cdot\,)$ for some $\sigma>0$. Since we work on a bounded time interval $[0,T]$, this multiplication is an automorphism of all the spaces we consider in this theorem. Hence, it suffices to look for a solution of the equation
 \begin{align*}
    \partial_t v+\sigma v- A(D) v&=\tilde{f}\quad\;\text{in }[0,T]\times\R^{n}_+,\\
    B_j(D)v&=\tilde{g}_j\quad\text{on }[0,T]\times\R^{n-1},\\
    v(0,\,\cdot\,)&=u_0,
\end{align*}
where $\tilde{f}(t)=e^{-\sigma t} f(t)$ and $\tilde{g}_j(t)=e^{-\sigma t} g_j(t)$. We split $v$ into two parts $v=r_{[0,T]}v_1+v_2$ which are defined as follows: $v_1$ solves the equation
 \begin{align*}
    \partial_t v_1+\sigma v_1- A(D) v_1&=0\quad\;\text{in }\R\times\R^{n}_+,\\
    B_j(D)v_1&=\mathscr{E}\tilde{g}_j\quad\text{on }\R\times\R^{n-1},
\end{align*}
where $\mathscr{E}$ is a suitable extension operator and $r_{[0,T]}$ is the restriction to $[0,T]$. Moreover, $v_2$ is the solution of
\begin{align}
\begin{aligned}\label{Eq:InitialPart}
    \partial_t v_2+\sigma v_2- A(D) v_2&=\tilde{f}\quad\;\text{in }[0,T]\times\R^{n}_+,\\
    B_j(D)v_2&=0\quad\text{on }[0,T]\times\R^{n-1},\\
    v_2(0,\,\cdot\,)&=v_0-v_1(0,\,\cdot\,).
    \end{aligned}
\end{align}

For $v_1$ it follows from Theorem \ref{Thm:PBVP} that
\[
 v_1\in\sum_{j=1}^m\bigcap_{(r',t',l',k',p')\in P_j}\mathscr{C}^{l'}(\R,v_{\mu};W^{k'}_{p'}(\R_+,|\operatorname{pr}_n|^{r'};\mathscr{A}^{t'}))
\]
and for all $(r',t',l',k',p')\in \bigcap_{j=1}^mP_j$ there is a constant $C>0$ independent of $g_1,\ldots,g_m$ such that
\[
 \|v_1\|_{\mathscr{C}^{l'}(\R,v_{\mu};W^{k'}_{p'}(\R_+,|\operatorname{pr}_n|^{r'};\mathscr{A}^{t'}))}\leq C\sum_{j=1}^m\|\mathscr{E}\tilde{g}_j\|_{\mathscr{C}^{l_j}(\R,v_{\mu};\mathscr{A}^{s_j})}.
\]
 In particular, if $l'>\frac{1+\mu}{q_2}$ we have $v_1\in BUC([0,T];W^{k'}_{p'}(\R_+,|\operatorname{pr}_n|^{r'};\mathscr{A}^{t'}))$ so that we can take the time trace $v_1(0)$, see \cite[Proposition 7.4]{Meyries_Veraar_2012}. If condition \eqref{Eq:RestrictionTimeRegularity} is satisfied, then we can choose  $l>\frac{1+\mu}{q_2}$ small enough such that $(r,t_0,l,k,p)\in \bigcap_{j=1}^mP_j$. Hence, under this condition we obtain
 \[
  v_1(0)\in W^{k}_p(\R_+,|\operatorname{pr}_n|^r;\mathscr{A}^{t_0})
 \]
is well-defined. 
But since we have $r\in(-1,p-1)$, it follows from Theorem \ref{Thm:RSectorial_Rnplus} that $A_B-\sigma$ generates a holomorphic $C_0$-semigroup $(T(t))_{t\geq0}$ in $W^{k}_p(\R_+,|\operatorname{pr}_n|^r;\mathscr{A}^{t_0})$. In addition, $v_2$ is given by
\[
	v_2(t)=T(t)[v_0-v_1(0,\,\cdot\,)]+\int_{0}^t T(t-s) f(s)\,ds.
\]
 It follows from standard semigroup theory that $$v_2\in C((0,T];D(A_B))\cap C^1((0,T];X)\cap C([0,T];X),$$ 
 where
 \[
 X=H^{k}_{p}(\R_{+,x_n},|\operatorname{pr}_n|^{r};\mathscr{A}^{t_0}),\quad D(A_B)=D^{k,2m,s}_{r,B}(\R_+)\hookrightarrow H^{k+2m}_p(\R_{+,x_n},|\operatorname{pr}_n|^{r};\mathscr{A}^{t_0})\cap H^k_p(\R_{+,x_n},|\operatorname{pr}_n|^{r};\mathscr{A}^{{t_0}+2m}),
 \]
 see for example \cite[Chapter 4, Corollary 3.3]{Pazy_1983}. Since also $v_1\in BUC([0,T];W^{k}_{p}(\R_+,|\operatorname{pr}_n|^{r};\mathscr{A}^{t_0}))$, we obtain 
 \[
 u\in C([0,T];H^{k}_{p}(\R_{+,x_n},|\operatorname{pr}_n|^{r};\mathscr{A}^{t_0})),
 \]
 and, since $v_1$ is arbitrarily smooth away from the boundary, also
 \begin{align*}
  u&\in C^1((0,T];H^{k}_{p}([\delta,\infty)_{x_n},|\operatorname{pr}_n|^{r};\mathscr{A}^{t_0})),\\
 u&\in C((0,T];H^{k+2m}_p([\delta,\infty)_{x_n},|\operatorname{pr}_n|^{r};\mathscr{A}^{t_0})\cap H^k_p([\delta,\infty)_{x_n},|\operatorname{pr}_n|^{r};\mathscr{A}^{t_0+2m}))
 \end{align*}
 for all $\delta>0$. Concerning the value of $l^*$, we note that if $(r,{t_0},l,k,p)\in \bigcap_{j=1}^mP_j$, then also $(r,{t_0}-2m,l-1,k+2m,p)\in \bigcap_{j=1}^mP_j$. Hence, we just have to take $l^*\leq l-1$ such that
 \[
 C((0,T];H^{k+2m}_p(\R_{+,x_n},|\operatorname{pr}_n|^{r};\mathscr{A}^{t_0}))\hookrightarrow \mathscr{C}^{l^*}((0,T],v_{\mu};H^{k+2m}_p(\R_{+,x_n},|\operatorname{pr}_n|^{r};\mathscr{A}^{t_0-2m})).
 \]
 Altogether, this finishes the proof.
\end{proof}

\begin{remark}
 While we can treat arbitrary space regularity of the boundary data in Theorem \ref{Thm:PIBVP1}, it is important to note that \eqref{Eq:RestrictionTimeRegularity} poses a restriction on the time regularity of the boundary data. Even if we take ${t_0}\leq\min_{j=1,\ldots,m} s_j$, $k=0$, $r$ very close to $p-1$ and $q_2$ very large, we still have the restriction
 \[
  l_j>-\frac{1+m_j}{2m}.
 \]
In the case of the heat equation with Dirichlet boundary conditions, this would mean that the boundary data needs to have a time regularity strictly larger than $-\frac{1}{2}$. Having boundary noise in mind, it would be interesting to go beyond this border. It would need further investigation whether this is possible or not. In fact, \eqref{Eq:RestrictionTimeRegularity} gives a restriction on the time regularity only because we do not allow $r\geq p-1$, otherwise we could just take $r$ very large and allow arbitrary regularity in time. The reason why we have to restrict to $r<p-1$ is that we want to apply the semigroup  to the time trace $v_1(0)$. However, until now we can only do this for $r\in(-1,p-1)$. Hence, if one wants to improve Theorem \ref{Thm:PIBVP1} to the case of less time regularity, there are at least two possible directions:
\begin{enumerate}
 \item One could try to generalize Theorem \ref{Thm:RSectorial_Rnplus} to the case in which $r>p-1$. In fact, in \cite{Lindemulder_Veraar_2018} Lindemulder and Veraar derive a bounded $\mathcal{H}^{\infty}$-calculus for the Dirichlet Laplacian in weighted $L_p$-spaces with power weights of order $r\in(-1,2p-1)\setminus\{p-1\}$. It would be interesting to see whether their methods also work for $L_p(\R_+,|\operatorname{pr}|^r;\mathscr{A}^s)$ with $r\in(p-1,2p-1)$.
 \item One could try to determine all initial data $u_0$ which is given by $u_0=\tilde{u}_0+v_1(0)$ where $\tilde{u}_0\in H^k_p(\R_+,|\operatorname{pr}_n|^r;\mathscr{A}^t)$ and $v_1$ is the solution to 
 \begin{align*}
    \partial_t v_1+\sigma v_1- A(D) v_1&=0\quad\;\text{in }\R\times\R^{n}_+,\\
    B_j(D)v_1&=\tilde{g}_j\quad\text{in }\R\times\R^{n}_+,
\end{align*}
for some $\tilde{g}_j\in \mathscr{C}^{l_j}(\R,v_{\mu};\mathscr{A}^{s_j})$ satisfying $\tilde{g}_j\vert_{[0,T]}=g_j$. For such initial data, the initial boundary value problem can be solved with our methods for arbitrary time regularity of the boundary data. Indeed, in this case we just have to take the right extension of $g_j$ so that $u_0-v_1(0)\in H^k_p(\R_+,|\operatorname{pr}|^r;\mathscr{A}^t)$. Then we can just apply the semigroup in order to obtain the solution of \eqref{Eq:InitialPart}.
\end{enumerate}

\end{remark}

\section*{Acknowledgment} Since most of the material in this work is based on some results of my Ph.D. thesis and just contains generalizations, simplifications and corrections of mistakes, I would like to thank my Ph.D. supervisor Robert Denk again for his outstanding supervision.\\
I thank Mark Veraar for the instructive discussion on the necessity of finite cotype in certain estimates, which helped me to prove Proposition \ref{Prop:CounterExampleEpsilon}.\\
I also thank the Studienstiftung des deutschen Volkes for the scholarship during my doctorate and the EU for the partial support within the TiPES project funded by the European Union's Horizon 2020 research and innovation programme under grant agreement No 820970. Moreover, I acknowledge partial support of the SFB/TR109 “Discretization in Geometry and Dynamics”.


%

\end{document}